%% file: main_conser_energy.tex
\documentclass[11pt]{article}

\usepackage{graphicx}
\usepackage{latexsym,amsmath,amsfonts,amscd, amsthm, dsfont}
\usepackage{bm,color}
\usepackage{epsfig,verbatim,epstopdf,graphics}
\usepackage{subfigure}
\usepackage{changebar}
\usepackage{multirow}

\usepackage[ruled,vlined]{algorithm2e}

\usepackage{yhmath}
 \usepackage{booktabs} 
 \usepackage{tikz}
\usepackage{verbatim}
\usetikzlibrary{arrows,backgrounds,snakes,shapes}

 \numberwithin{equation}{section}

\graphicspath{{./}{./figure/}}
\allowdisplaybreaks

\newtheoremstyle{plainNoItalics}{}{}{\normalfont}{}{\bfseries}{.}{ }{}

\theoremstyle{plain}
\newtheorem{thm}{Theorem}[section]

\theoremstyle{plainNoItalics}

\newtheorem{prop}[thm]{Proposition}
\newtheorem{exa}[thm]{Example}

\newcommand{\bx}{{\bf x}}

\newcommand{\bv}{{\bf v}}
\newcommand{\bw}{{\bf w}}
\newcommand{\bB}{{\bf B}}
\newcommand{\bE}{{\bf E}}
\newcommand{\bJ}{{\bf J}}

\newcommand{\bU}{{\bf U}}

\newcommand{\bV}{{\bf V}}

\newcommand{\mT}{{\mathcal T}}

\newcommand{\br}{{\bf r}}

\newcommand{\beq}{\begin{equation}}
\newcommand{\eeq}{\end{equation}}
\newcommand{\bit}{\begin{itemize}}
\newcommand{\eit}{\end{itemize}}
\newcommand{\be}{\begin{eqnarray}}
\newcommand{\ee}{\end{eqnarray}}
\newcommand{\beno}{\begin{eqnarray*}}
\newcommand{\eeno}{\end{eqnarray*}}

\newcommand{\Rmnum}[1]{\expandafter\@slowromancap\romannumeral #1@}
\makeatother

\usepackage{enumerate}

\begin{document}

\baselineskip=1.8pc


\input{title}

\input{intro}

\input{algorithm}

\input{numerical_energy}

\input{conclusion}

\bibliographystyle{abbrv}
\bibliography{refer,ref_guo,ref_cheng,ref_cheng_2}

\end{document}

%% file: title.tex
\begin{center}
{\bf
A Local Macroscopic Conservative (LoMaC) low rank tensor method for the Vlasov dynamics
}
\end{center}

\vspace{.2in}
\centerline{
 Wei Guo\footnote{
Department of Mathematics and Statistics, Texas Tech University, Lubbock, TX, 70409. E-mail:
weimath.guo@ttu.edu. Research is supported by NSF grant NSF-DMS-1830838 and NSF-DMS-2111383, Air Force Office of Scientific Research FA9550-18-1-0257.
} and 
Jing-Mei Qiu\footnote{Department of Mathematical Sciences, University of Delaware, Newark, DE, 19716. E-mail: jingqiu@udel.edu. Research supported by NSF grant NSF-DMS-1818924 and 2111253, Air Force Office of Scientific Research FA9550-18-1-0257.}
}

\bigskip
\noindent
{\bf Abstract.} In this paper, we propose a novel Local Macroscopic Conservative (LoMaC) low rank tensor method for simulating the Vlasov-Poisson (VP) system. The LoMaC property refers to the exact local conservation of macroscopic mass, momentum and energy at the discrete level. This is a follow-up work of our previous development of a conservative low rank tensor approach for Vlasov dynamics (arXiv:2201.10397). In that work, we applied a low rank tensor method with a conservative singular value decomposition (SVD) to the high dimensional VP system to mitigate the curse of dimensionality, while maintaining the local conservation of mass and momentum. However, energy conservation is not guaranteed, which is a critical property to avoid unphysical plasma self-heating or cooling. The new ingredient in the LoMaC low rank tensor algorithm is that we simultaneously evolve the macroscopic conservation laws of mass, momentum and energy using a flux-difference form with kinetic flux vector splitting; then the LoMaC property is realized by projecting the low rank kinetic solution onto a subspace that shares the same macroscopic observables by a conservative orthogonal projection. The algorithm is extended to the high dimensional problems by hierarchical Tuck decomposition of solution tensors and a corresponding conservative projection algorithm. Extensive numerical tests on the VP system are showcased for the algorithm's efficacy. 
\vfill

\newpage

%% file: intro.tex
\section{Introduction}

Numerical simulation of the Vlasov-Poisson (VP) system plays a fundamental role in understanding complex dynamics of plasma and has a wide range of applications in science and engineering, such as fusion energy. The well-known challenges for VP simulations include the high dimensionality of the phase space, resolution of multiple scales in time and in phase space, preservation of physical invariants, among many others. In this paper, we develop a novel Local Macroscopic Conservative (LoMaC) low rank tensor method with explicit time integrators that can conserve locally the mass, momentum and energy densities at the discrete level.

Over the past few decades, various types of numerical methods for the VP system have been successfully developed. The Particle-In-Cell (PIC) method employs a collection of sampled macro particles to represent the distribution function \cite{dawson1983particle,birdsall2004plasma} in the Lagrangian fashion, hence avoiding the curse of dimensionality. Meanwhile, it is well-known that the PIC method suffers the inherent statistical noise. Deterministic methods are developed under the grid-based Eulerian or semi-Lagrangian (SL) framework to compute the VP system, and are becoming popular recently, see e.g. \cite{filbet2003comparison}. Despite the high order accuracy for deterministic solvers, they are known to suffer from the bottleneck caused by the curse of dimensionality.
Several dimension reduction techniques have been developed.
One such example is the sparse grid approach \cite{smolyak1963quadrature,zenger1991sparse,griebel1990parallelizable}, which can effectively reduce the computational complexity and is well-suited for the problems with moderately high dimensions. For the Vlasov simulations, we mention the sparse grid SL method \cite{kormann2016sparse} and the sparse grid discontinuous Galerkin method \cite{guo2016sparse1,tao2018sparseguo}. Recently, the tensor approach emerged as a promising tool for feasible simulations of high-dimensional PDEs.  Such an approach aims to extract the underlying low rank structure of the solution data with advanced tensor decompositions, potentially breaking the curse of dimensionality. The popular tensor formats include the canonical polyadic (CP) format \cite{hitchcock1927expression,carroll1970analysis,harshman1970foundations,kolda2009tensor}, Tucker format \cite{tucker1966some,de2000multilinear}, hierarchical Tucker (HT) format \cite{hackbusch2009new,grasedyck2010hierarchical}, and tensor train (TT) format \cite{oseledets2009breaking,oseledets2011tensor,oseledets2012solution}. There are several pioneering works employing the low rank tensor approach for nonlinear simulations, including the low rank SL method in the TT format  \cite{kormann2015semi}, a low rank method with the CP format based on the underlying Hamiltonian formulation \cite{ehrlacher2017dynamical}, a dynamical low rank method proposed in \cite{einkemmer2018low,einkemmer2020low} for which the dynamical low rank approximation of the Vlasov solution is evolved on the low rank manifold using a tangent space projection, and dynamical tensor approximations for high dimensional linear and nonlinear PDEs based on functional tensor decomposition and dynamical tensor approximation \cite{dektor2020dynamically}. 

 In \cite{guo2021lowrank}, we proposed a low rank tensor VP solver to dynamically and adaptively build up low rank solution basis based on 
 the observation that the differential operator in the Vlasov equation can be represented in a tensorized form. In particular, we start from a low rank solution in a tensor format and add additional basis by applying the well-established high order finite difference upwind method coupled with the strong-stability-preserving (SSP) multi-step time discretizations \cite{gottlieb2011strong}; the solutions are being further updated by an SVD-type truncation to remove redundant bases.  We further generalize the algorithm to high-dimensional problems with the HT decomposition, which attains a storage complexity that is linearly scaled with the dimension, mitigating the curse of dimensionality. 
 
 On the other hand, due to the SVD truncation step, conservation properties are loss. Several techniques exist in the literature to correct conservation errors for low rank methods.  In \cite{kormann2015semi}, the low rank solution is rescaled so that  the total mass is conserved, and a similar mass correction technique is proposed in \cite{peng2020low} for a dynamical low rank method. 
In \cite{allmann2022parallel}, moment fitting is applied to the low rank solution so that the corrected moments match those solved from the macroscopic fluid equations.
In \cite{einkemmer2019quasi}, a dynamical low rank method with Lagrangian multipliers is developed to improve conservation properties for the total mass and momentum as well as local projected moment equations. More recently, along the same line, the truly local conservation of mass, momentum, and energy is attained for the dynamical low rank method \cite{einkemmer2021mass}. The idea is to fix certain basis functions in the dynamical low rank approximation and employ a modified Petrov–Galerkin formulation which is compatible with the remainder of the approximation.
In our recent work \cite{guo2022conservative}, a conservative SVD truncation is developed via an orthogonal projection to a subspace with conservation of macroscopic moments followed by a weighted SVD truncation performed on the remainder term. As a result, local mass and momentum conservation is achieved at the discrete level. However, the algorithm does not enjoy global or local energy conservation, as the associated full rank scheme can not conserve energy. In fact, an implicit symplectic time discretization is usually needed for exact energy conservation of a fully discrete scheme \cite{cheng2014energy}.   

In this paper, we develop a novel LoMaC low rank tensor method for the high dimensional Vlasov simulations. The key new ingredient is the simultaneous update of macroscopic conservation laws alongside the VP system and using them to define a reference subspace that shares the same macroscopic observables.  Figure~\ref{f1} highlights the flow chart of the algorithm. 
\begin{figure}[h!]
	\centering
	        {\includegraphics[height=20mm]{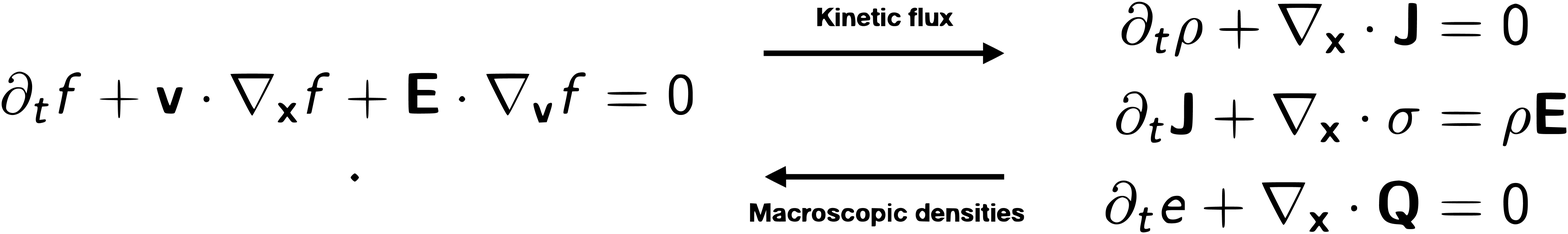}}
	        \caption{Illustration of LoMaC scheme.}
\label{f1}
\end{figure}
To be precise, the kinetic solution is used to construct numerical fluxes to update the macroscopic densities via the kinetic flux vector splitting (KFVS) for local conservation  \cite{mandal1994kinetic, xu1995gas}.
The low rank kinetic solution is orthogonally projected onto the reference subspace defined by macroscopic densities from conservation laws; 
 then a weighted SVD truncation is applied to the projection remainder to remove redundancy for data sparsity. We further develop the LoMaC algorithm for the 2D2V VP system with the HT tensor format using a dimension tree that separates the spatial and phase variables. For high order HT tensors, an additional projection step is needed after the hierarchical  high order SVD (HOSVD) truncation of the remainder term to ensure exact local conservation of macroscopic moments. In the proposed scheme, kinetic and fluid models complement each other. Kinetic model offers higher moments but lack certain conservation properties; while fluid models use kinetic solutions for fluxes and enjoys local conservation of the lower moments.  In the implementation, macroscopic fluid solvers and kinetic solvers are implemented alongside with each other in a self-consistent fashion, with little additional computational cost. We remark that, to the best of our knowledge, this is the first explicit low rank VP solver that achieves local energy conservation at the discrete level. The LoMaC low rank tensor algorithm is theoretically proved and numerically verified to be locally mass, momentum and energy conservative. 

This paper is organized as follows. In Section 2, we introduce the kinetic Vlasov model and the corresponding macroscopic conservation laws. In Section 3, we first review the low rank tensor approach for the 1D1V Vlasov equation in Section 3.1, then we review the conservative SVD truncation in Section 3.2, followed by the LoMaC algorithm with simultaneous update of macroscopic conservation laws using KFVS in Section 3.3. In Section 4, we develop the LoMaC algorithm for the 2D2V Vlasov model. In Section 5, we present an extensive set of 1D1V and 2D2V numerical results to demonstrate the effectiveness and the conservation properties of the proposed low rank tensor algorithm.  We conclude the main contributions of the paper and comment on future research directions in Section 6.

%% file: algorithm.tex
\section{The kinetic Vlasov model and the corresponding macroscopic systems}
We consider the dimensionless VP system 
\beq
\frac{\partial f}{\partial t}  
+  {\bf{v}} \cdot \nabla_{\bf{x}}  f 
+ {\bf{E}} ({\bf{x}},t) \cdot \nabla_{\bf{v}}  f = 0,
\label{vlasov1}
\eeq
\beq
 {\bf E}( {\bf x},t) = - \nabla_{\bf x} \phi({\bf x},t),  \quad -\triangle_{\bf x} \phi ({\bf x},t) = {{\bf \rho} ({\bf x},t)} - \rho_0,
\label{poisson}
\eeq
which describes the dynamics of  the probability distribution function $f({\bf x}, {\bf v},t)$ of electrons in a collisionless plasma. 
Here ${\bf E}$ is the electric field and $\phi$ is the self-consistent electrostatic potential  determined by  Poisson's equation. $f$ couples to the long range fields via the density ${\bf \rho}({\bf x},t) = \int_{\Omega_{\bv}} f({\bf x}, {\bf v},t) d {\bf v}$, where we take the limit of uniformly distributed infinitely massive ions in the background.  

The Vlasov dynamics are well-known to conserve several physical invariants. In particular, let 
\be
\label{eq: mass_d}
\mbox{mass density:}&& \rho (\bx, t) = \int_{\Omega_{\bv}} f(\bx, \bv,t) d \bv, \\
\label{eq: current_d}
\mbox{current density:} &&\bJ (\bx, t) = \int_{\Omega_{\bv}}f(\bx, \bv,t) \bv d \bv,\\
\label{eq: kenergy_d}
\mbox{kinetic energy density:} && \kappa(\bx,t) = \frac{1}{2} \int_{\Omega_{\bv}} |\bv|^{2}  f(\bx, \bv,t) d \bv,\\
\label{eq: energy_d}
\mbox{energy density:} && e(\bx,t)=\kappa(\bx,t)+\frac{1}{2} \bE(\bx)^2.
\ee
Then, by taking the first few moments of the Vlasov equation,
the following conservation laws of mass, momentum and energy can be derived
\begin{align}
\partial_{t} \rho + \nabla_\bx \cdot \bJ &= 0\label{eq:mass}\\
\partial_{t} \bJ +\nabla_{\bx} \cdot \mathbf{\sigma}&= \rho\bE \label{eq:mom}\\
\partial_{t} e +\nabla_{\bx} \cdot \mathbf{Q}& =0,\label{eq:ener} 
\end{align}
where $ \sigma(t, \bx)=\int_{\Omega_{\bv}}(\bv \otimes \bv) f(\bx, \bv,t) d \bv$ and $\mathbf{Q}(\bx,t) =\frac12\int_{\Omega_{\bv}}\bv|\bv|^2 f(\bx, \bv,t) d \bv$. 
It is well-known that local conservation property is essential to capture correct entropy solutions of hyperbolic systems such as \eqref{eq:mass}-\eqref{eq:ener}. 

\section{A low rank tensor approach for the Vlasov dynamics with local conservation}

For simplicity of illustrating the basic idea, we only discuss a 1D1V example in this section. 

\subsection{Review of a low rank tensor approach for Vlasov dynamics \cite{guo2021lowrank}}
The low rank tensor approach \cite{guo2021lowrank} is designed based on the assumption that our solution at time $t$ has a low rank representation in the form of 
\begin{equation}
\label{eq: fn1}
f(x, v, t) = \sum_{l=1}^{r} \left(C_l(t) \  U_l^{(1)}(x, t) U_l^{(2)}(v, t)\right),
\end{equation}
where $\left\{U_l^{(1)}(x, t)\right\}_{l=1}^{r}$ and $\left\{U_l^{(2)}(v, t)\right\}_{l=1}^{r}$ are a set of time-dependent low rank orthonormal  basis in $x$ and $v$ directions, respectively, $C_l$ is the coefficient for the basis $U_l^{(1)}(x, t)U_l^{(2)}(v, t)$, and $r$ is the representation rank.  \eqref{eq: fn1} can be viewed as a Schmidt decomposition of functions in $(x, v)$ by truncating small singular values up to rank $r$. 

We assume a finite difference discretization of $f$ on a truncated 1D1V domain of $[x_{\min}, x_{\max}] \times [-v_{\max}, v_{\max}]$ with uniform tensor product $N_x \times N_v$ grid points 
\beq
\label{eq: x_grid}
x_{\text{grid}}: \quad x_{\min}=x_1< \cdots < x_i < \cdots < x_{N_x} = x_{\max}, 
\eeq
\beq
\label{eq: v_grid}
v_{\text{grid}}: \quad -v_{\max}=v_1< \cdots < v_j <\cdots < v_{N_v} = v_{\max},
\eeq
and denote $h_x$ and $h_v$ as the mesh sizes in $x$- and $v$-directions, respectively.
The numerical solution ${\bf f} \in \mathbb{R}^{N_x \times N_v}$, as an approximation to point values of the solution on the grids \eqref{eq: x_grid}-\eqref{eq: v_grid}, has the corresponding low rank counterpart to \eqref{eq: fn1} as
\begin{equation}
\label{eq: fn2}
{\bf f} = \sum_{l=1}^{r} \left(C_l \  {\bf U}_l^{(1)}  \otimes {\bf U}_l^{(2)}\right), \quad 
(\mbox{or element-wise:} \quad
{f}_{ij} = \sum_{l=1}^{r}  C_l \ {U}_{l, i}^{(1)} {U}_{l, j}^{(2)}), 
\end{equation}
where ${\bf U}_l^{(1)} \in \mathbb{R}^{N_x}$ and ${\bf U}_l^{(2)} \in \mathbb{R}^{N_v}$ can be viewed as approximations to corresponding basis functions in \eqref{eq: fn1}. \eqref{eq: fn2} can also be viewed as an SVD of the matrix ${\bf f} \in \mathbb{R}^{N_x \times N_v}$. The associated storage cost is $\mathcal{O}(r N)$, where we assume $N=N_x = N_v$. 

Our low rank  tensor approach adaptively updates low-rank basis and associated coefficients by two steps: an adding basis step by conservative hyperbolic solvers and a removing basis step via an SVD-type truncation. We apply a second order SSP multi-step temporal discretization of 1D1V Vlasov equation \eqref{vlasov1} to illustrate the main idea.  We assume the solution in the form of \eqref{eq: fn2} with superscript $n$ for the solution at $t^n$.
\begin{enumerate}
\item {\em Add basis and obtain an intermediate solution ${\bf f}^{n+1, *}$.} 
A second order multi-step discretization of time derivative in \eqref{vlasov1} gives
\begin{equation}
\label{eq: fn3}
{f}^{n+1, *} = \frac14{f}^{n-2}+\frac34 {f}^{n}- \frac32\Delta t (v \partial_x ({f}^n) + E^n \partial_v ({f}^n)).
\end{equation}
Here the electric field $E^n$ is solved by a Poisson solver. Thanks to the tensor friendly form of the Vlasov equation, assuming the low rank format of solutions at $t^{n-2}$ and $t^n$, ${\bf f}^{n+1, *}$ can be represented in the following low rank format:
\begin{align}
\label{eq:lowrankmethod}
{\bf f}^{n+1, *} =& \frac14 \sum_{l=1}^{r^{n-2}} C_l^{n-2} \left( {\bf U}_l^{(1), n-2} \otimes {\bf U}_l^{(2), n-2}\right) 
+\frac34 \sum_{l=1}^{r^n} C_l^n \left( {\bf U}_l^{(1), n} \otimes {\bf U}_l^{(2), n}\right)  \\
& -\frac32  \Delta t 
\left( D_x {\bf U}_l^{(1), n} \otimes \bv \star {\bf U}_l^{(2), n} + \bE^n \star {\bf U}_l^{(1), n} \otimes D_v {\bf U}_l^{(2), n}
\right),
\end{align}
Here, with a slight abuse of notation, $\bv \in\mathbb{R}^{N_v}$ denotes the coordinates of $v_{grid}$ introduced in \eqref{eq: v_grid}. $D_x$ and $D_v$ represent high order {\em locally conservative upwind discretization} of spatial differentiation terms, and $\star$ denotes an element-wise multiplication operation. For example the discretization of $D_x {\bf U}_l^{(1), n} \otimes \bv \star {\bf U}_l^{(2), n}$ follows 
\begin{equation}
D^+_x {\bf U}_l^{(1), n} \otimes \bv^+ \star {\bf U}_l^{(2), n} + D^-_x {\bf U}_l^{(1), n} \otimes \bv^- \star {\bf U}_l^{(2), n}, 
\end{equation}
where $D^+_x$ and $D^-_x$ are a fifth order upwind finite difference discretization of positive and negative velocities respectively, with $\bv^+ = \max(\bv, 0)$ and $\bv^-=\min(\bv, 0)$. 
Similarly, the discretization of $\bE^n \star {\bf U}_l^{(1), n} \otimes D_v {\bf U}_l^{(2), n}$ follows 
 \beq
 \bE^{n, +} \star {\bf U}_l^{(1), n} \otimes D^+_v {\bf U}_l^{(2), n}+\bE^{n, -} \star {\bf U}_l^{(1), n} \otimes D^-_v {\bf U}_l^{(2), n}
 \eeq
where $D^+_v$ and $D^-_v$ are a fifth order upwind finite difference discretization of positive and negative velocities respectively, with $\bE^+ = \max(\bE, 0)$ and $\bE^-=\min(\bE, 0)$. 
\item {\em Remove basis of ${\bf f}^{n+1, *}$ to update solution ${\bf f}^{n+1}$.} Since the number of bases has increased in a single step update, we perform an SVD-type truncation to remove redundant bases with a prescribed threshold $\varepsilon$. The truncation step has no guarantee of any mass, momentum or energy conservation property. 
The removing basis step costs $\mathcal{O}(r^2 N + r^3)$, where $r$ is the SVD rank of the numerical solution.  
\end{enumerate}
In this two-step process, both the basis and coefficients are updated. Extensions to schemes with high order spatial and temporal discretizations and to high dimensional problems, are developed in \cite{guo2021lowrank}. The low rank approach  \cite{guo2021lowrank} is built upon the classical high order methods for conservation laws and kinetic equations, yet it optimizes the computational efficiency by dynamically building low rank global basis and updating the corresponding coefficients via an SVD truncation procedure. While the SVD truncation significantly reduces the computational storage and cost complexity, it also destroys the desired conservation property.

\subsection{A review of conservative SVD truncation for preserving mass, momentum and kinetic energy density \cite{guo2022lowrank}.} 
We proposed a conservative SVD truncation in \cite{guo2022lowrank}
for preservation of mass, momentum and kinetic energy density. The original idea in \cite{guo2022lowrank}, inspired by those in \cite{einkemmer2021mass}, is to first project the updated solution, ${\bf f}^{n+1, *}$ from \eqref{eq:lowrankmethod}, to a subspace 
\beq
\mathcal{N}\doteq \text{span}\{{\bf 1}_v, \bv, \bv^2\},
\eeq
where  ${\bf 1}_v\in \mathbb{R}^{N_v}$ is the vector of all ones, and $\bv^2$ $\in \mathbb{R}^{N_v}$ is the elment-wise square of $\bv$. To ensure proper decay of the projected function as $v \to \infty$, 
we introduced a weight function $w(v)$ with exponential decay.  One such example is $w(v) = \exp(-v^2/2)$, which is used throughout the paper unless otherwise specified. With the introduction of the weight function, a scaling and re-scaling procedure is needed for the projection step, as well as for the SVD-truncation step. 

To review the conservative truncation procedure \cite{guo2022lowrank}, we introduce the following definitions, 
\bit
\item Standard $l^2$ inner product and the associated norm: 
\beq
\label{eq: inner_prod_2_d}
\langle {\bf f},  {\bf g} \rangle = h_v \sum_j f_j g_j,  \quad \quad \|{\bf f}\|_2 =\sqrt{\langle {\bf f},  {\bf f} \rangle}
\eeq
where $h_v$ is the mesh size in $v$-direction, serving as the quadrature weights for the uniform $v_{\text{grid}}$ \eqref{eq: v_grid}. 
Correspondingly, we let
$
l^2 = \{{\bf f}\in\mathbb{R}^{N_v}: \|{\bf f}\|_2 < \infty\}.
$
\item Weighted inner product and the associated norm as
\beq
\label{eq: inner_prod_d}
\langle {\bf f},  {\bf g} \rangle_{\bf w} = \sum_j f_j g_j w_j, \quad 
\quad \|{\bf f}\|_{\bf w} =\sqrt{\langle {\bf f},  {\bf f} \rangle_{\bf w}}, 
\eeq
where ${\bf w}\in \mathbb{R}^{N_v}$ with $w_j = w(v_j) h_v$ is the quadrature weights for $v$-integration with weight function $w(v)$. Correspondingly, we let
$
l^2_{\bf w} = \{{\bf f}\in\mathbb{R}^{N_v}: \|{\bf f}\|_{\bf w} < \infty\}.
$
\eit 
Consider the subspace $\mathcal{N}\subset l^2_w$, a conservative low rank truncation of a numerical solution ${\bf f}\in \mathbb{R}^{N_x\times N_v} $ written in the low rank form of \eqref{eq: fn2} can be obtained from steps below. 
\begin{enumerate}
\item {\bf Compute macroscopic quantities of ${\bf f}$.} We compute the discrete macroscopic charge, current and kinetic energy density ${\boldsymbol \rho}$, ${\bf J}$ and ${\boldsymbol \kappa} \in \mathbb{R}^{N_x}$ by quadrature 
   \begin{align}
\left(\begin{array}{l}
{\boldsymbol \rho}\\
{\bf J}\\
{\boldsymbol \kappa} 
\end{array}
\right )
 = \sum_{l=1}^{r} C_l
 \left 
 \langle \bU^{(2)}_{l}, 
 \left(\begin{array}{l}
 {\bf 1}_v \\
\bv\\
\frac12\bv^2
\end{array}
\right )
\right \rangle
\ \bU^{(1)}_l. 
\label{eq:rho_j_kappa}
\end{align}   
\item {\bf Scale.} We scale ${\bf f}$ as 
\beq
\label{eq: rescale}
\tilde{\bf f} = \frac{1}{{\bf w}} \star {\bf f} =  \sum_{l=1}^{r} \left(C_l \ \ {\bf U}_l^{(1)}  \otimes \left(\frac{1}{{\bf w}} \star  {\bf U}_l^{(2)}\right)\right), 
\eeq
where $\star$ is the element-wise product in the $v$-dimension. 
\item {\bf Project.} We perform an orthogonal projection of $\tilde{\bf f}$ with respect to the inner product \eqref{eq: inner_prod_d} onto subspace $\mathcal{N}$, i.e.
\beq
\label{eq: proj}
\langle P_{\mathcal{N}}(\tilde{\bf f}), {\bf g} \rangle_\bw
= \langle \tilde{\bf f}, {\bf g} \rangle_\bw, 
\quad \forall {\bf g}\in \mathcal{N}. 
\eeq
It can be shown that ${\bf w}\star P_{\mathcal{N}}(\tilde{\bf f})$ preserves the mass, momentum and kinetic energy densities of ${\bf f}$ in the discrete sense. 
With the orthogonal project, a conservative decomposition of ${\bf f}$ \cite{guo2022lowrank} can be performed as 
\beq
\label{eq: f_decom_d}
{\bf f} = {\bf w} \star (P_{\mathcal{N}}(\tilde{{\bf f}}) + (I-P_{\mathcal{N}})(\tilde{{\bf f}})) 
\doteq {\bf w} \star (\tilde{{\bf f}}_1 + \tilde{{\bf f}}_2)
\doteq {\bf f}_1 + {\bf f}_2, 
\eeq
where $f_1$ can be represented as  a rank three tensor
 \begin{align}\label{eq:f1}
 {\bf f}_1 ({\boldsymbol \rho}, {\bf J}, {\boldsymbol \kappa})= & \frac{\boldsymbol \rho}{\|{\bf 1}_v\|_\bw^2} \otimes ({\bf w} \star {\bf 1}_v) 
 +  \frac{\bf J}{\|{\bf v}\|_{\bf w}^2} \otimes ({\bf w} \star {\bf v})  +  \frac{2 {\boldsymbol \kappa}-c{\boldsymbol \rho}}{\|{\bf v}^2- c {\bf 1}_v\|_{\bf w}^2} \otimes ({\bf w} \star ({\bf v}^2- c {\bf 1}_v) ), 
 \end{align} 
 where $c=\frac{\langle \bm{1}_v, {\bf v}^2\rangle_{\bf w}}{\|{\bf 1}_v\|_{\bf w}^2}$ is computed so that $\{ {\bf 1}_v, {\bf v}, {\bf v}^2- c {\bf 1}_v\}$ forms an orthogonal set of basis and 
   ${\boldsymbol \rho}$, ${\bf J}$ and ${\boldsymbol \kappa}$ are the discrete mass, momentum and kinetic energy density of ${\bf f}$ from \eqref{eq:rho_j_kappa}. ${\bf f}_1$ preserves the discrete mass, momentum and kinetic energy density of ${\bf f}$, while the remainder part ${\bf f}_2 = {\bf f} -{\bf f}_1$ has zero of them.
\item {\bf Truncate in $l^2_{\bf w}$.} We then perform an SVD truncation of the remainder part $\tilde{\bf f}_2$ from \eqref{eq: f_decom_d} with respect to the weighted inner product \eqref{eq: inner_prod_d}. With the scaling and rescaling by the weight function, the weighted SVD truncation writes
$
\mathcal{T}^{\bf w}_{\varepsilon} (\tilde{\bf f}_2) =\frac{1}{\sqrt{\bf w}}  \star \mathcal{T}_\varepsilon (\sqrt{\bf w} \star \tilde{\bf f}_2).
$
That is, ${\bf f}_2$ is truncated to
\beq
\label{eq: cons_trun_f2}
{\bf w}\star\mathcal{T}^{\bf w}_{\varepsilon} (\tilde{\bf f}_2) = \sqrt{\bf w}  \star \mathcal{T}_\varepsilon (\sqrt{\bf w} \star \tilde{\bf f}_2)= \sqrt{\bf w}  \star \mathcal{T}_\varepsilon (\frac{1}{\sqrt{\bf w}} \star {\bf f}_2).
\eeq
\item {\bf Update.} We obtain the low rank truncation of ${\bf f}$ with local mass, momentum and energy conservation, denoted as  
\beq
T_c({\bf f}) = {\bf f}_1 + {\bf w}\star\mathcal{T}^{\bf w}_{\varepsilon} (\tilde{\bf f}_2) ={\bf f}_1 +\sqrt{\bf w}  \star \mathcal{T}_\varepsilon (\frac{1}{\sqrt{\bf w}} \star {\bf f}_2). 
\label{eq: Tc}
\eeq
We call the proposed truncation \eqref{eq: Tc} the conservative truncation, as $T_c({\bf f})$ exactly preserves the mass, momentum and kinetic energy density of ${\bf f}$. 
\end{enumerate}

In \cite{guo2022lowrank}, we established the local conservation of mass and momentum in the low rank tensor approach with the conservative truncation \eqref{eq: Tc}. Since the associated full rank algorithm (without truncation) does not have energy conservation property, the low rank tensor scheme cannot preserve energy conservation. In fact, an implicit symplectic type time discretization is often needed for the kinetic scheme for energy conservation, e.g. see \cite{cheng2014energy}.  In Figure \ref{fig:bump1d_old},   we present the time evolution of relative deviation of the total mass, total momentum and total energy of the method in \cite{guo2022lowrank} for simulating the bump-on-tail instability test with truncation threshold $\varepsilon=10^{-4}$ (see Example \ref{ex:bumpontail} in Section \ref{sec:numerical}). It is observed that the total mass is well conserved up to the machine precision. Meanwhile, with a coarse mesh size $32\times64$, it is found that the conservation error of the total momentum starts to increase at $t=15$, which is attributed to the boundary error as discussed in \cite{guo2022lowrank}. The total energy conservation is not observed as expected.
\begin{figure}[h!]
	\centering
			\subfigure[]{\includegraphics[height=40mm]{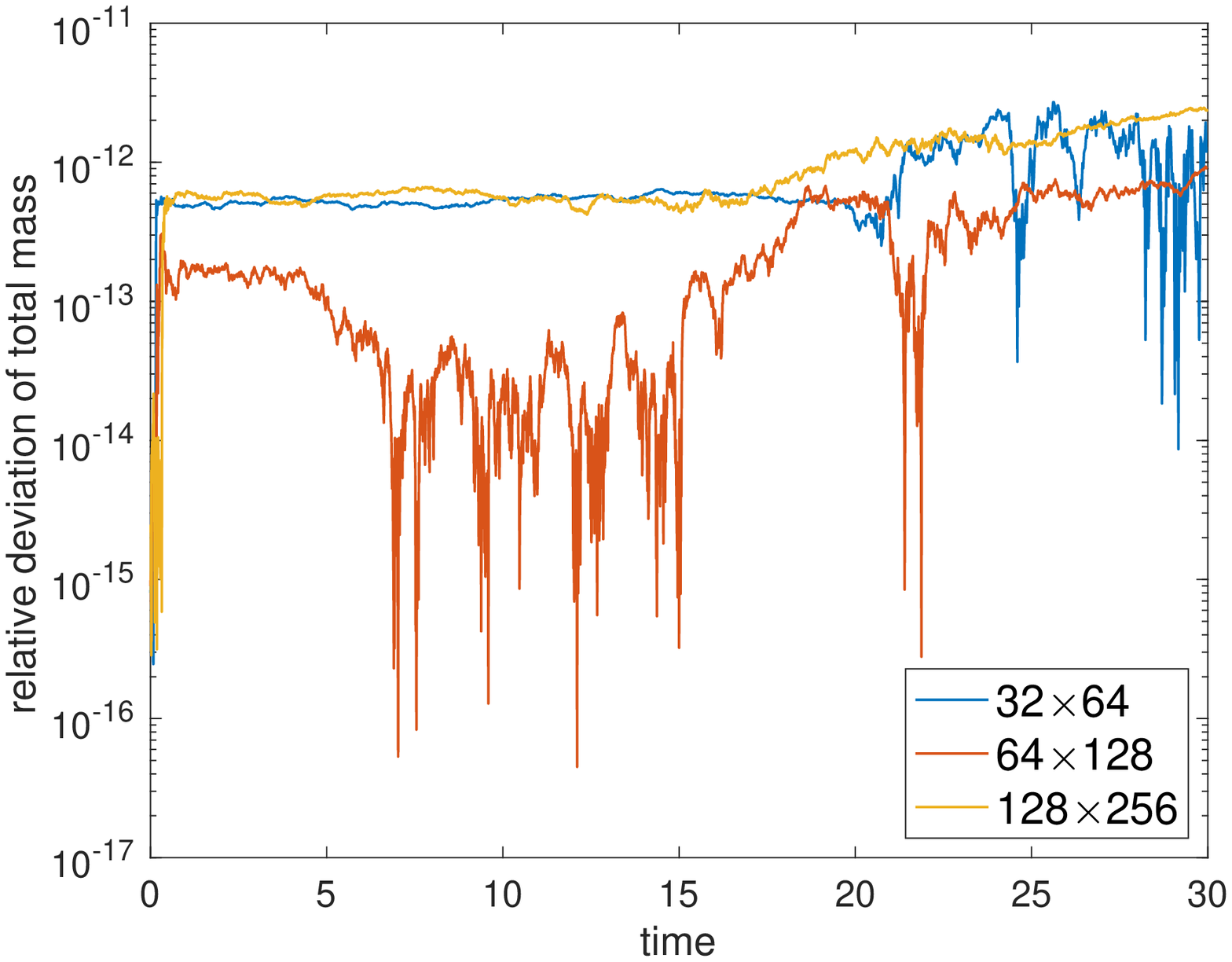}}
	       \subfigure[]{\includegraphics[height=40mm]{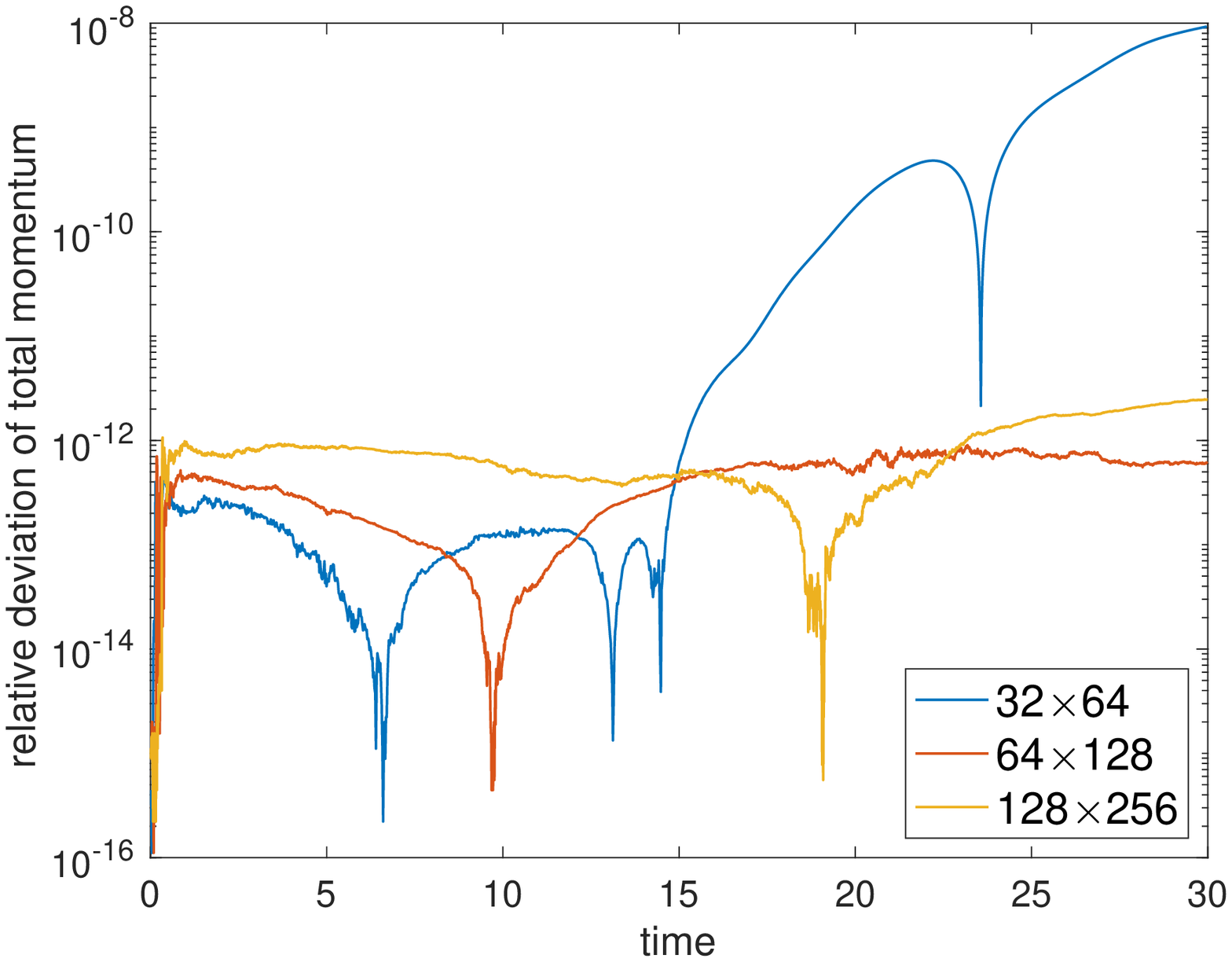}}
		\subfigure[]{\includegraphics[height=40mm]{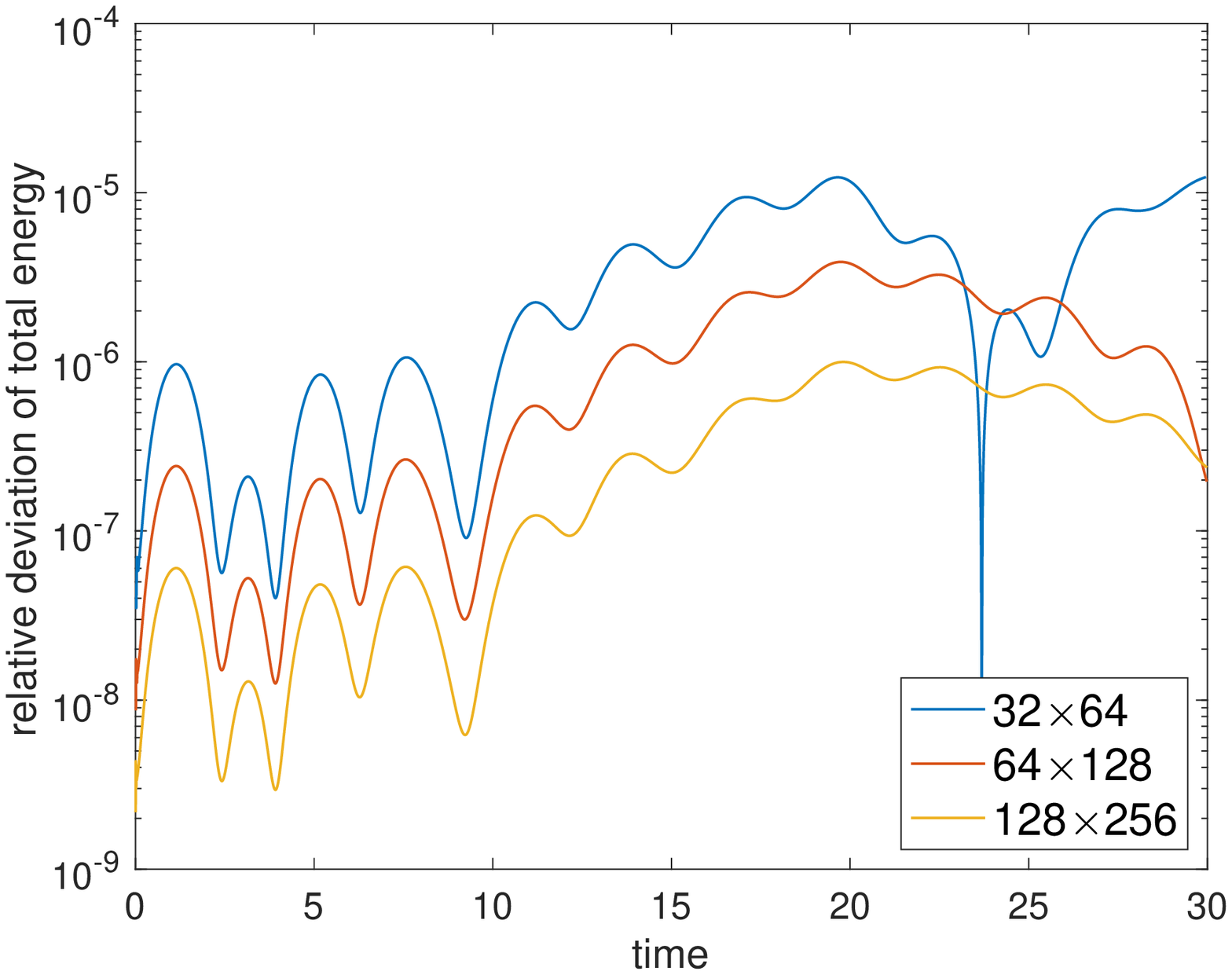}}
	\caption{Bump-on-tail instability.  Relative deviation of the total mass (a),  total momentum (b), and total energy (c). Conservative low rank method \cite{guo2022lowrank}. $\varepsilon=10^{-4}$.}
	\label{fig:bump1d_old}
\end{figure}

Finally, the proposed conservative low rank tensor algorithm, based on a finite difference scheme with fifth order spatial discretization and the second order SSP multi-step temporal discretization,  is summarized in the Algorithm \ref{alg: low_rank}.

\bigskip
\begin{algorithm}[H]
\label{alg: low_rank}
  \caption{The conservative low rank tensor algorithm for the 1D1V VP system.}
    \begin{enumerate}
\item Initialization:
  \begin{enumerate}
  \item 
   Initial distribution function $f(x, v, t=0)$ in a low rank format \eqref{eq: fn2}. 
 \end{enumerate}
\item For each time step evolution from $t^n$ to $t^{n+1}$: update ${\bf f}^{n+1}$ from ${\bf f}^n$ in the low rank format.
  \begin{enumerate}
  \item Compute the charge density $\boldsymbol \rho^{n}$ in the low rank format, followed by computing $\bE^n$ from Poisson equation's \eqref{poisson} by fast Fourier transform or a high order finite difference algorithm.
  \item Add basis by a finite difference upwind spatial discretization coupled with a second order SSP multi-step method
  \be
{\bf f}^{n+1, *} &&=  \frac14 \sum_{l=1}^{r^{n-2}} \left(C_l^{n-2} \ \ {\bf U}_l^{(1), n-2} \otimes {\bf U}_l^{(2), n-2}\right)+ \frac34 \sum_{l=1}^{r^n} \left(C_l^n \ \ {\bf U}_l^{(1), n} \otimes {\bf U}_l^{(2), n}\right) \nonumber
\\\nonumber
&&- \frac32\Delta t \sum_{l=1}^{r^n} C_l^n \left[ 
\left( D^+_x {\bf U}_l^{(1), n} \otimes \bv^+ \star {\bf U}_l^{(2), n} + D^-_x {\bf U}_l^{(1), n} \otimes \bv^- \star {\bf U}_l^{(2), n} \right.\right.\\ \nonumber
&&\left.\left.+ \bE^{n, +} \star {\bf U}_l^{(1), n} \otimes D^+_v {\bf U}_l^{(2), n}+\bE^{n, -} \star {\bf U}_l^{(1), n} \otimes D^-_v {\bf U}_l^{(2), n}
\right)\right]. 
\label{eq: fn+1star}
\ee
\item Remove basis by a conservative SVD truncation from ${\bf f}^{n+1,*}$ in the low rank format.
  \begin{enumerate}
  \item Compute ${\boldsymbol \rho}$,  ${\bf J}$,  ${\boldsymbol \kappa}$ of ${\bf f}^{n+1,*}$ from \eqref{eq:rho_j_kappa}. 
  \item Compute ${\bf f}_1$ from \eqref{eq:f1}.
\item Perform the truncation of ${\bf f}_2$ by \eqref{eq: cons_trun_f2}. 
  \item Update the compressed low-rank solution via \eqref{eq: Tc}, 
  $${\bf f}^{n+1} = T_c({\bf f}) = {\bf f}_1 +\sqrt{\bf w}  \star \mathcal{T}_\varepsilon (\frac{1}{\sqrt{\bf w}} \star {\bf f}_2).$$
\end{enumerate}
\end{enumerate}
\end{enumerate}
  \end{algorithm}

  \subsection{Local macroscopic conservation (LoMaC) achieved by kinetic flux vector splitting schemes for macroscopic equations}
  \label{sec: LoMac1D1V}

\eqref{eq:f1} implies the following observation for the orthogonal projection and decomposition of ${\bf f}$  in the low rank format:  ${\bf f}_1$ can be uniquely determined by macroscopic $\boldsymbol\rho$, $\bJ$ and $\boldsymbol\kappa$. On the other hand, it has been known that numerical methods for system of conservation laws, if being written in the flux-difference form, can locally preserve the macroscopic quantities. 

We propose to update the macroscopic mass, momentum and kinetic energy density by classical numerical methods in a flux-difference form via a high order discretization of macroscopic system \eqref{eq:mass}-\eqref{eq:ener}. Since the kinetic solution ${\bf f}$ is known, flux functions for \eqref{eq:mass}-\eqref{eq:ener} can be found by taking the upwind components and perform integration in velocity directions as in kinetic flux vector splitting \cite{mandal1994kinetic, xu1995gas}. Once these macroscopic quantities are updated, they are plugged into \eqref{eq:f1} to construct a new ${\bf f}_1^{M}$ (replace the ${\bf f}_1$ from the projection of the kinetic solution). The remainder part ${\bf f}_2 ={\bf f}-{\bf f}_1$ stays the same with zero macroscopic $\boldsymbol\rho$, $\bJ$ and $\boldsymbol\kappa$. In other words, we perform a correction step on the first few moments of ${\bf f}$, from using a conservative kinetic flux vector splitting scheme for macroscopic equations, to ensure local macroscopic conservation. 

Below we describe the conservative update of macroscopic variables, denoted as $\boldsymbol\rho^{n+1, M}$, $\bJ^{n+1, M}$, $\boldsymbol\kappa^{n+1, M}$, by a conservative scheme in the flux-difference form with the kinetic flux vector splitting. Let $U \doteq (\rho, {\bf J}, {e})^\top$,  $F \doteq ({\bf J}, \sigma,  {\bf Q})^\top$ and  $S = (0, \rho E,0)^\top$, then the macroscopic system \eqref{eq:mass}-\eqref{eq:ener} becomes
\beq
\label{eq:U}
U_t + F_x = S.
\eeq
Assuming the same spatial grid \eqref{eq: x_grid},  the algorithm with the high order upwind finite difference spatial discretization coupled with the second order SSP multi-step time integrator for system \eqref{eq:U} becomes
\begin{equation}
\label{eq:Uupdate}
U_j^{n+1} = \frac14U^{n-2}_j + \frac34U^{n}_j + \frac32\Delta t\left(-\frac{1}{h_x} \left( \hat{F}^n_{j+\frac12} -\hat{F}^n_{j-\frac12}\right) + S_j^n\right),
\end{equation}
where $U^{n}_j = (\rho_j^n, \bJ^{n}_j , e^{n}_j)^\top$ and $S_j^n = (0, \rho_j^nE_j^n,0)^\top$, $j=1,\ldots,N_x$. 
The numerical fluxes are uniquely defined at cell interfaces and is given by the following upwind splitting 
\begin{equation}
\hat{F}^n_{j+\frac12} = \hat{F}^{n, +}_{j+\frac12} + \hat{F}^{n, -}_{j+\frac12}, \quad j=1,\ldots,N_x.
\end{equation}
To obtain $\hat{F}^{n, \pm}_{j+\frac12}$ with high order spatial accuracy in an upwind fashion, assuming the kinetic solution ${\bf f}^n$ in a low rank format \eqref{eq: fn2},  we first compute ${\bf F}^{n, +}$ and ${\bf F}^{n, -}\in \mathbb{R}^{N_x}$
 \begin{align}
{\bf F}^{n, +}
 = \sum_{l=1}^{r^n} C^{n}_l
 \left 
 \langle \bU^{(2), n}_{l}, 
 \left(\begin{array}{c}
\bv^+\\
(\bv^+)^2\\
\frac12 (\bv^+)^3
\end{array}
\right )
\right \rangle
\ \bU^{(1), n}_l, \quad
{\bf F}^{n, -}
 = \sum_{l=1}^{r^n} C^n_l
 \left 
 \langle \bU^{(2), n}_{l}, 
 \left(\begin{array}{c}
\bv^-\\
(\bv^-)^2\\
\frac12 (\bv^-)^3
\end{array}
\right )
\right \rangle
\ \bU^{(1), n}_l,
\label{eq:Fpm}
\end{align}
where $\bv^+ = \max(\bv, 0)$, $\bv^- = \min(\bv, 0)$ and the inner product $\langle \cdot, \cdot \rangle$ is in the sense of \eqref{eq: inner_prod_2_d}. 
Let $F^{n, \pm}_j ={\bf F}^{n, \pm}(j)$, the upwind fluxes $\hat{F}^{n, \pm}_{j+\frac12}$ are reconstructed from ${\bf F}^{n, \pm}(:)$ in the following way using the corresponding high order upwind stencils \cite{shu2009high}, 
\begin{align*}
\hat{F}^{n, -}_{j+\frac12} &= -\frac{1}{20}F_{j-1}^{n, -} + \frac{9}{20}F_{j}^{n, -} + \frac{47}{60}F_{j+1}^{n, -} - \frac{13}{60}F_{j+2}^{n, -}+ \frac{1}{30}F_{j+3}^{n, -},\\
 \hat{F}^{n, +}_{j+\frac12} &= \frac{1}{30}F_{j-2}^{n, +} -\frac{13}{60}F_{j-1}^{n, +} + \frac{47}{60}F_{j}^{n, +} +  \frac{9}{20}F_{j+1}^{n, +} -\frac{1}{20}F_{j+2}^{n, +}.
\end{align*}
Then we let 
$
 \left(\begin{array}{l}
{\rho}_j^{n+1, M} \\
{\bJ}_j^{n+1, M}\\
 {e}_j^{n+1, M}
\end{array}
\right )$ be $U^{n+1}_j$ updated from \eqref{eq:Uupdate}, from which we can compute 
\begin{equation}
\label{eq:kinetic_update}
{\kappa}_j^{n+1, M} = {e}_j^{n+1, M} - \frac12|E^{n+1}_j|^2
\end{equation}
from \eqref{eq: energy_d} where $\bE^{n+1}$ can be computed directly from $\boldsymbol \rho^{n+1,M}$ via Poisson's equation.  Finally, we construct ${\bf f}^M_1$ according to \eqref{eq:f1}, which replaces ${\bf f}_1$ in \eqref{eq:f1}. Such a replacement can be viewed as a correction step for macroscopic conservation. 
Meanwhile, the treatment for ${\bf f}_2$ in the orthogonal decomposition \eqref{eq: f_decom_d} stays the same. That is ${\bf f}_2$ is truncated from \eqref{eq: cons_trun_f2}, making sure  it still contains zero mass, momentum and kinetic energy densities after truncation.

We summarize the newly proposed LoMaC low rank tensor algorithm, based on a finite difference scheme with fifth order spatial discretization and second order SSP multi-step temporal discretization, in Algorithm \ref{alg: low_rank_lomac}.

\bigskip
\begin{algorithm}[H]
\label{alg: low_rank_lomac}
  \caption{The LoMaC low rank tensor algorithm for the 1D1V VP system.}
    \begin{enumerate}
\item Initialization:
  \begin{enumerate}
  \item Algorithm 1 Step 1.
 \end{enumerate}
\item For each time step evolution from $t^n$ to $t^{n+1}$: update ${\bf f}^{n+1}$ from ${\bf f}^n$ in the low rank format.
  \begin{enumerate}
  \item Algorithm 1 Step 2 (a).
  \item Update ${\bf f}^{n+1,*}$ by Algorithm 1 Step 2 (b).
  Compute ${\boldsymbol \rho}^{n+1, *}$,  ${\bf J}^{n+1, *}$,  ${\boldsymbol \kappa}^{n+1, *}$ by numerical integration in velocity, i.e. \eqref{eq:rho_j_kappa}. Compute ${\bf f}_1$ from \eqref{eq:f1} with $\boldsymbol\rho^{n+1, *}$, $\bJ^{n+1, *}$, $\boldsymbol\kappa^{n+1, *}$. 
  \item  Compute ${\bf f}_2\doteq {\bf f} - {\bf f}_1$ and perform a weighted SVD truncation on ${\bf f}_2$ in the low rank format to obtain
  $
  \sqrt{\bf w}  \star \mathcal{T}_\varepsilon (\frac{1}{\sqrt{\bf w}} \star {\bf f}_2).
  $
  \item Compute ${\bf f}^M_1$.
  \begin{enumerate}
  \item Update macroscopic mass, momentum and energy density, $\boldsymbol\rho^{n+1, M}$, $\bJ^{n+1, M}$, ${\boldsymbol e}^{n+1, M}$, using the kinetic flux vector splitting, in a flux-difference form for \eqref{eq:mass}-\eqref{eq:ener} using the same second order SSP multi-step method \eqref{eq:Uupdate}. 
  \item Compute $E^{n+1, M}$ from $\boldsymbol\rho^{n+1, M}$ by Poisson solver. 
  \item Compute $\boldsymbol\kappa^{n+1, M}$ via \eqref{eq:kinetic_update}. 
\item  Construct ${\bf f}^M_1$ by $\boldsymbol\rho^{n+1, M}$, $\bJ^{n+1, M}$, $\boldsymbol\kappa^{n+1, M}$ according to \eqref{eq:f1}.
\end{enumerate}
  \item Update the compressed low-rank solution via \eqref{eq: Tc}, 
  $${\bf f}^{n+1} \doteq T^M_c({\bf f}) = {\bf f}^M_1 +\sqrt{\bf w}  \star \mathcal{T}_\varepsilon (\frac{1}{\sqrt{\bf w}} \star {\bf f}_2).$$
\end{enumerate}
\end{enumerate}
  \end{algorithm}

In summary, the proposed LoMaC low rank update of the VP solution starts with an adding basis step that employs a traditional high order finite difference scheme and an SSP multi-step time integrator. The algorithm is followed by an update of macroscopic conservation laws using KFVS, together with a projection of the low rank solution to enjoy the same macroscopic mass, momentum and energy density as the macroscopic conservation laws. Last, we apply an SVD type truncation step to remove redundancy in basis to ensure the low rank solution representation. Note that for one step evolution,  macroscopic and kinetic parts are independent except using $\boldsymbol\rho^{n+1, M}$, $\bJ^{n+1, M}$, $\boldsymbol\kappa^{n+1, M}$ to construct ${\bf f}^M_1$ from \eqref{eq:f1}.

\begin{prop} (Local mass, momentum and energy conservation.) The proposed LoMaC low rank algorithm locally conserves the macroscopic mass, momentum and energy. 
\end{prop}
\begin{proof} The proof follows directly from the construction of the algorithm.
\end{proof}

 \section{2D2V Vlasov-Poisson system by the HT format}
 We extend the proposed conservative algorithm to the 2D2V case by the HT format. 
 Below, we briefly review the fundamentals of the HT format for efficiently representing tensors in $d$ dimensions, and the low rank tensor method with the HT format for solving the 2D2V VP system \eqref{vlasov1}. 
\begin{equation}\label{eq:vp4d}
	f_t + v_1f_{x_1} + v_2f_{x_2}+ E_1f_{v_1} + E_2f_{v_2} = 0,
\end{equation}
where the electric field $(E_1, E_2)$ is solved from the coupled Poisson's equation. The macroscopic equations can be obtained from taking moments of \eqref{eq:vp4d} in the form of \eqref{eq:mass}-\eqref{eq:ener}. In this paper, we use full grid (i.e. not low rank) representation for the spatial variables $(x_1, x_2)$, due to the need to solve macroscopic equations by classical conservative flux-difference numerical schemes in the proposed LoMaC algorithm framework. It is possible to further explore the low rank structure in $(x_1, x_2)$ direction, which is left as our future work. 
 
The HT format is fully characterized by the three key components, including a dimension tree, frames at leaf nodes and transfer tensors at non-leaf nodes, see {Figure \ref{fig:dimtree1}} for the data layout. 
 In particular, we denote the dimension index $D=\{(1,2), 3, 4\}$  and define a \emph{dimension tree} $\mathcal{T}$ which is a binary tree containing a subset $\alpha\subset D$ at each node. Furthermore, $\mathcal{T}$ has $D$ as the root node and $\{(1, 2), 3, 4\}$ as the leaf nodes. The non-leaf node $\alpha$ has two children nodes.
  For example, the dimension tree $\mathcal{T}$ given in {Figure \ref{fig:dimtree1}} can be used to approximate $f((x_1,x_2),v_1,v_2)$ in \eqref{eq:vp4d} in the HT format, 
 \beq
\label{eq:htd_f_nested_0}
{\bf f} = \sum_{l_{12}=1}^{r_{12}}{\sum_{l_{34}=1}^{r_{34}}} \bB^{(1,2, 3, 4)}_{l_{12},l_{34}}\bU_{l_{12}}^{(1, 2)}\otimes \bU_{l_{34}}^{(3, 4)},
\eeq
  with 
\beq
\label{eq: U34_htd}
\bU_{l_{34}}^{(3, 4)} = \sum_{l_{3}=1}^{r_{3}}{\sum_{l_{4}=1}^{r_{4}}} \bB^{(3, 4)}_{l_{3},l_{4}, l_{34}}\bU_{l_{3}}^{(3)}\otimes \bU_{l_{4}}^{(4)},\quad l_{34} =1,\ldots,r_{34}. 
\eeq
 Here the tensor stores frames at each leaf node (i.e. $\bU^{(1, 2)}$, $\bU^{(3)}$ and $\bU^{(4)}$) and a third order transfer tensor at each non-leaf node (i.e. $\bB^{(1,2, 3, 4)}$ and $\bB^{(3, 4)}$) based on the dimension tree.  Denote $ \br = \{r_\alpha\}_{\alpha\in\mathcal{T}}$ as the hierarchical ranks. The storage of the HT format scales as $\mathcal{O}(2r^3+rN_{x_1}N_{x_2} + r(N_{v_1}+N_{v_2}))$, where $r=\max \br$ and $N_{\cdot}$ is the number of grid points in the corresponding dimension. If $r$ is reasonably low, then the HT format avoids the curse of dimensionality. 
 
 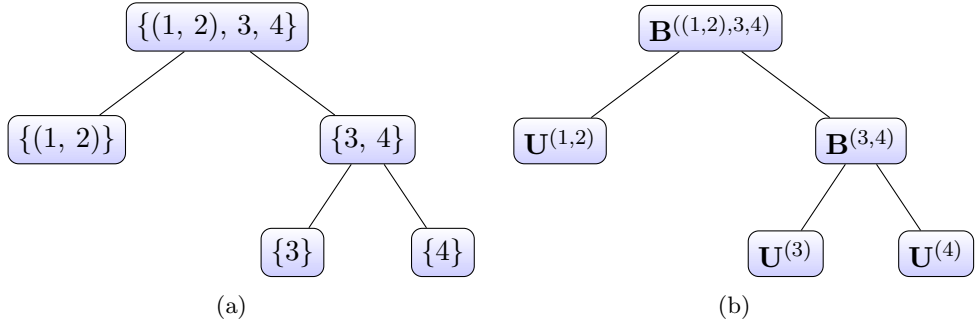
\begin{figure}
\centering
\subfigure[]{
 \begin{tikzpicture}[
     level/.style={sibling distance=40mm/#1},
  every node/.style = {shape=rectangle, rounded corners,
    draw, align=center,
    top color=white, bottom color=blue!20}
    ]
  \node {$\{(1,\,2),\,3,\,4\}$}
    child { node {$\{(1,\,2)\}$} 
    }
    child { node {$\{3,\,4\}$}
    	child{ node{$\{3\}$}}
	child{ node{$\{4\}$}} 
	};
\end{tikzpicture}}\quad
\subfigure[]{
\begin{tikzpicture}[
  every node/.style = {shape=rectangle, rounded corners,
    draw, 
    top color=white, bottom color=blue!20},
    level/.style={sibling distance=40mm/#1}
    ]
  \node {$\bB^{((1,2),3,4)}$}
      child { node {$\bU^{(1,2)}$} 
    }
    child { node {$\bB^{(3,4)}$}
    	child{ node{$\bU^{(3)}$}}
	child{ node{$\bU^{(4)}$}} 
	};
\end{tikzpicture}}
 \caption{Dimension tree $\mathcal{T}$ and associated data layout to express fourth-order tensors in the HT format. }
\label{fig:dimtree1}
\end{figure}


\subsection{A LoMaC low-rank tensor method in HT for the 2D2V VP system}
We follow the conservative low rank tensor method for updating the 2D2V VP solution in \cite{guo2022conservative}, and further propose a new LoMaC version for local energy conservation property in a similar spirit to the 1D1V system.  We assume at each time step, the solution ${\bf f}$ is expressed as the third-order tensor in the HT format \eqref{eq:htd_f_nested_0}-\eqref{eq: U34_htd} 
with dimension tree $\mathcal{T}$ as shown in Figure \ref{fig:dimtree1}. 

In the proposed 2D2V LoMaC algorithm, the computation of the projection operator $P_{\mathcal{N}}$, as well as how ${\bf f}_1$ depends on macroscopic conservative variables, are essential. Their computations in the 2D2V case, with the new dimension tree in Figure \ref{fig:dimtree1}, are slightly different from that in \cite{guo2022conservative}. Here 
 \beq
 \label{eq: 2d2v_subspa}
 \mathcal{N}
 = \text{span}\{{\bf 1}_{v_1 \otimes v_2}, {\bf v}_1\otimes {\bf 1}_{v_2},  {\bf 1}_{v_1}\otimes{\bf v}_2, {\bf v}_1^2\otimes {\bf 1}_{v_2}+{\bf 1}_{v_1}\otimes{\bf v}_2^2\}.
 \eeq
 We consider the Hilbert space with the weighted inner product $ \langle f, g\rangle_{{\bf w}^{(1)}}$,  $ \langle f, g\rangle_{{\bf w}^{(2)}}$,  $ \langle f, g\rangle_{({\bf w}^{(1)}\otimes {\bf w}^{(2)})}$. Here, ${\bf w}^{(1)}$ and ${\bf w}^{(2)}$ are vectors consists of point values of the weight function (e.g. $w(v) = \exp(-\frac{v^2}{2})$) on the corresponding velocity grid. 
In particular, $ \langle f, g\rangle_{{\bf w}^{(1)}}$ and $ \langle f, g\rangle_{{\bf w}^{(2)}}$ are defined similar to 
\eqref{eq: inner_prod_d}, and
\beq
 \label{eq: 2d2v_inner_w}
  \langle f, g\rangle_{({\bf w}^{(1)}\otimes {\bf w}^{(2)})} = h_{v_1}h_{v_2} \sum_{j_1=1}^{N_{v_1}} \sum_{j_2=1}^{N_{v_2}}  f_{j_1,j_2} g_{j_1,j_2} w^{(1)}_{j_1} w^{(2)}_{j_2},
 \eeq
which is in analog to the weighted inner product at the continuous level $\int f(\bv)g(\bv)w(\bv)d\bv$. 
 
 We first construct a set of orthonormal basis of $\mathcal{N}$, 
 $\{\bV_1, \cdots, \bV_4\}$ in the $(v_1,v_2)$ dimensions from a set of orthonormal basis for $v_1$ and $v_2$ directions as
%
\begin{eqnarray}
\bV_1 &=& \frac{1}{c_1^2} {\bf 1}_{v_1}\otimes {\bf 1}_{v_2},  \quad
\bV_2 = \frac{1}{c_1 c_2} {\bf v}_1\otimes {\bf 1}_{v_2}, \quad
\bV_3 =  \frac{1}{c_1 c_2} {\bf 1}_{v_1}\otimes {\bf v}_2, \nonumber\\
\bV_4 &=& \frac{1}{\sqrt{2}}\left(\frac{1}{c_1c_3}\left(({\bf v}_1^2-c{\bf 1}_{v_1})\right) \otimes ({\bf 1}_{v_2})+\frac{1}{c_1c_3}({\bf 1}_{v_1})\otimes\left(({\bf v}_2^2-c {\bf 1}_{v_2})\right)\right), 
\label{eq: V1_4}
\end{eqnarray} 
with constant $c= \frac{\langle \mathbf{1}_{v_1},\bv_1^2 \rangle_{{\bf w}^{(1)}} }{\langle {\mathbf{1}_{v_1},\mathbf{1}_{v_1}}\rangle_{{\bf w}^{(1)}}}$ for orthogonalization of the basis. $c_l$, $l=1, 2, 3$ are normalization constants for the corresponding basis of ${\bf 1}_{v_1}$, ${\bf v}_1$ and ${\bf v}_1^2-c  {\bf 1}_{v_1}$, where we have assumed the same weight function and discretization in the $v_1$ and $v_2$ directions for simplicity.
With the rescaling, 
\beq
\label{eq: U34_htd_f1}
(\bU_1^{(3, 4)})_{l_{34}} = \sum_{l_{3}=1}^{3}{\sum_{l_{4}=1}^{3}} (\bB_1^{(3, 4)})_{l_{3},l_{4}, l_{34}} ({\bf w}^{(1)}\star(\bU_1^{(3)})_{l_{3}})\otimes ({\bf w}^{(2)}\star(\bU_1^{(4)})_{l_{4}}),\quad l_{34} =1,\ldots,4.
\eeq 
That is, 
\begin{eqnarray}
(\bU_1^{(3, 4)})_1 &=& \frac{1}{c_1^2} ({\bf w}^{(1)}\star{\bf 1}_{v_1})\otimes ({\bf w}^{(2)}\star{\bf 1}_{v_2}),  \nonumber\\
(\bU_1^{(3, 4)})_2 &=& \frac{1}{c_1 c_2} ({\bf w}^{(1)}\star{\bf v}_1)\otimes ({\bf w}^{(2)}\star{\bf 1}_{v_2}), \nonumber\\
(\bU_1^{(3, 4)})_3 &=& \frac{1}{c_1 c_2} ({\bf w}^{(1)}\star{\bf 1}_{v_1})\otimes ({\bf w}^{(2)}\star{\bf v}_2), \nonumber\\
(\bU_1^{(3, 4)})_4 &=& \frac{1}{\sqrt{2}}\left(\frac{1}{c_1c_3}\left({\bf w}^{(1)}\star({\bf v}_1^2-c {\bf 1}_{v_1})\right) \otimes ({\bf w}^{(2)}\star{\bf 1}_{v_2})+\frac{1}{c_1c_3}({\bf w}^{(1)}\star{\bf 1}_{v_1})\otimes\left({\bf w}^{(2)}\star({\bf v}_2^2-c {\bf 1}_{v_2})\right)\right). \nonumber\\
\label{eq: U1_4}
\end{eqnarray} 
In particular, we construct the three frame vectors for node $\{3\}$ as
\begin{eqnarray}
(\bU_1^{(3)})_1 = \frac{1}{c_1} {\bf w}^{(1)}\star{\bf 1}_{v_1} , \quad
(\bU_1^{(3)})_2 = \frac{1}{c_2}{\bf w}^{(1)}\star{\bf v}_1 , \quad
(\bU_1^{(3)})_3 = \frac{1}{c_3}{\bf w}^{(1)}\star({\bf v}_1^2-c {\bf 1}_{v_1}).
\label{eq: U3U4}
\end{eqnarray}
We have the same three frame vectors for the node $\{4\}$ but for $v_2$, again assuming that the weight function and discretization in $v_2$ is the same as $v_1$, 
\begin{eqnarray}
(\bU_1^{(4)})_1 = \frac{1}{c_1} {\bf w}^{(2)}\star{\bf 1}_{v_2} , \quad
(\bU_1^{(4)})_2 = \frac{1}{c_2}{\bf w}^{(2)}\star{\bf v}_2 , \quad
(\bU_1^{(4)})_3 = \frac{1}{c_3}{\bf w}^{(2)}\star({\bf v}_2^2-c {\bf 1}_{v_2}).
\label{eq: U3U4_b}
\end{eqnarray} 
The transfer tensor $\bB_1^{(3, 4)}$ is a third order tensor of size $3\times 3 \times4$. It has zero elements, except the following specification for $(\bB_1^{(3, 4)})_{l_{3},l_{4}, l_{34}}$
\begin{equation}
\label{eq:B34f}
(\bB_1^{(3,4)})_{1,1,1} =(\bB_1^{(3,4)})_{2,1,2}=(\bB_1^{(3,4)})_{1,2,3}= 1, \quad
(\bB_1^{(3,4)})_{3,1,4} = (\bB_1^{(3,4)})_{1,3,4} = \frac1{\sqrt{2}}.
\end{equation} 


\begin{prop}\label{prop:2d2v} 
Let ${\bf f}_1$ come from the scaling/rescaling, together with the orthogonal projection of scaled ${\bf f}$ with respect to the weighted inner product \eqref{eq: 2d2v_inner_w} onto the subspace $\eqref{eq: 2d2v_subspa}$.  Assume ${\bf f}$ is written in the low rank HT format \eqref{eq:htd_f_nested_0}. ${\bf f}_1$ can be represented in  low rank HT format (consistently with the subscript $1$ in the notations),
\begin{eqnarray}
\label{eq:htd_f1_nested}
{P}_{\mathcal{N}} ({\bf f}) & \doteq{\bf f}_1 = & \sum_{l=1}^{4}
(\bU_1^{(1, 2)})_{l}\otimes (\bU_1^{(3, 4)})_{l},
\end{eqnarray}
where we introduce the notation of ${P}_{\mathcal{N}}$ as the rescaled orthogonal projection onto ${\mathcal{N}}$. 
Let discrete macroscopic charge, current and kinetic energy densities of ${\bf f}$ be
\begin{align}
\left(\begin{array}{l}
{\boldsymbol\rho}\\
{\bf J}_1\\
{\bf J}_2\\
{\boldsymbol \kappa} 
\end{array}
\right )
&= \sum_{l_{12}} \sum_{l_{34}} \bB^{(1,2, 3, 4)}_{l_{12},l_{34}}
 \left 
 \langle \bU_{l_{34}}^{(3, 4)}, 
 \left(\begin{array}{l}
{\bf 1}_{v_1 \otimes v_2} \\
{\bf v}_1\otimes {\bf 1}_{v_2}\\
{\bf 1}_{v_1}\otimes{\bf v}_2\\
\frac12{\bf v}_1^2\otimes {\bf 1}_{v_2}+\frac12{\bf 1}_{v_1}\otimes{\bf v}_2^2
\end{array}
\right )
\right \rangle
\bU_{l_{12}}^{(1, 2)}.
\label{eq:rho_j_kappa_2d2v}
\end{align}
The specifications of the frame vectors and transfer tensors of ${\bf f}_1$ are outlined below.
\bit
\item $(\bU_1^{(3, 4)})_k$ 
in \eqref{eq: U1_4} is constructed with the frame vectors for nodes $\{3\}$ and $\{4\}$ as \eqref{eq: U3U4} and \eqref{eq: U3U4_b} and the transfer tensor $\bB_1^{(3, 4)}$ from \eqref{eq:B34f}. 
\item 
$(\bU_1^{(1, 2)})_{k}$, $k=1, \cdots, 4$, are given as
\beq
\label{eq: U12_htd_b}
(\bU_1^{(1, 2)})_{1} =\frac{1}{c_1^2} {\boldsymbol\rho}, \quad
(\bU_1^{(1, 2)})_{2} =\frac{1}{c_1 c_2} {\bf J}_1, \quad
(\bU_1^{(1, 2)})_{3} =\frac{1}{c_1 c_2}  {\bf J}_2, \quad
(\bU_1^{(1, 2)})_{4} =\frac{\sqrt{2}}{c_1c_3}({\boldsymbol\kappa} - c\boldsymbol\rho). 
\eeq 
\eit
\end{prop}
\begin{proof} The construction of ${\bf f}_1$ in \eqref{eq:htd_f1_nested} comes from first constructing the orthonormal basis of 
$ \mathcal{N}$ from \eqref{eq: 2d2v_subspa}, followed by rescaling for the $(\bU_1^{(3, 4)})_k$, $k=1, \cdots, 4$. 
\eqref{eq: U12_htd_b} comes from obtaining the coefficients from the weighted orthogonal projection onto $\mathcal{N}$. 
\end{proof}

Now we are ready to introduce the LoMaC low-rank tensor 2D2V algorithm:  
\begin{enumerate}[Step 1:]
\item The 2D2V low rank algorithm presented in \cite{guo2022low}. In particular, starting from the solution ${\bf f}^n$ in the low rank format \eqref{eq:htd_f_nested_0}, we can add basis from a step of time integration (e.g. second order multi-step method) to obtain the intermediate solution ${\bf f}^{n+1, \star}$ in the same low rank format but with higher rank. 
\item Update macroscopic solutions $\boldsymbol\rho^{n+1, M}$, $\bJ_1^{n+1, M}$, $\bJ_2^{n+1, M}$, $\boldsymbol\kappa^{n+1, M}$ by using a conservative high order finite difference scheme with KFVS, coupled with the corresponding time integrator, to solve the macroscopic conservation laws \eqref{eq:mass}-\eqref{eq:ener} in a 2D setting. 
\item Perform the decomposition of ${\bf f}^{n+1, \star}$ obtained from Step 1 as ${\bf f}^{n+1, \star} = P_{\mathcal{N}} ({\bf f}) + (I - P_{\mathcal{N}})({\bf f}) \doteq {\bf f}_1 + {\bf f}_2$. 
\begin{enumerate}
\item Replace ${\bf f}_1$ with a new ${\bf f}_1^M$ via \eqref{eq: U12_htd_b}, but with $\boldsymbol\rho^{n+1, M}$, $\bJ_1^{n+1, M}$, $\bJ_2^{n+1, M}$ obtained from {Step 2} for local conservation of macroscopic variables. 
\item Perform a weighted SVD truncation to ${\bf f}_2$, followed with a projection operator $(I-P_{\mathcal{N}})$ to ensure zero mass momentum and kinetic energy. That is, to compute $(I-P_{\mathcal{N}}) (\sqrt{\bf w}\star \mathcal{T}_\varepsilon (\frac{1}{\sqrt{\bf w}}\star{\bf f}_2))$. 
\item The updated ${\bf f}^{n+1} = {\bf f}_1^M + (I-P_{\mathcal{N}}) (\sqrt{\bf w} \star \mathcal{T}_\varepsilon (\frac{1}{\sqrt{\bf w}}\star{\bf f}_2))$ from previous two sub-steps. 
\end{enumerate} 
\end{enumerate}
Similar to the 1D1V case, the proposed algorithm enjoys the local conservation in macroscopic mass, momentum and energy. There are two crucial ingredients in the LoMaC algorithm. On one hand computational efficiency is realized by the low rank representation of the solution, mitigating the curse of dimensionality. On the other hand, we simultaneously evolve the macroscopic conservation laws by using kinetic fluxes in a local conservative manner during each time step; we then project the low rank kinetic solution onto a subspace with conservation on macroscopic mass, momentum and energy. 
We summarize the flowchart as the following Algorithm \ref{alg: low_rank_lomac_2D2V}. 

 \bigskip
\begin{algorithm}[H]
\label{alg: low_rank_lomac_2D2V}
  \caption{The LoMaC low rank tensor algorithm for the 2D2V VP system.}
    \begin{enumerate}
\item Initialization:
  \begin{enumerate}
  \item Initial distribution function $f(x_1, x_2, v_1, v_2, t=0)$ in a low rank format \eqref{eq:htd_f_nested_0}-\eqref{eq: U34_htd}. 
 \end{enumerate}
\item For each time step evolution from $t^n$ to $t^{n+1}$: update ${\bf f}^{n+1}$ from ${\bf f}^n$ in the low rank format.
  \begin{enumerate}
  \item Compute the charge density $\boldsymbol \rho^{n}$ in the $(x_1, x_2)$ full grid format, followed by computing $\bE^n = (E_1^n, E_2^n)$ from the Poisson equation \eqref{poisson} by fast Fourier transform or a high order finite difference algorithm.
  \item Update ${\bf f}^{n+1,*}$, by adding basis according to the dimension tree $\mathcal{T}$ as shown in Figure \ref{fig:dimtree1}. The procedure is similar to that outlined in \cite{guo2022low}. 
  \item  
  Compute ${\boldsymbol \rho}^{n+1, *}$,  ${\bf J}_1^{n+1, *}$, ${\bf J}_2^{n+1, *}$, ${\boldsymbol \kappa}^{n+1, *}$ by low rank numerical integration in velocity, i.e. \eqref{eq:rho_j_kappa}. Compute ${\bf f}_1$ from \eqref{eq: U12_htd_b} with $\boldsymbol\rho^{n+1, *}$, ${\bf J}_1^{n+1, *}$, ${\bf J}_2^{n+1, *}$, $\boldsymbol\kappa^{n+1, *}$.
  \item  Compute ${\bf f}_2\doteq {\bf f} - {\bf f}_1$ and perform a weighted SVD truncation on ${\bf f}_2$ in the low rank 2D2V format \cite{guo2022conservative} to obtain
  $
\left( \sqrt{\bf w}  \star \mathcal{T}_\varepsilon (\frac{1}{\sqrt{\bf w}} \star {\bf f}_2)\right).
  $
Finally, we apply the $(I-P_\mathcal{N})$ operator to $\left( \sqrt{\bf w}  \star \mathcal{T}_\varepsilon (\frac{1}{\sqrt{\bf w}} \star {\bf f}_2)\right)$ to ensure its zero mass, momentum and kinetic energy after truncation.
  \item Compute ${\bf f}^M_1$.
  \begin{enumerate}
  \item Update macroscopic mass, momentum and energy density, $\boldsymbol\rho^{n+1, M}$, $\bJ_1^{n+1, M}$, $\bJ_2^{n+1, M}$, ${\boldsymbol e}^{n+1, M}$, using the kinetic flux vector splitting, in a flux-difference form using the same second order SSP multi-step method in Step 2(b). 
  \item Compute $\bE^{n+1}$ from $\boldsymbol\rho^{n+1, M}$ by Poisson solver as in Step 2(a). 
  \item Compute $\boldsymbol\kappa^{n+1, M}$ by subtracting energy from the electrostatic field from ${\boldsymbol e}^{n+1, M}$. 
\item  Construct ${\bf f}^M_1$ from \eqref{eq: U12_htd_b}, but with $\boldsymbol\rho^{n+1, M}$, $\bJ_1^{n+1, M}$, $\bJ_2^{n+1, M}$, $\boldsymbol\kappa^{n+1, M}$. 
\end{enumerate}
  \item Update the compressed low-rank solution via \eqref{eq: Tc}, 
  $${\bf f}^{n+1} \doteq T^M_c({\bf f}) = {\bf f}^M_1 +(I-P_\mathcal{N})\left(\sqrt{\bf w}  \star \mathcal{T}_\varepsilon (\frac{1}{\sqrt{\bf w}} \star {\bf f}_2)\right).$$
\end{enumerate}
\end{enumerate}
  \end{algorithm}

%% file: numerical_energy.tex
\section{Numerical results}\label{sec:numerical}
In this section we present a collection of numerical examples to demonstrate the efficacy of the proposed LoMaC low rank tensor methods for simulating the VP system.  
In particular, besides the efficiency gain from the low rank representation of the solution shown in our previous work \cite{guo2022conservative}, we verify numerically the ability of the proposed method to conserve the total mass, momentum and energy up to the machine precision.  


\subsection{1D1V Vlasov-Poisson system}
\begin{exa}\label{ex:forced}(A forced VP system \cite{de2012high}.) In this example, we consider the VP system with a forcing term and periodic conditions in $x$- direction $x\in[-\pi,\pi]$
\begin{align*}
\frac{\partial f}{\partial t} + vf_x + Ef_v &= \psi(x,v,t),\\
E(x,t)_x &= \rho(x,t) - \sqrt{\pi},
\end{align*}
where $\psi$ is defined as
$$
\psi(x, v, t)=\left(\left(\left(4 \sqrt{\pi}+2\right) v-\left(2 \pi+\sqrt{\pi}\right)\right) \sin (2 x-2 \pi t)
+\sqrt{\pi}(\frac14-v) \sin (4 x-4 \pi t)\right)\exp\left(-\frac{(4 v-1)^{2}}{4}\right) 
$$
so that the system has the exact solution
$$
\begin{aligned}
f(x, v, t) &=\left(2-\cos (2 x-2 \pi t)\right) \exp\left(-\frac{(4 v-1)^{2}}{4}\right), \\
E(x, t) &=\frac{\sqrt{\pi}}{4} \sin (2 x-2 \pi t).
\end{aligned}
$$
Note that the forced system satisfies the following the macroscopic system
\begin{align*}
\partial_{t} \rho + \bJ_x &= \frac{\sqrt{\pi}}{4}(1-4\pi)\sin(2x-2\pi t)\\
\partial_{t} \bJ +\mathbf{\sigma}_x&= \rho E + \frac{\sqrt{\pi}}{16}(3+ 4\sqrt{\pi} -4 \pi)\sin(2x-2\pi t) -\frac{\pi}{16}\sin(4x-4\pi t)\\
\partial_{t} e + \mathbf{Q}_x& = \frac{\sqrt{\pi}}{128}(7 + 8\sqrt{\pi}-12\pi)\sin(2x-2\pi t)-\frac{\pi}{64}\sin(4x-4\pi t) \\
& + \frac{\sqrt{\pi}}{8}\left( 2  - (1-4\pi)\cos(2x-2\pi t) \right)E,
\end{align*}
and conserves the total mass, total momentum, and total energy. Moreover, the exact solution is known and remains rank one over time. Hence, we will make use of this example to demonstrate the accuracy, efficiency as well as the ability of the proposed LoMaC low rank method to conserve the physical invariants. In the simulation, we set the truncation threshold $\varepsilon=10^{-4}$ and set the computational domain $[-\pi,\pi]\times[-L_v,L_v]$ with $L_v=4$. We compute the problem with one period to $t=1$ and summarize the convergence study in Table \ref{tb:forced}. Second order of convergence in the $L^\infty$ and $L^2$ errors is observed due to the second order SSP multi-step method used. In Figure \ref{fig:forced1d_invar}, we report the time histories numerical ranks of the low rank solutions, relative deviation of the total mass, total momentum and total energy. It is observed that the ranks of the numerical solutions stay four over time, and it is because $\mathbf{f}_1$ is of rank three to conserve locally the mass, momentum and kinetic energy densities,  and the truncated $\widetilde{\mT_\varepsilon}({\bf f}_2)$ stays rank one. Furthermore, the total mass, momentum and energy are conserved up to the machine precision. 

\begin{table}[!hbp]
	\centering
	\caption{Example \ref{ex:forced}. $t=1$. Convergence study.}
	\label{tb:forced}
	\begin{tabular}{|c|c|c|c|c|}
		\hline
		 $N$ & $L^\infty$ error & order & $L^2$ error & order \\\hline
32	&	3.39E-03&	 --      &2.28E-03& --		\\\hline
64	&	4.07E-04&	3.06&	2.97E-04&	2.94	\\\hline
128	&	9.83E-05&	2.05&	7.13E-05&	2.06	\\\hline
256	&	2.46E-05&	2.00&	1.85E-05&	1.95	\\\hline
	\end{tabular}
\end{table}

\begin{figure}[h!]
	\centering
		\subfigure[]{\includegraphics[height=50mm]{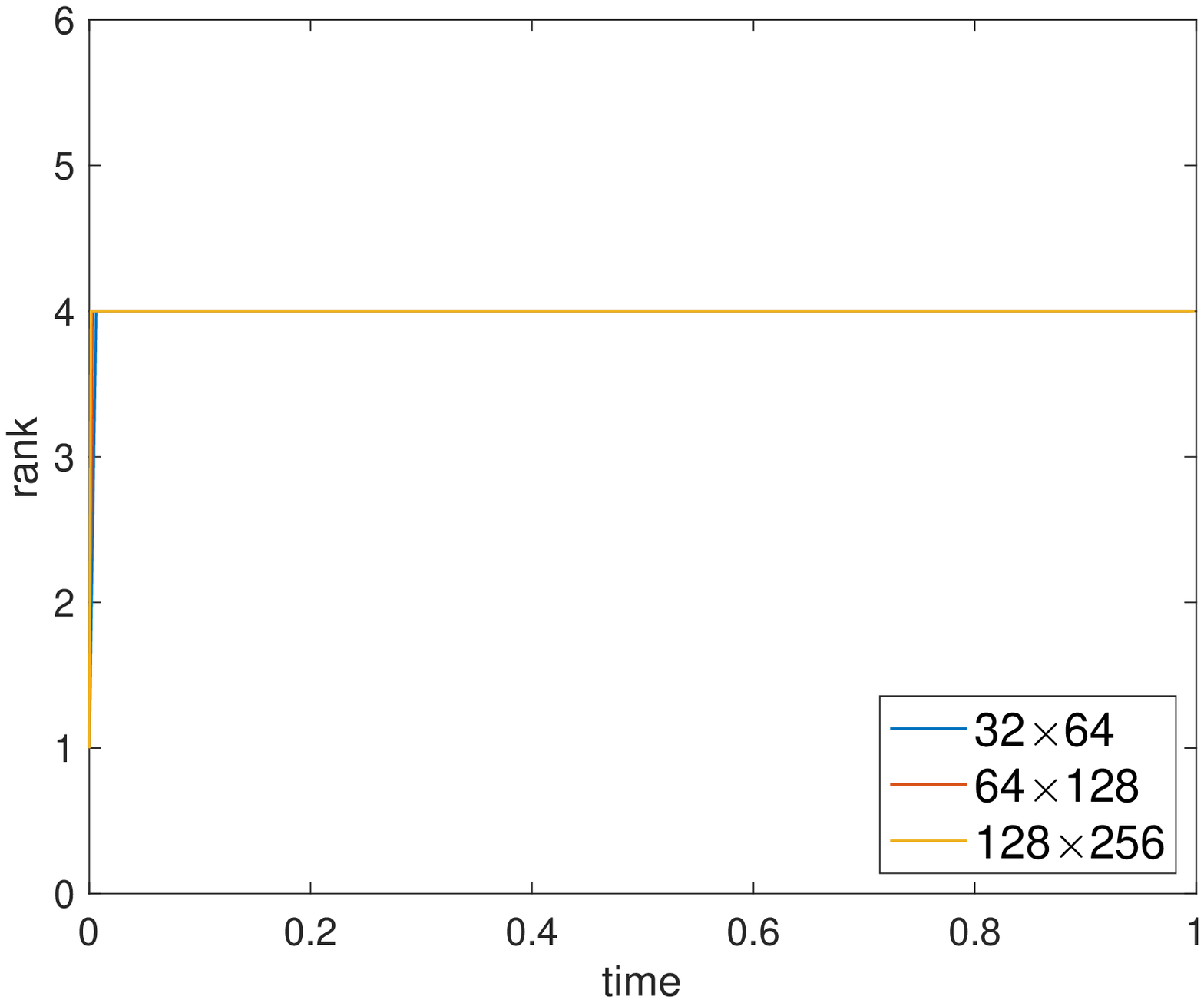}}
		\subfigure[]{\includegraphics[height=50mm]{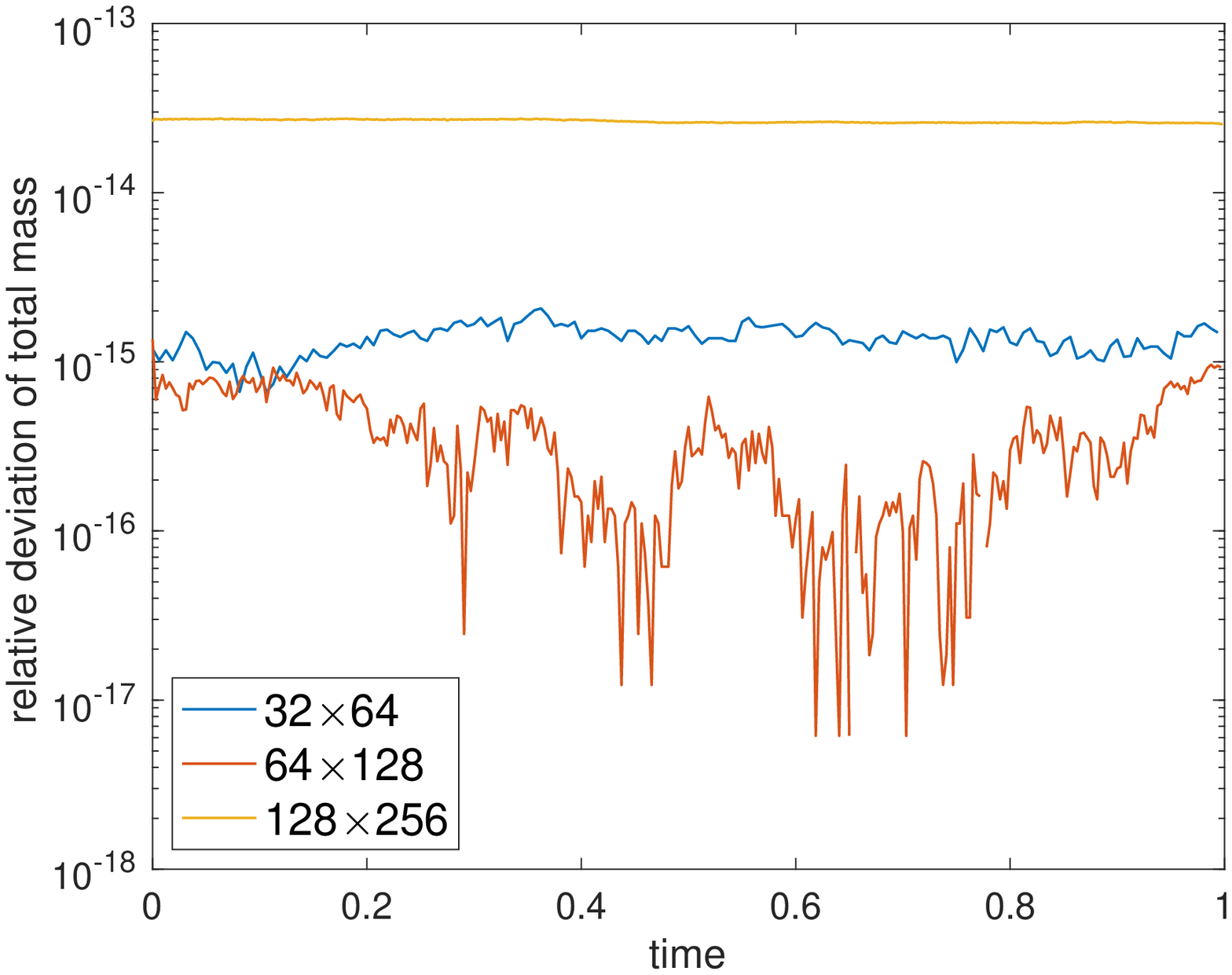}}
	       \subfigure[]{\includegraphics[height=50mm]{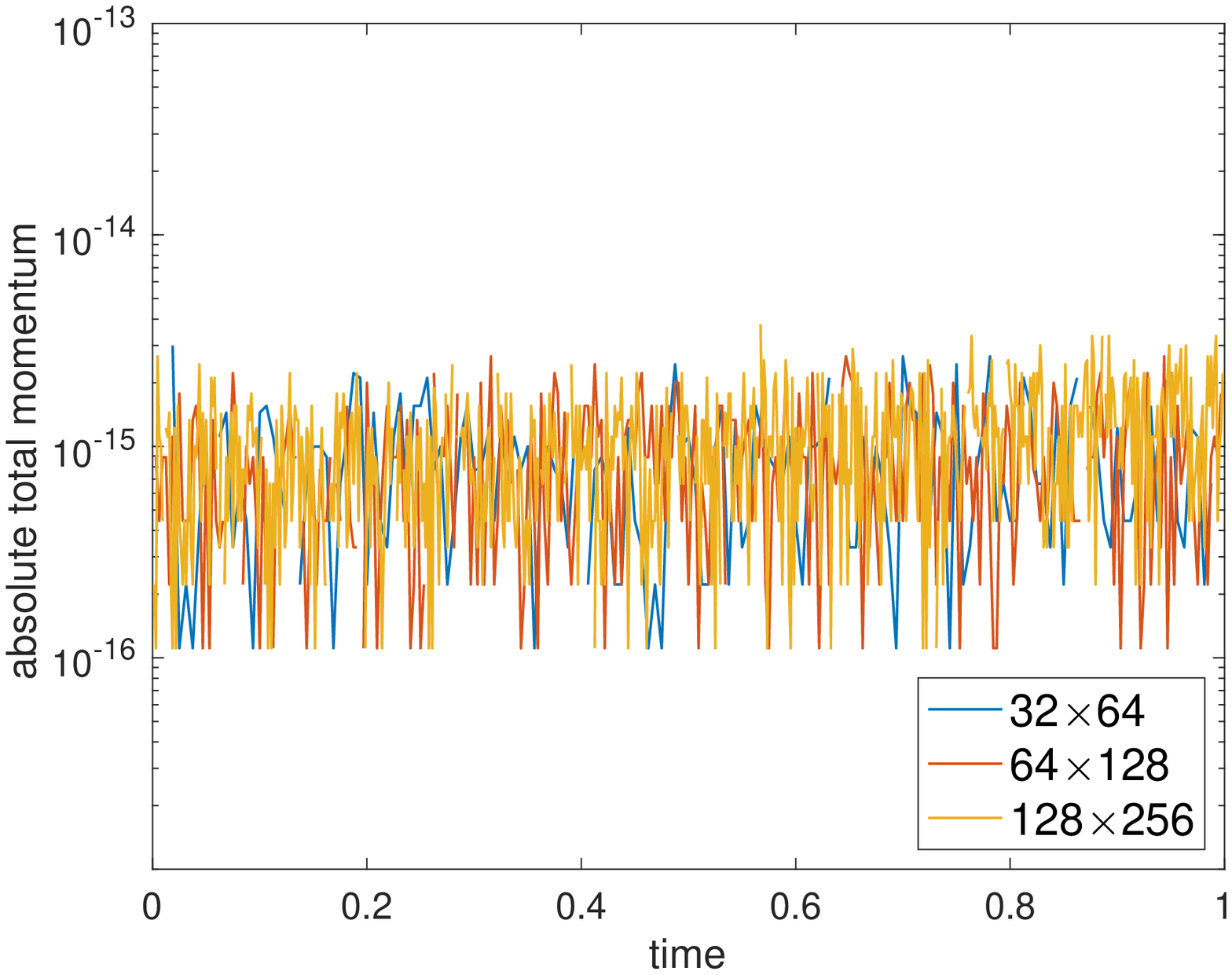}}
		\subfigure[]{\includegraphics[height=50mm]{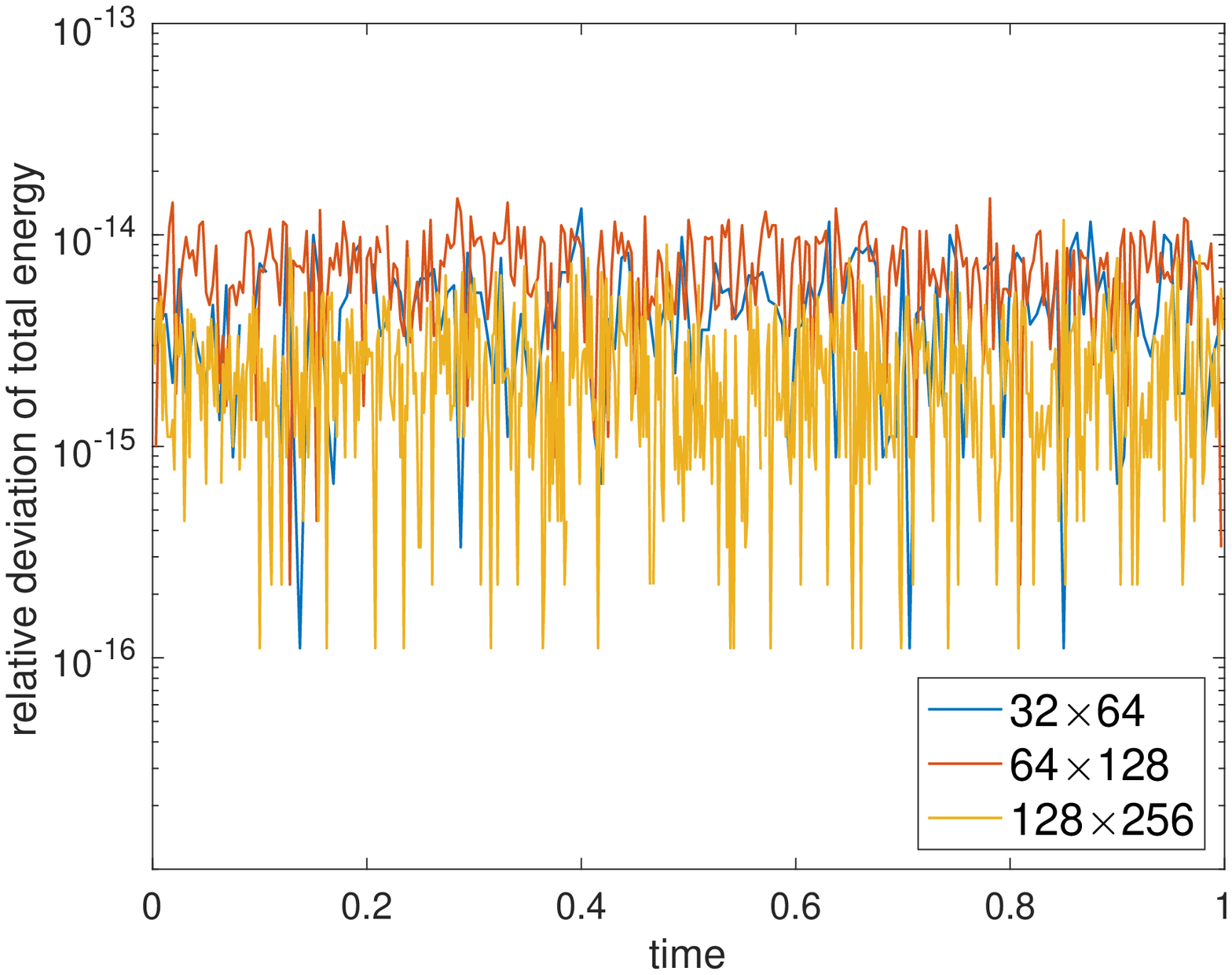}}
	\caption{Example \ref{ex:forced}.  The time evolution of  ranks of the numerical solutions (a), relative deviation of total mass (b),  total momentum (c), and total energy (d). $\varepsilon=10^{-4}$.}
	\label{fig:forced1d_invar}
\end{figure}	

\end{exa}

\begin{exa}
	\label{ex:weak1d}(Weak Landau damping.) We consider the weak Landau damping test with initial condition 
	\begin{equation}
		\label{eq:landau1d}
		f(x,v,t=0) =\frac{1}{\sqrt{2 \pi}} \left(1+\alpha  \cos \left(k x\right)\right)\exp\left(-\frac{v^2}{2}\right),
	\end{equation}
	where $\alpha=0.01$ and $k=0.5$. The computational domain is set to be $[0,L_x]\times[-L_v,L_v]$ with $L_x=2\pi/k$ and $L_v=6$. We set $\varepsilon=10^{-5}$ for truncation. In Figure \ref{fig:weak1d_invar}, we report the simulation results from the proposed LoMaC low rank method, including the time histories of the electric energy, numerical ranks of the low rank solutions, relative deviation of the total mass, momentum and energy. It is observed that the method is able to predict the correct damping rate of the electric energy.  Furthermore, the method is able to conserve the total mass, momentum and energy up to the machine precision regardless of the mesh size used.

\begin{figure}[h!]
	\centering
	\subfigure[]{\includegraphics[height=40mm]{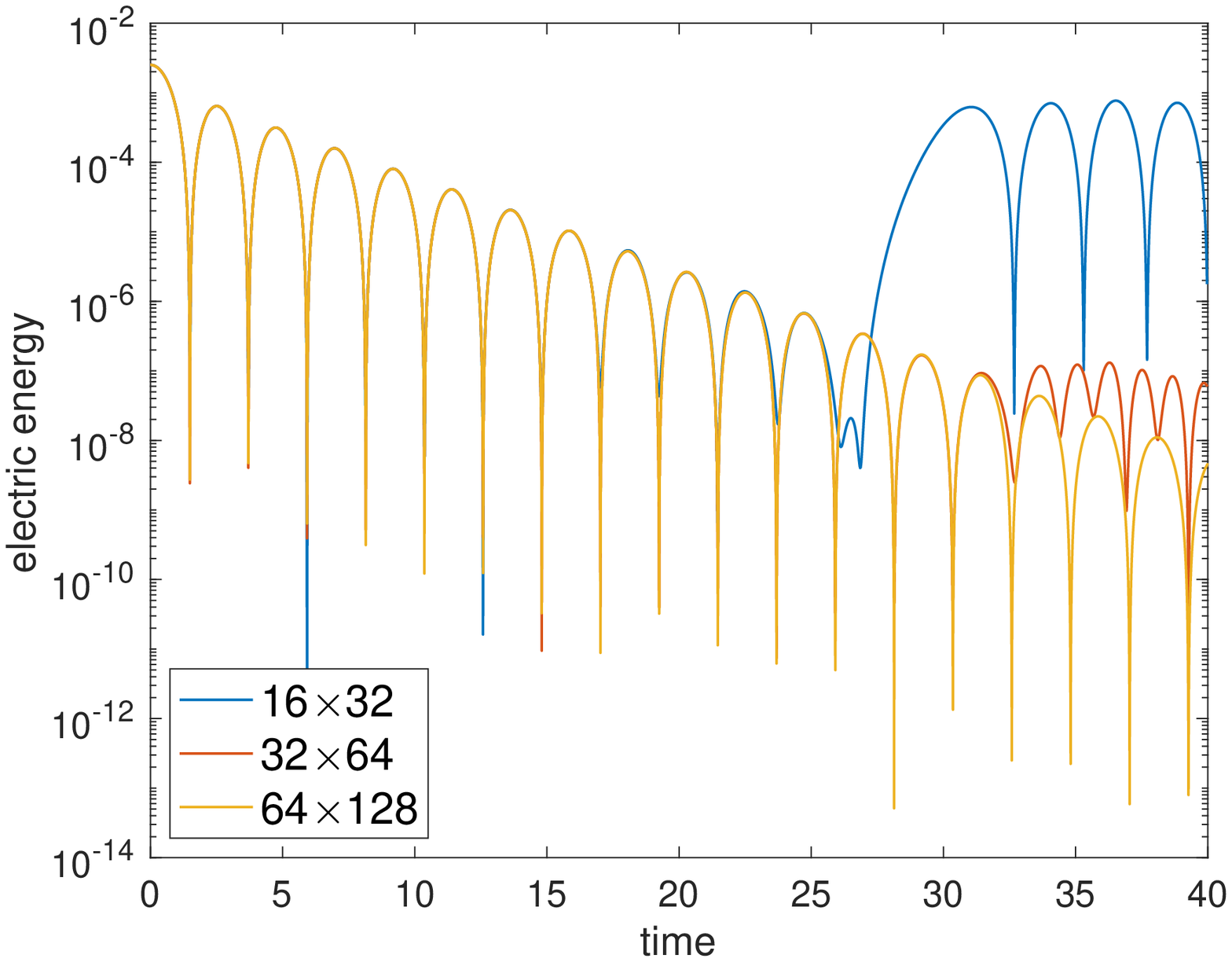}}
		\subfigure[]{\includegraphics[height=40mm]{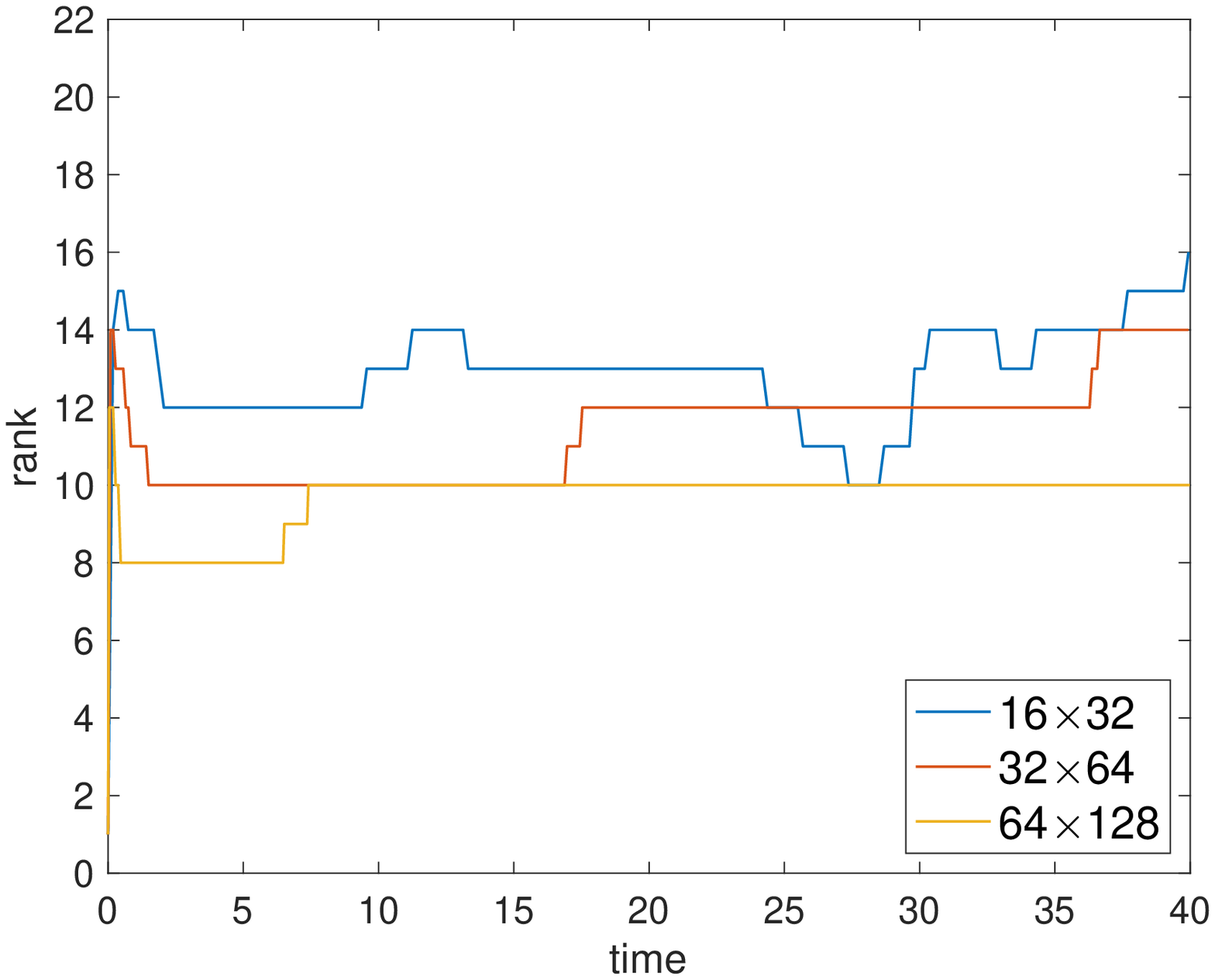}}
		\subfigure[]{\includegraphics[height=40mm]{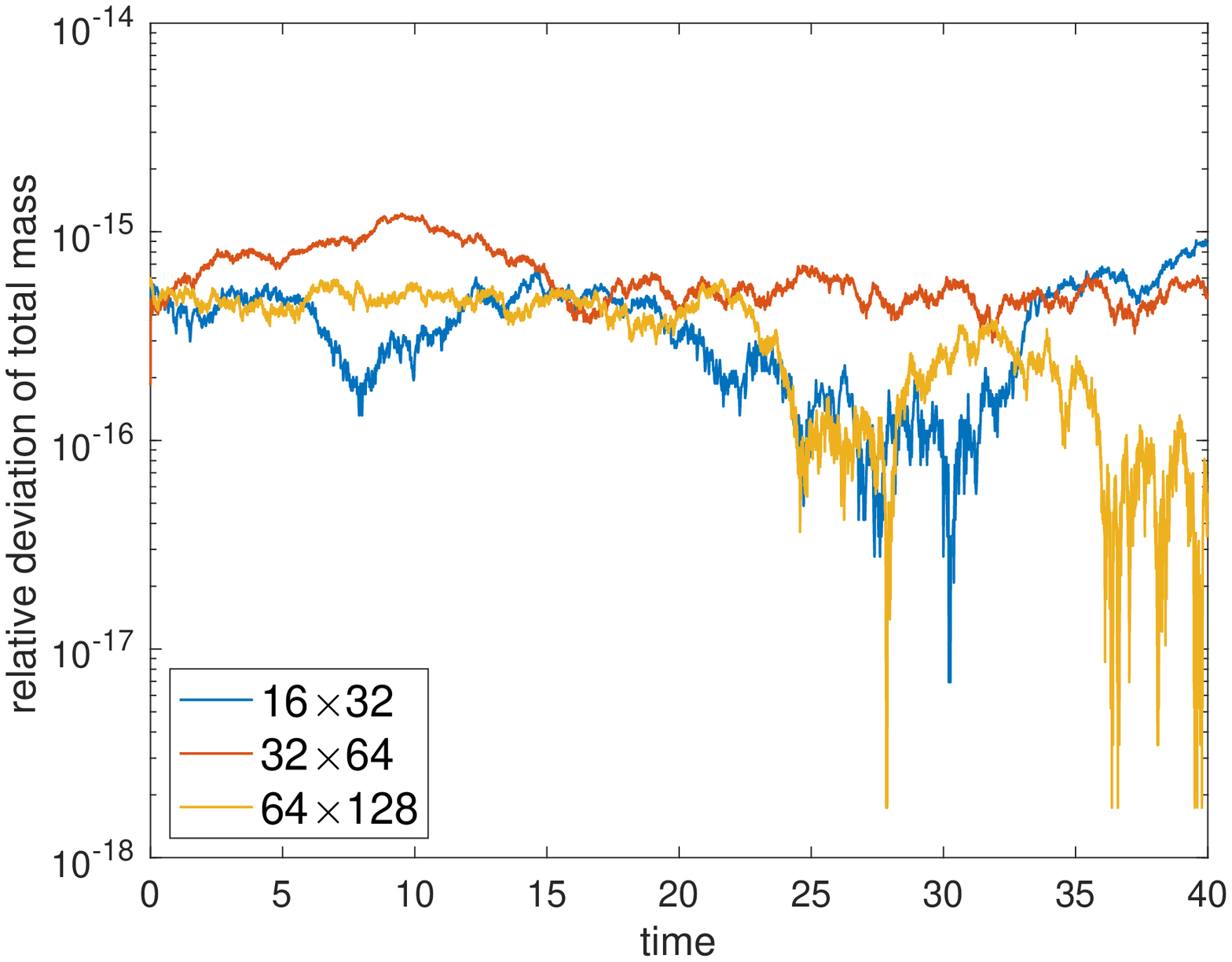}}
	       \subfigure[]{\includegraphics[height=40mm]{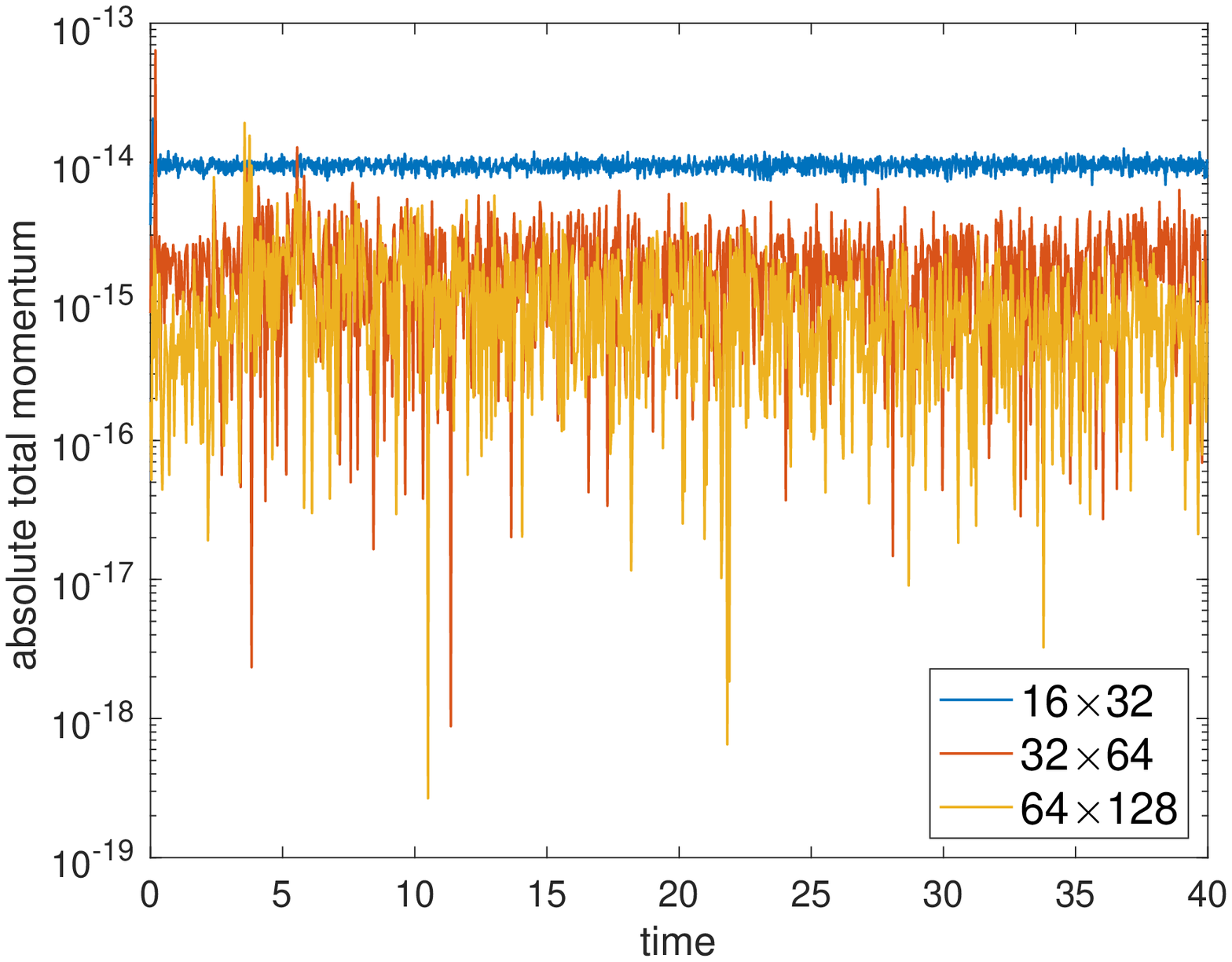}}
		\subfigure[]{\includegraphics[height=40mm]{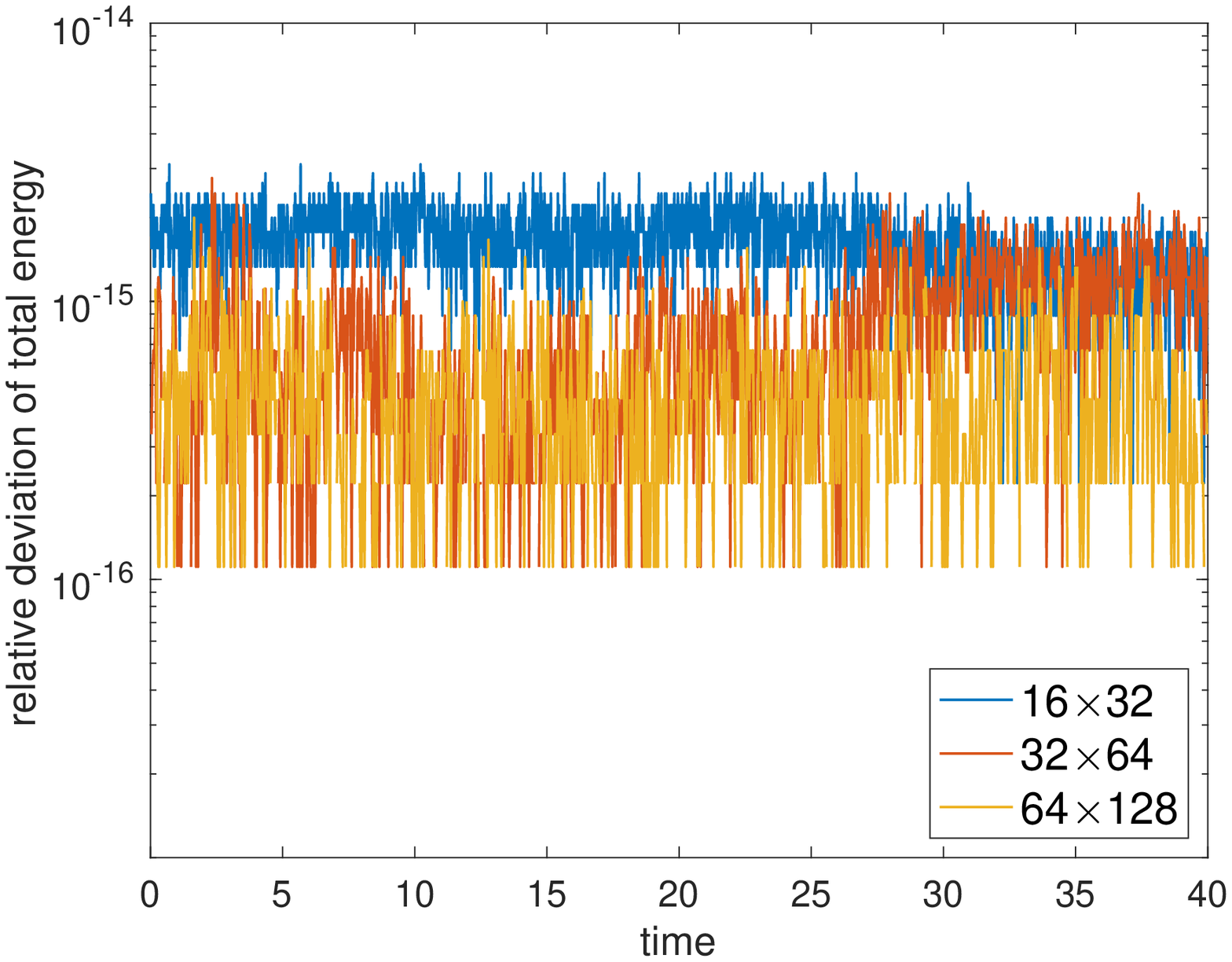}}
	\caption{Example \ref{ex:weak1d}.  The time evolution of  electric energy (a), ranks of the numerical solutions (b), relative deviation of total mass (c), absolute total momentum (e), and relative deviation of total energy (f). $\varepsilon=10^{-5}$.}
	\label{fig:weak1d_invar}
\end{figure}	
	
\end{exa}

\begin{exa}
	\label{ex:strong1d}(Strong Landau damping.) We consider the strong Landau damping test, the initial condition of which is the same as \eqref{eq:landau1d} but with parameters
$\alpha=0.5$ and $k=0.5$. The computational domain is set to be $[0,L_x]\times[-L_v,L_v]$ with $L_x=2\pi/k$ and $L_v=6$, and the truncation threshold is set to be $\varepsilon=10^{-3}$. Unlike the weak case, the dynamics of strong Landau damping cannot be predicted by the linear theory, as the nonlinear effect will dominate due to the large perturbation. We summarize the simulation results in Figure \ref{fig:strong1d_invar}. It is observed that the proposed method is able to capture the dynamics of the electric energy and conserve the physical invariants as expected up to machine precision.

\begin{figure}[h!]
	\centering
	\subfigure[]{\includegraphics[height=40mm]{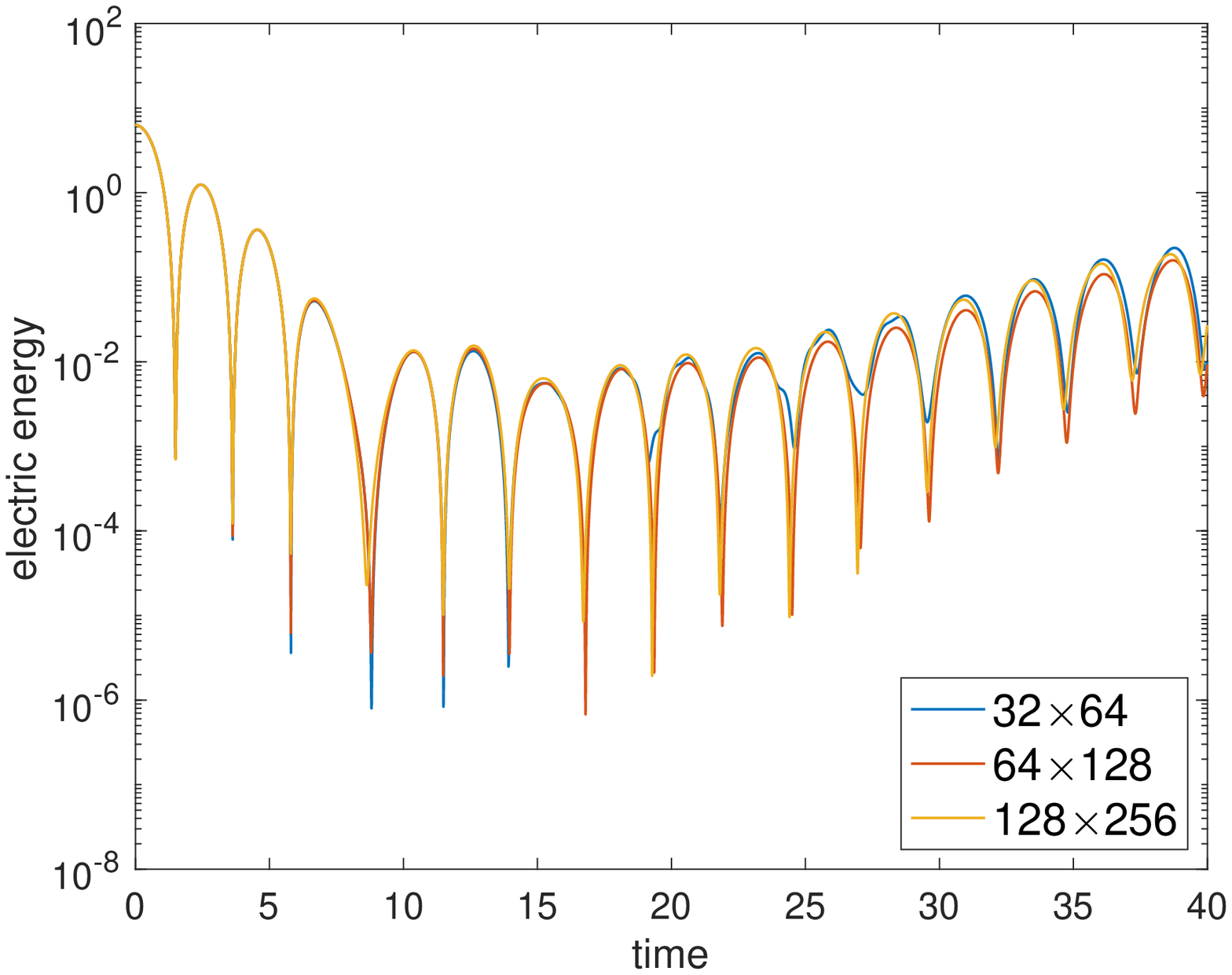}}
		\subfigure[]{\includegraphics[height=40mm]{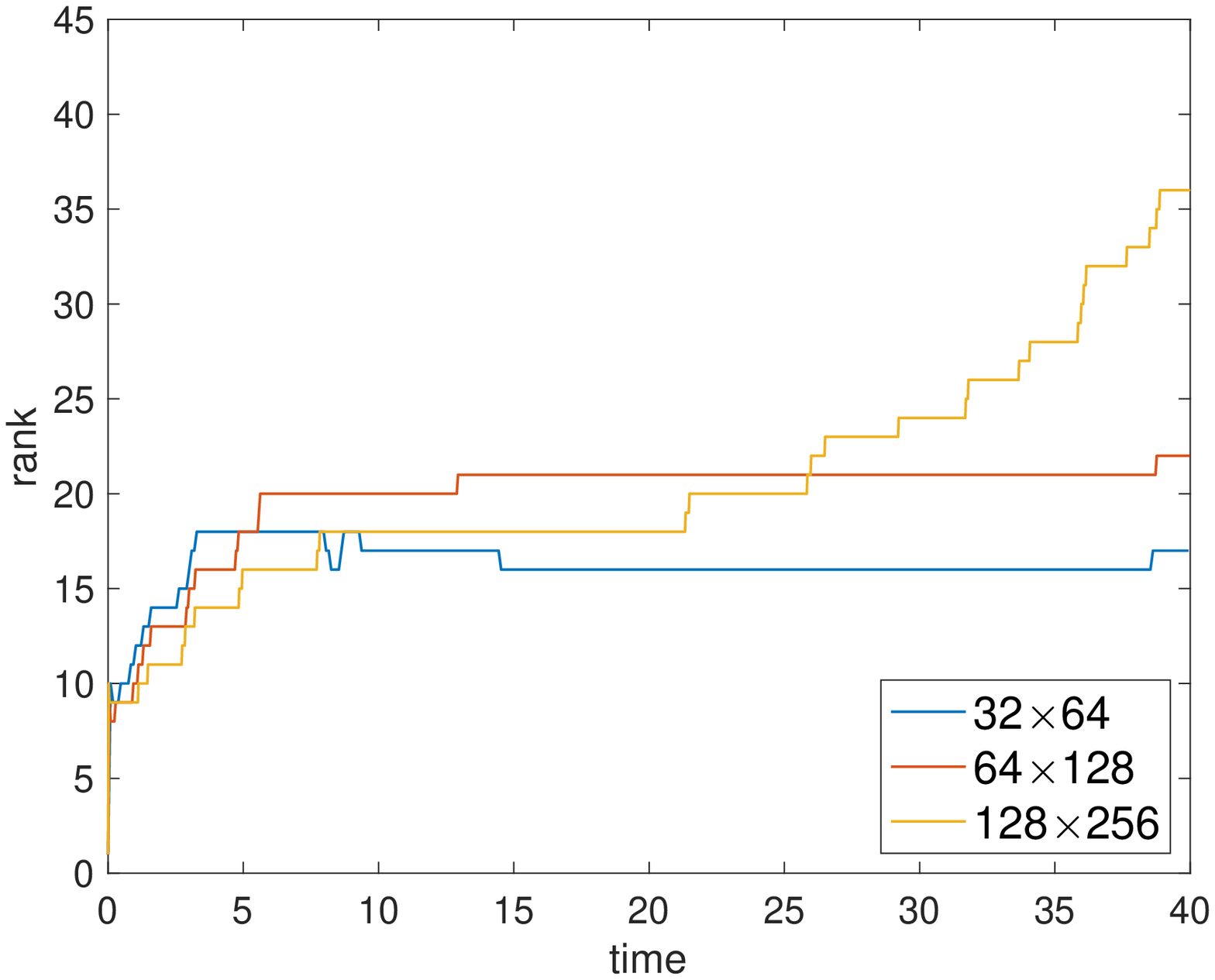}}
		\subfigure[]{\includegraphics[height=40mm]{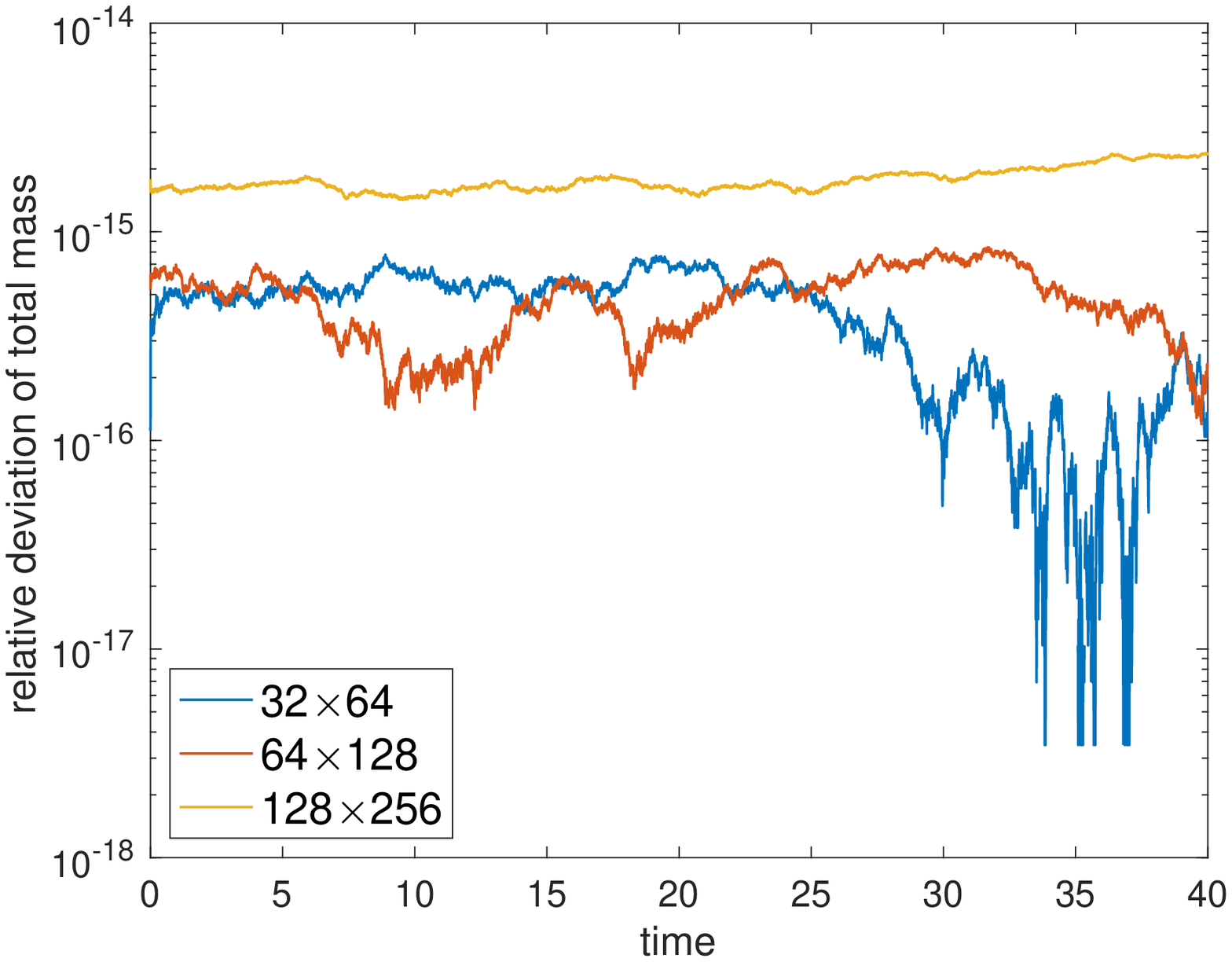}}
	       \subfigure[]{\includegraphics[height=40mm]{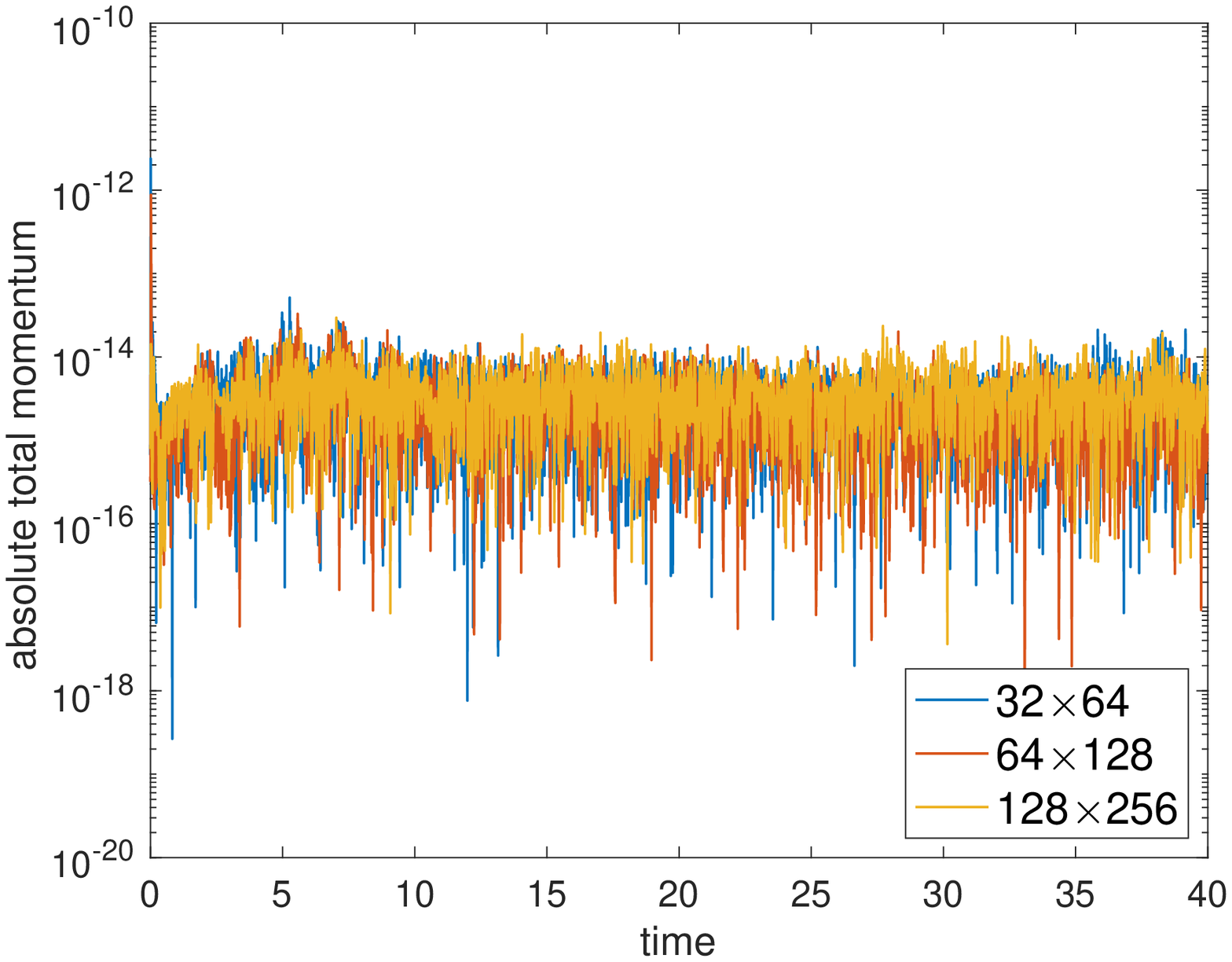}}
		\subfigure[]{\includegraphics[height=40mm]{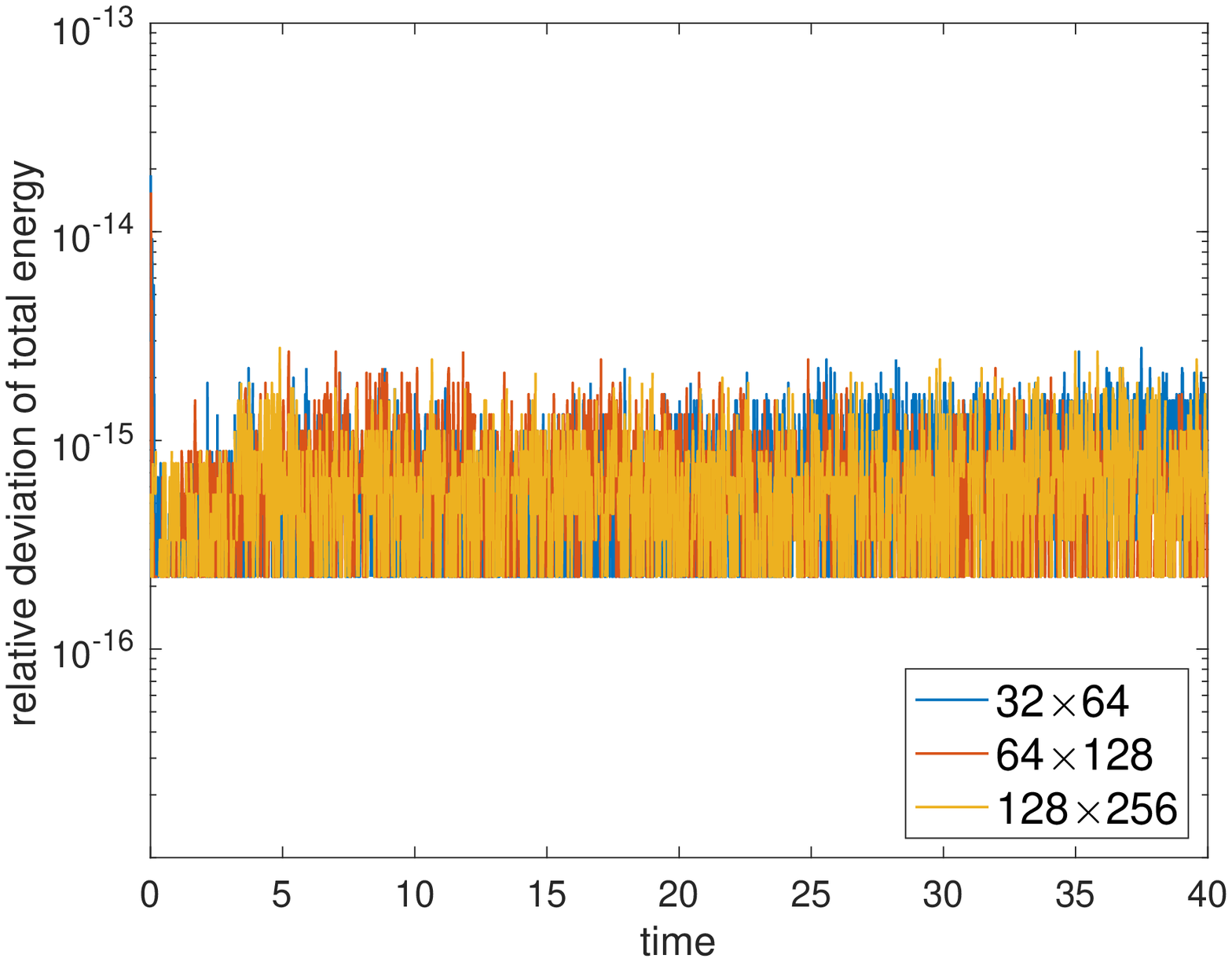}}
	\caption{Example \ref{ex:strong1d}.  The time evolution of  electric energy (a), ranks of the numerical solutions (b), relative deviation of total mass (c), absolute total momentum (e), and relative deviation of total energy (f).  $\varepsilon=10^{-3}$.}
	\label{fig:strong1d_invar}
\end{figure}	
	
\end{exa}

 \begin{exa}\label{ex:bumpontail} (Bump on tail.) In this example, we simulate the bump-on-tail test with the initial condition
 \begin{equation}
\label{eq:bump1d}
f(x,v,t=0) = \left(1+\alpha  \cos \left(k x\right)\right)\left(n_p\exp\left(-\frac{v^2}{2}\right) +n_b\exp\left(-\frac{(v-u)^2}{2v_{t}}\right) \right),
\end{equation}
where $\alpha=0.04$, $k=0.3$, $n_{p}=\frac{9}{10 \sqrt{2 \pi}}$, $n_{b}=\frac{2}{10 \sqrt{2 \pi}}$, $u=4.5$, $v_{t}=0.5$. The weight function $w(v) = \exp(-\frac{v^2}{3})$ is chosen. We compare the performance of the three low rank methods including the non-conservative method in \cite{guo2021lowrank}, the conservative method in \cite{guo2022lowrank}, and the proposed LoMaC low rank method. In the simulations, we set the mesh size as $N_x\times N_v=128\times256$ and the truncation threshold as $\varepsilon=10^{-4}$. The results are plotted in Figure \ref{fig:bump_com}. It is observed that all three methods generate numerical solutions with consistent electric energy evolution histories and comparable numerical ranks.  Furthermore, it is found that 
the non-conservative method can preserve the total mass, total momentum, and total energy up to the truncation threshold $\varepsilon=10^{-4}$, the conservative method can conserve the total mass and total momentum on the scale of $10^{-12}$ but not the total energy, and the proposed LoMaC low rank method can conserve the total mass, total momentum, and total energy on the scale of $10^{-14}$. In Figure \ref{fig:bumpontail_contour}, we report the contour plots of the solutions by the three methods. Meanwhile, we notice that results by the conservative method and the proposed method are more consistent, which is partly because of their excellent conservation properties.  

\begin{figure}[h!]
	\centering
	\subfigure[]{\includegraphics[height=40mm]{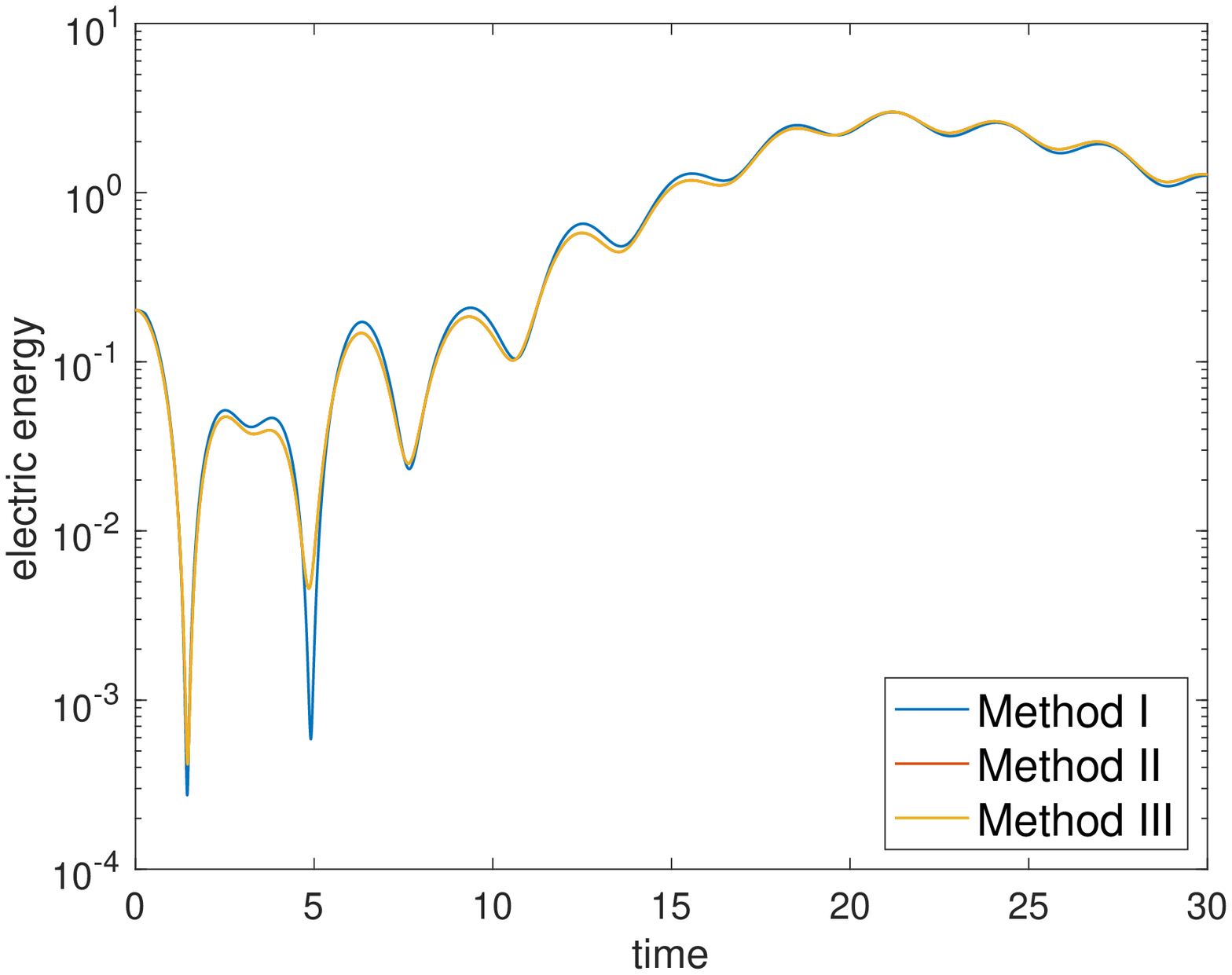}}
		\subfigure[]{\includegraphics[height=40mm]{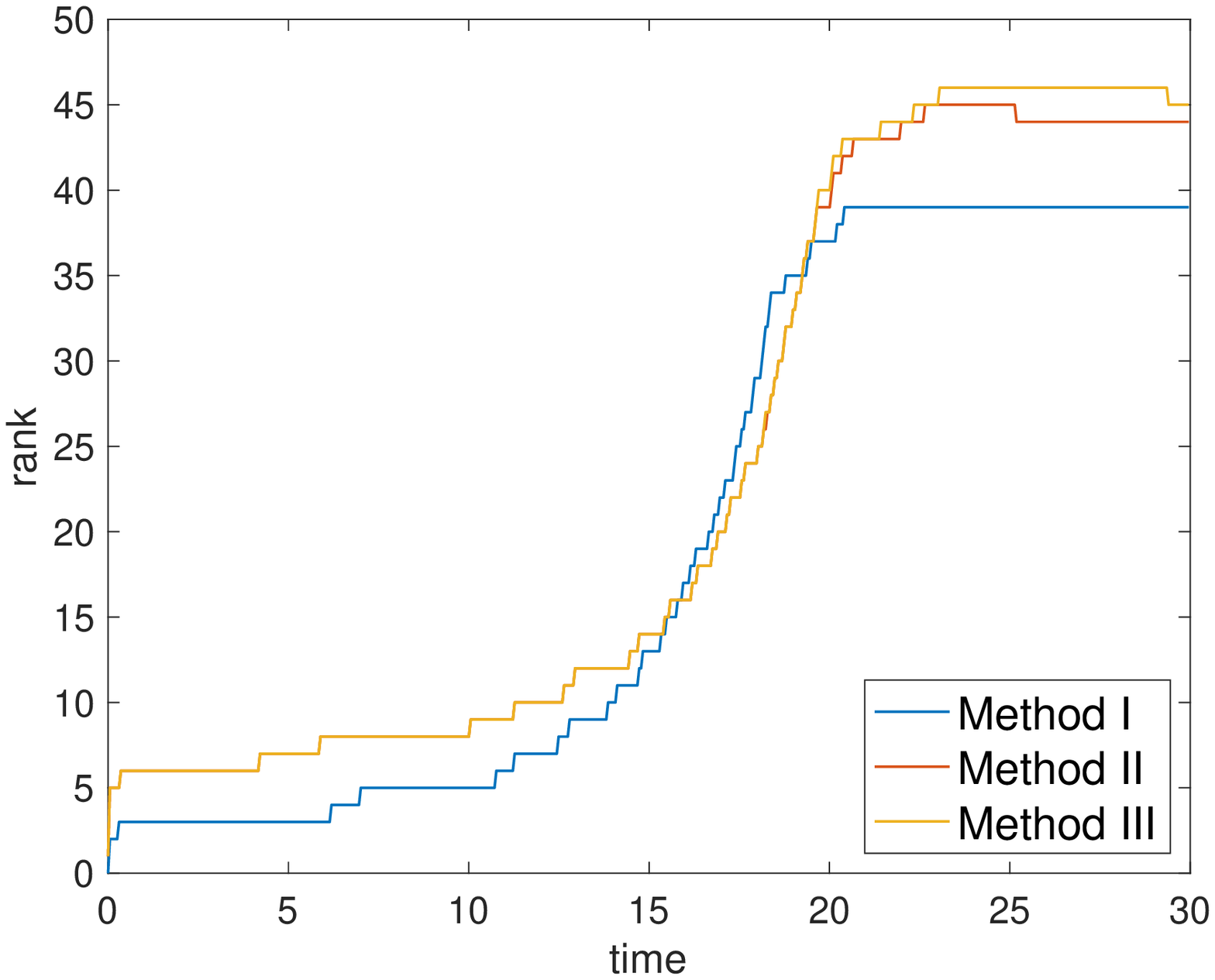}}
		\subfigure[]{\includegraphics[height=40mm]{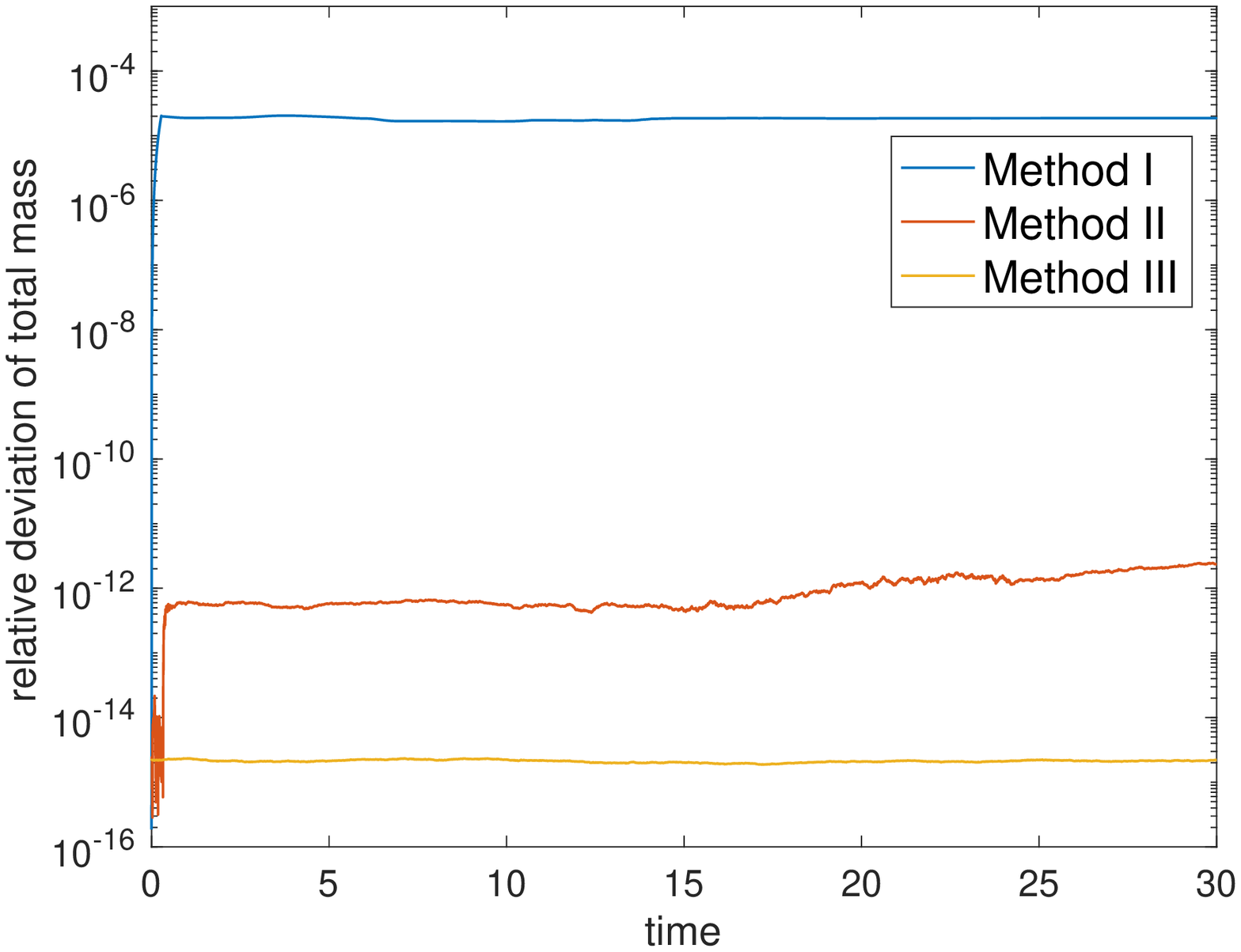}}
	       \subfigure[]{\includegraphics[height=40mm]{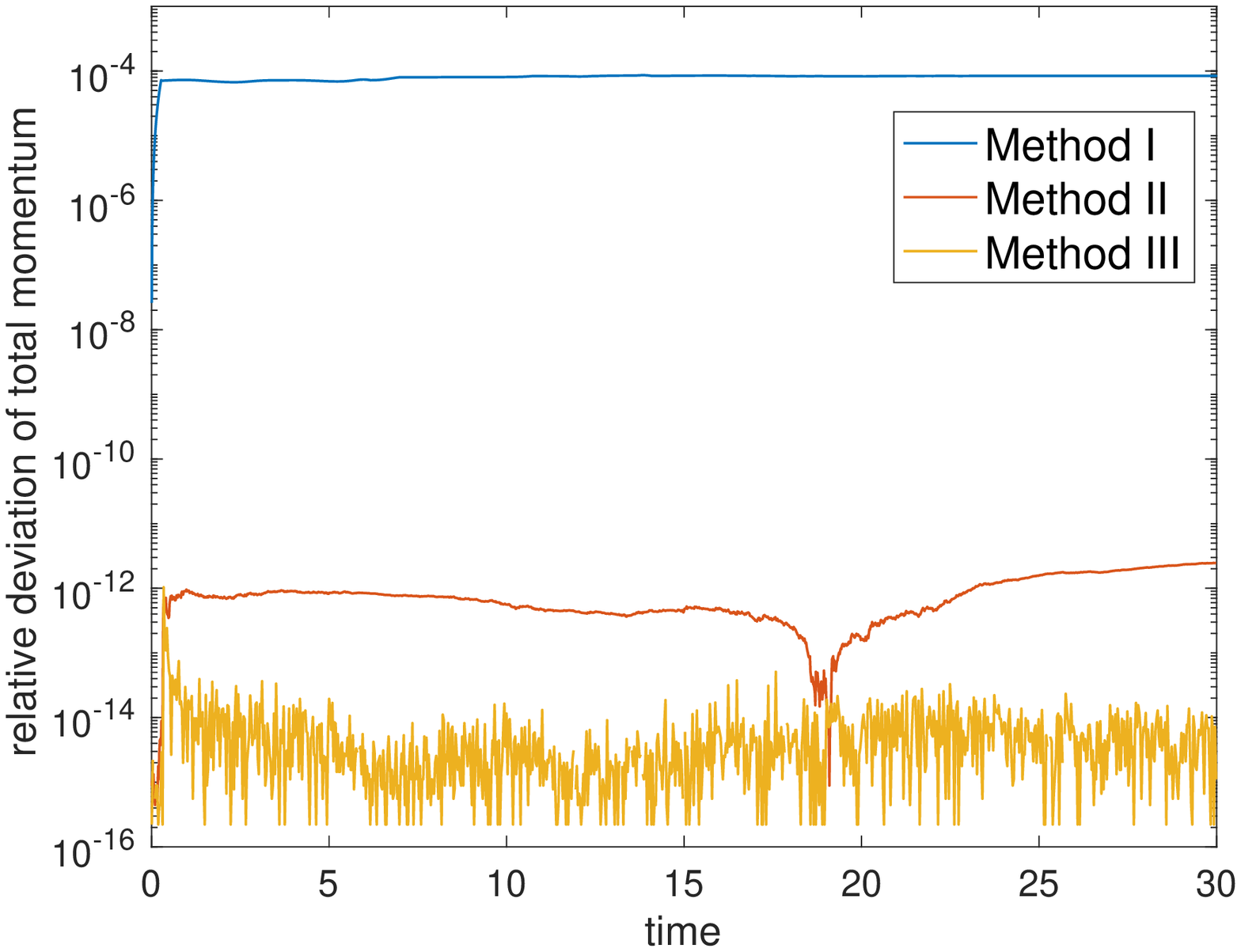}}
		\subfigure[]{\includegraphics[height=40mm]{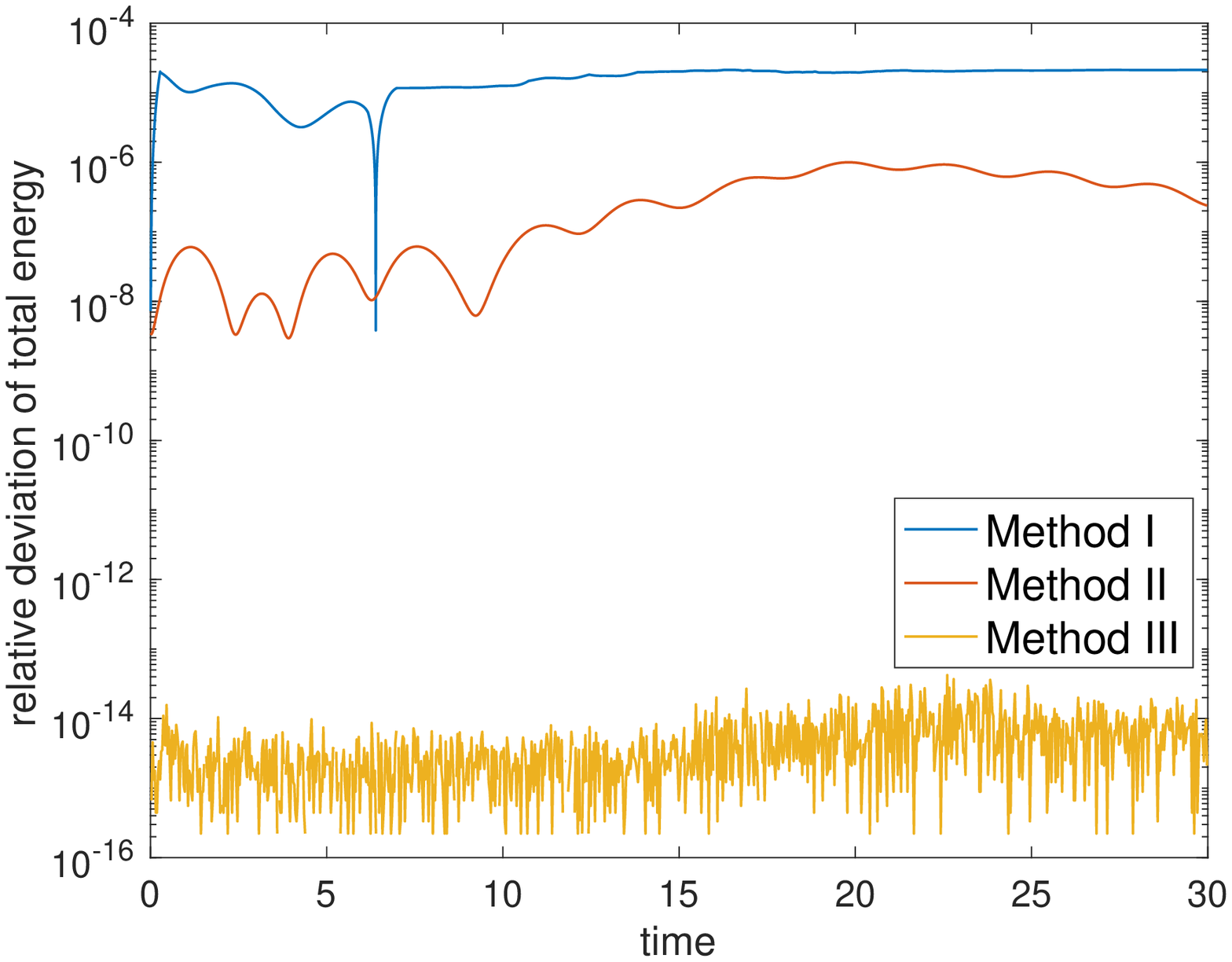}}
	\caption{Example \ref{ex:bumpontail}. Comparison of three low rank methods including the non-conservative low rank method in \cite{guo2021lowrank} denoted by method I, the conservative method in \cite{guo2022lowrank} denoted by method II and the proposed LoMaC low rank method denoted by method III.  The time evolution of  electric energy (a), ranks of the numerical solutions (b), relative deviation of total mass (c), absolute total momentum (e), and relative deviation of total energy (f). $N_x\times N_v = 128\times256$. $\varepsilon=10^{-4}$.}
	\label{fig:bump_com}
\end{figure}	

 \begin{figure}[h!]
	\centering
       \subfigure[Method I]{\includegraphics[height=40mm]{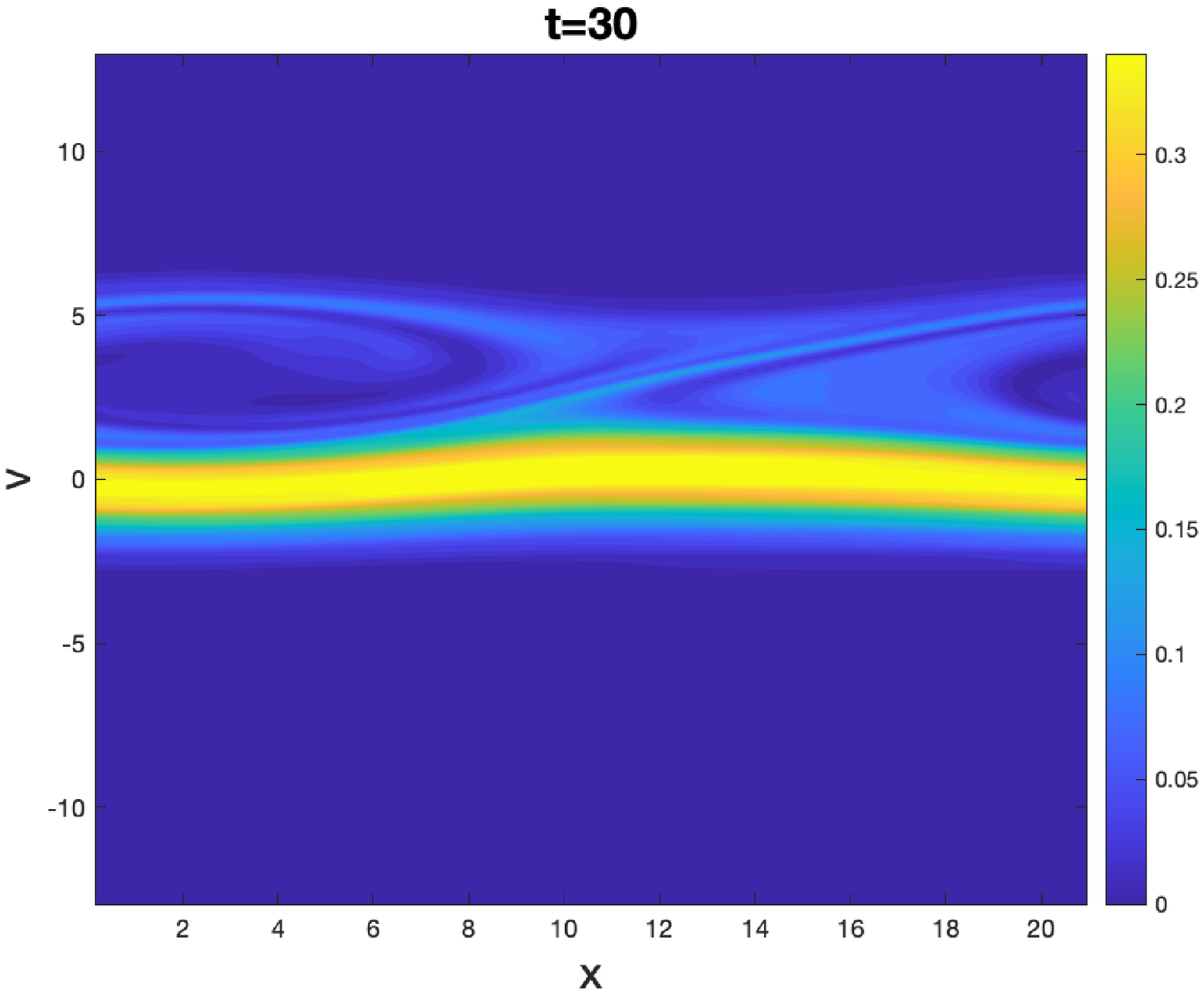}}
	\subfigure[Method II]{\includegraphics[height=40mm]{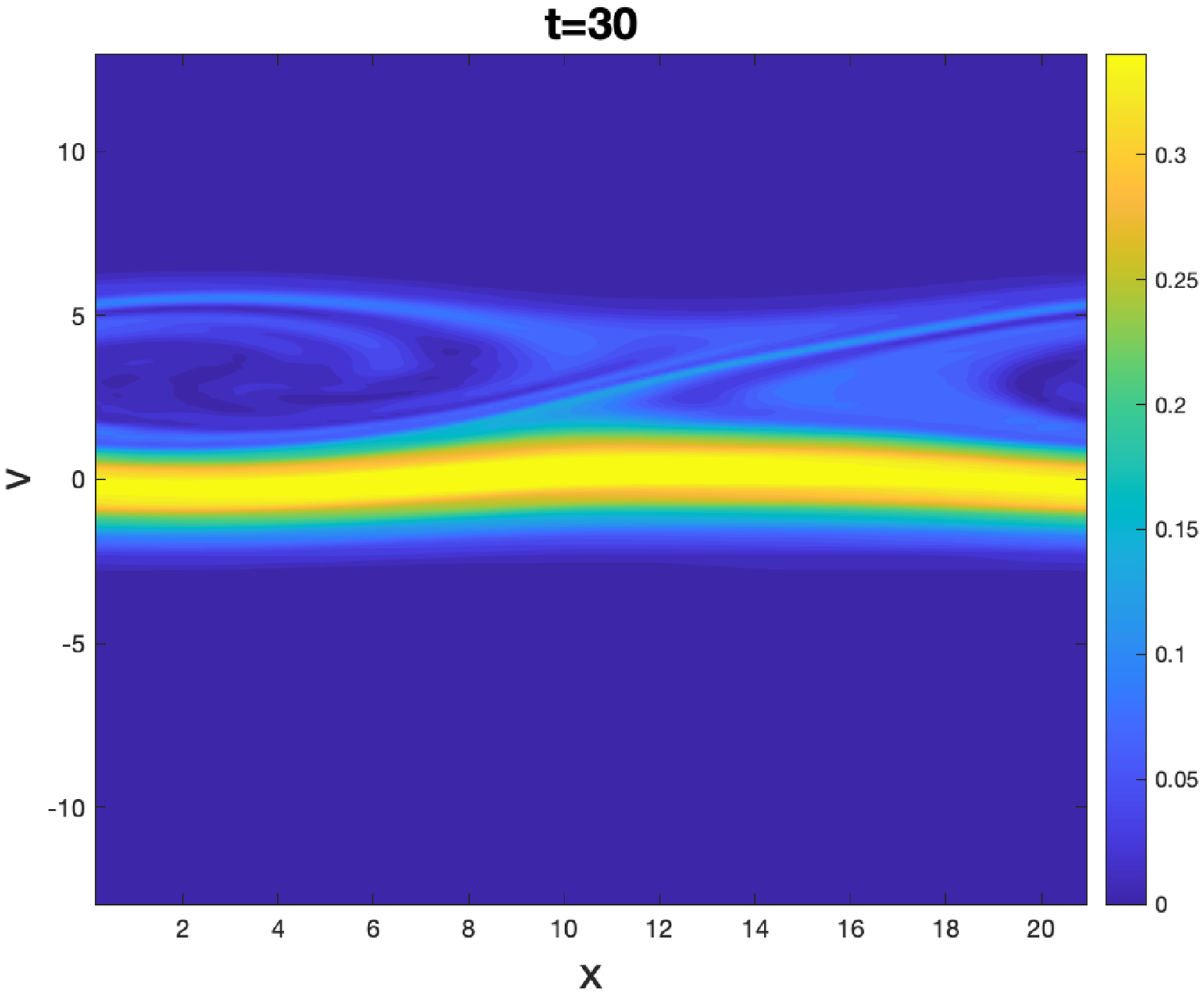}}
	\subfigure[Method III]{\includegraphics[height=40mm]{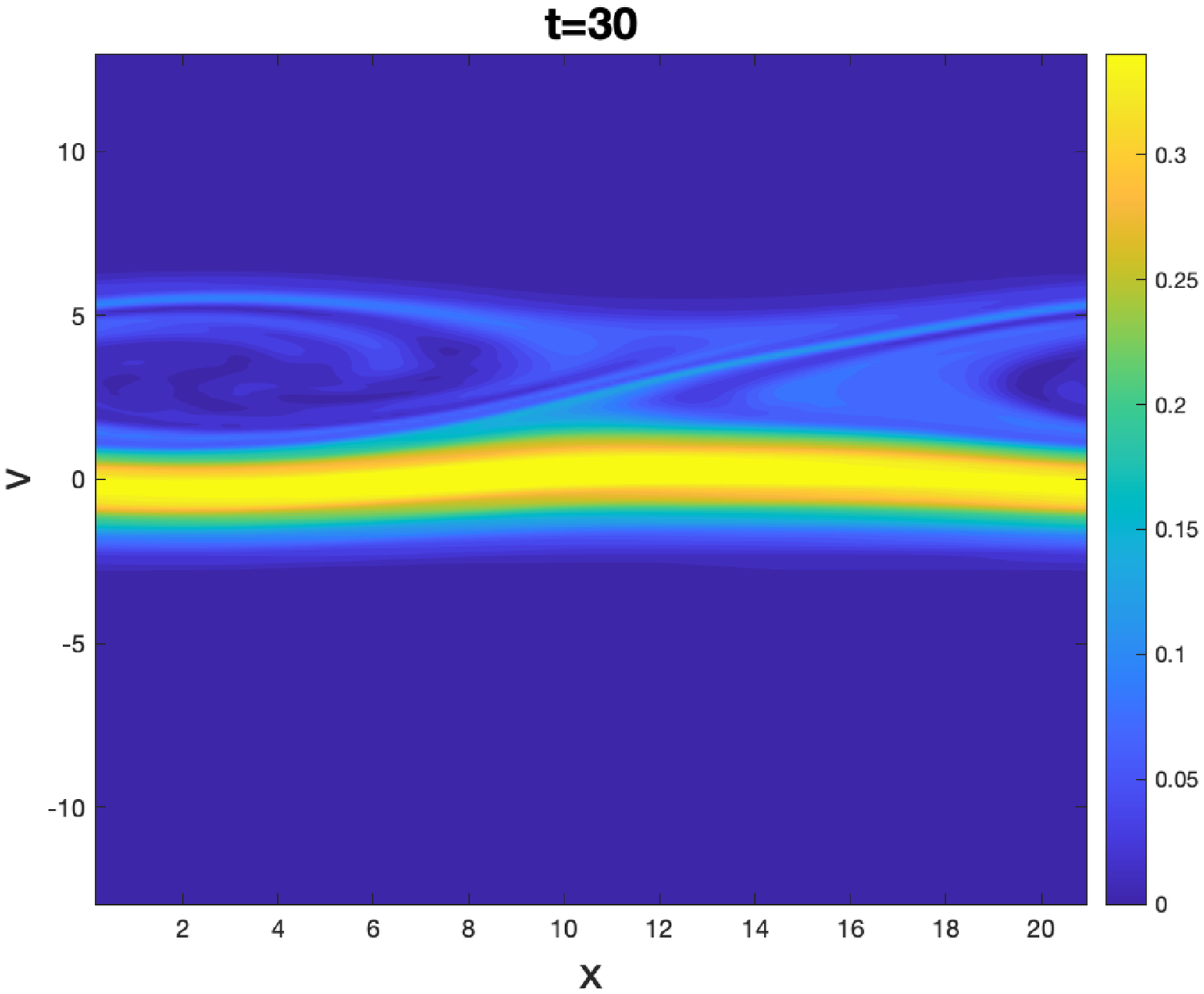}}
			\caption{Example \ref{ex:bumpontail}. Contour plots of the solutions at $t=30$ by three low rank methods including the non-conservative low rank method in \cite{guo2021lowrank} denoted by method I, the conservative method in \cite{guo2022lowrank} denoted by method II and the proposed LoMaC low rank method denoted by method III. $N_x\times N_v = 128\times256$. $\varepsilon=10^{-4}$.}
	\label{fig:bumpontail_contour}
\end{figure}

\end{exa}

\subsection{2D2V Vlasov-Poisson system}

\begin{exa} \label{ex:weak2d} (Weak Landau damping.) We simulate the 2D2V weak Landau damping. The initial  
condition is 
\begin{equation}
	\label{eq:weak}
	f(\bx,\bv,t=0) =\frac{1}{(2 \pi)^{d / 2}} \left(1+\alpha \sum_{m=1}^{d} \cos \left(k x_{m}\right)\right)\exp\left(-\frac{|\bv|^2}{2}\right),
\end{equation}
where $d=2$, $\alpha=0.01$, and $k=0.5$. We set the computation domain as $[0,L_x]^2\times[-L_v,L_v]^2$, where $L_x=\frac{2\pi}{k}$ and $L_v=6$, and the truncation threshold $\varepsilon=10^{-5}$. Note that the solutions are represented in the third order HT format, for which the dimension tree and data  are highlighted in Figure \ref{fig:dimtree1}. In Figure \ref{fig:weak2d_elec_con}, we report the time evolution of the electric energy, hierarchical ranks of the numerical solution,  relative deviation of total mass and energy together with absolute total momentum $J_1$ and $J_2$. It is observed that the proposed method to predict the damping rate of the electric energy as with the 1D1V case, and furthermore, the method is able to conserve the total mass and momentum $J_1$ and $J_2$ as well as the total energy up to the machine precision. We test the CPU time for with mesh refinement study. For a set of meshes $16^2\times32^2$, $32^2\times64^2$, $64^2\times128^2$ the CPU time is 377s, 670s, and 1177s, which are doubled with mesh refinement in each of direction. The CPU is only doubled, compared with $2^5$ times considering the 4D+time problem with mesh refinement in each direction. This implies  storage and CPU savings of several orders in magnitude.

\begin{figure}[h!]
	\centering
	\subfigure[]{\includegraphics[height=40mm]{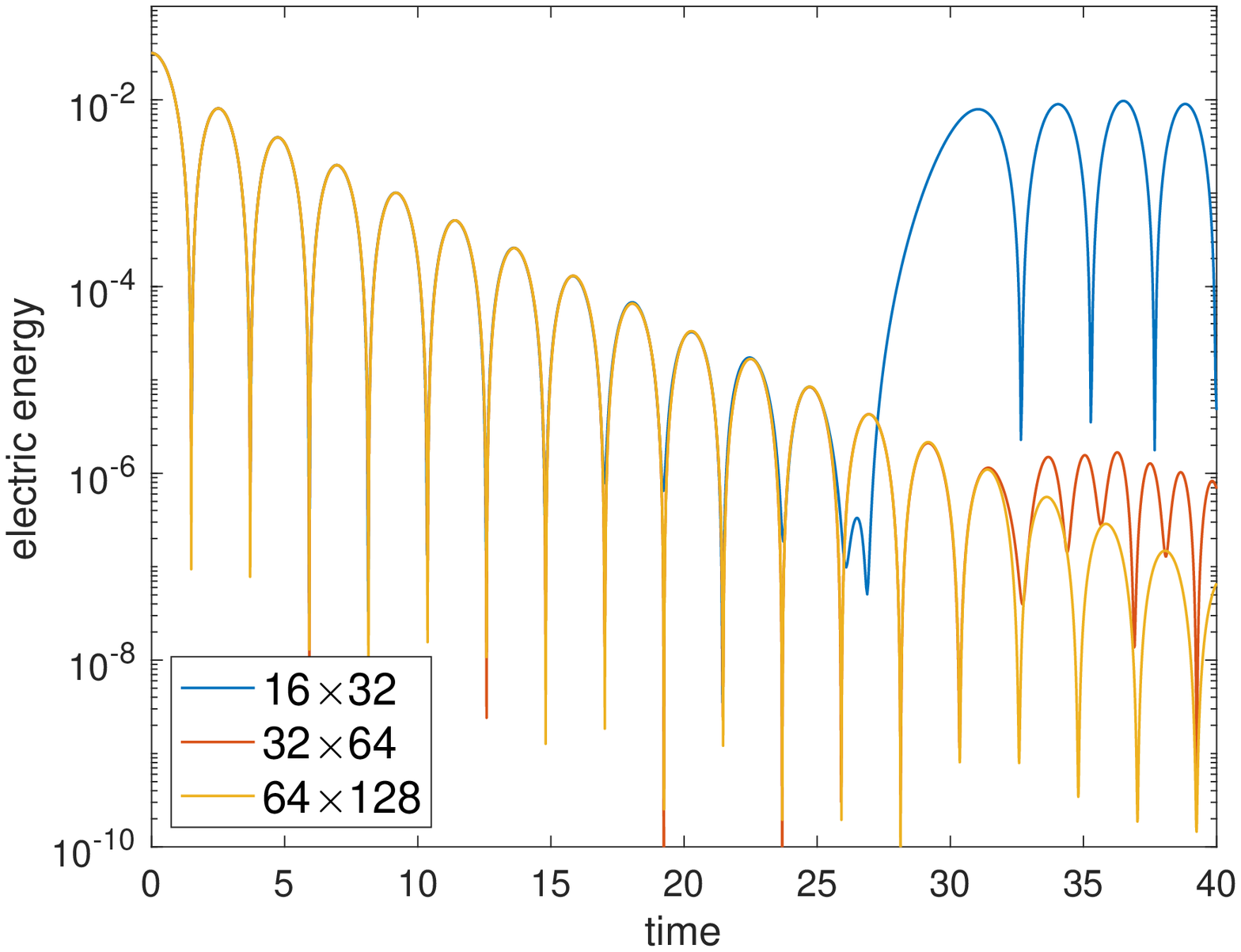}}
		\subfigure[]{\includegraphics[height=40mm]{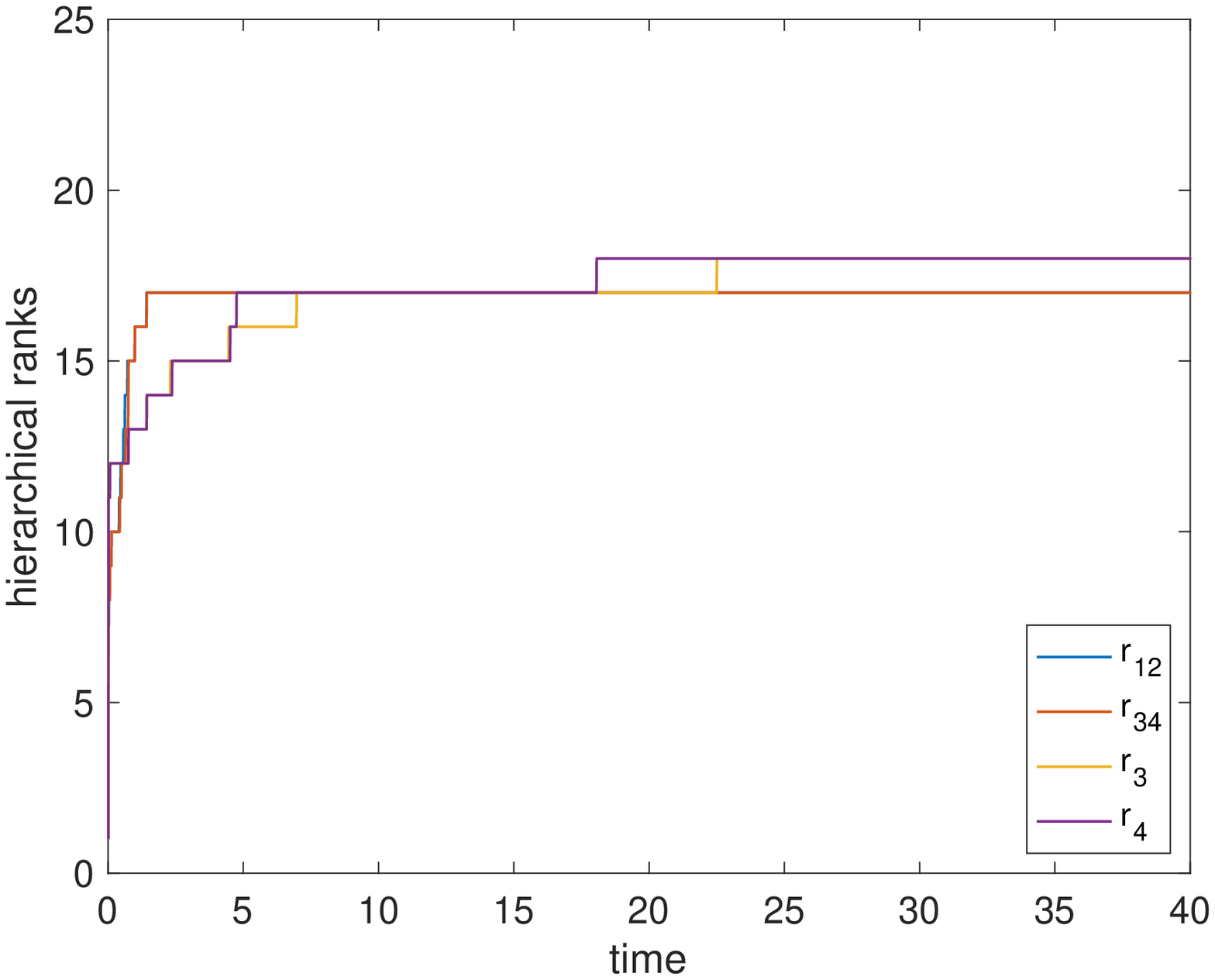}}
		\subfigure[]{\includegraphics[height=40mm]{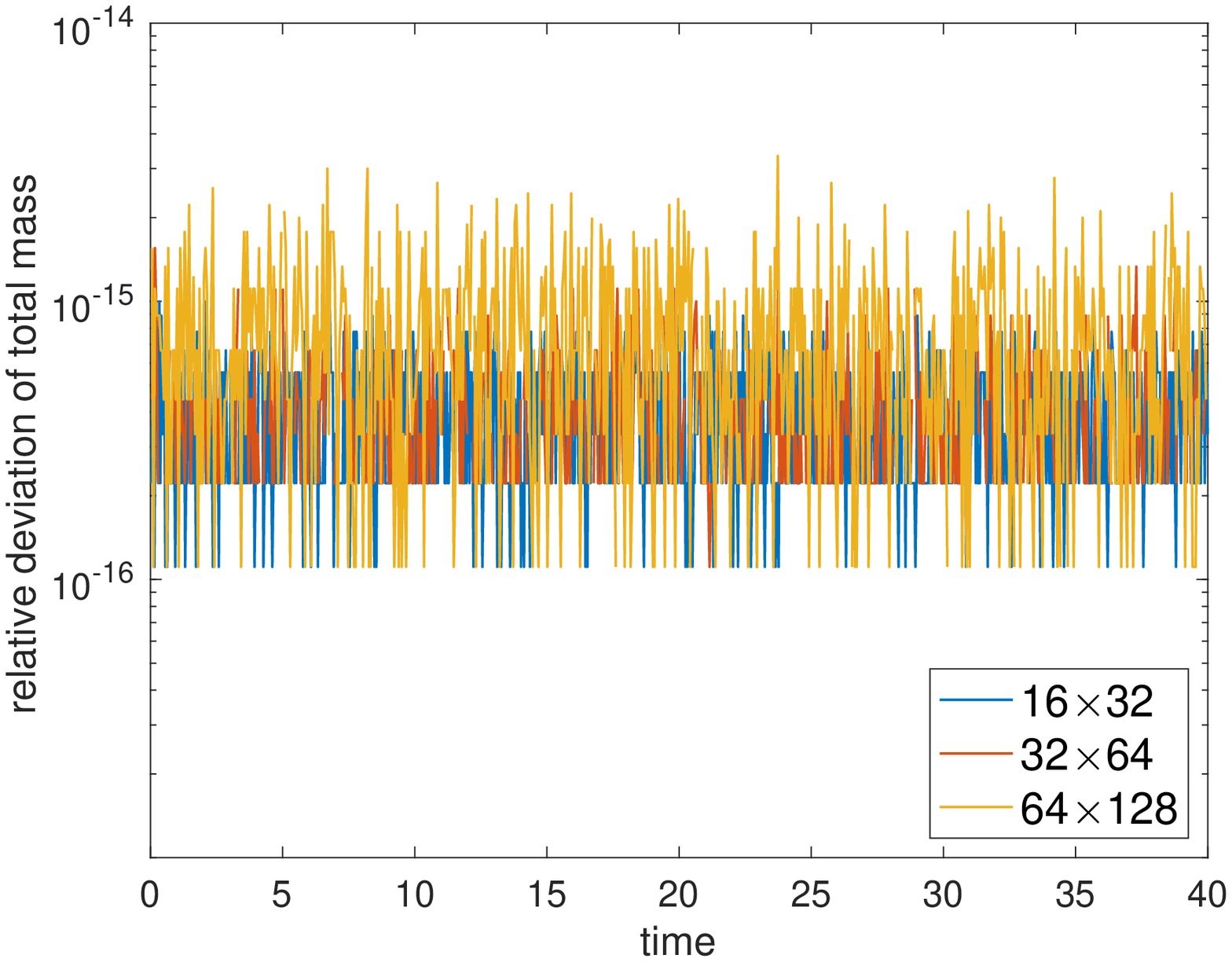}}
		\subfigure[]{\includegraphics[height=40mm]{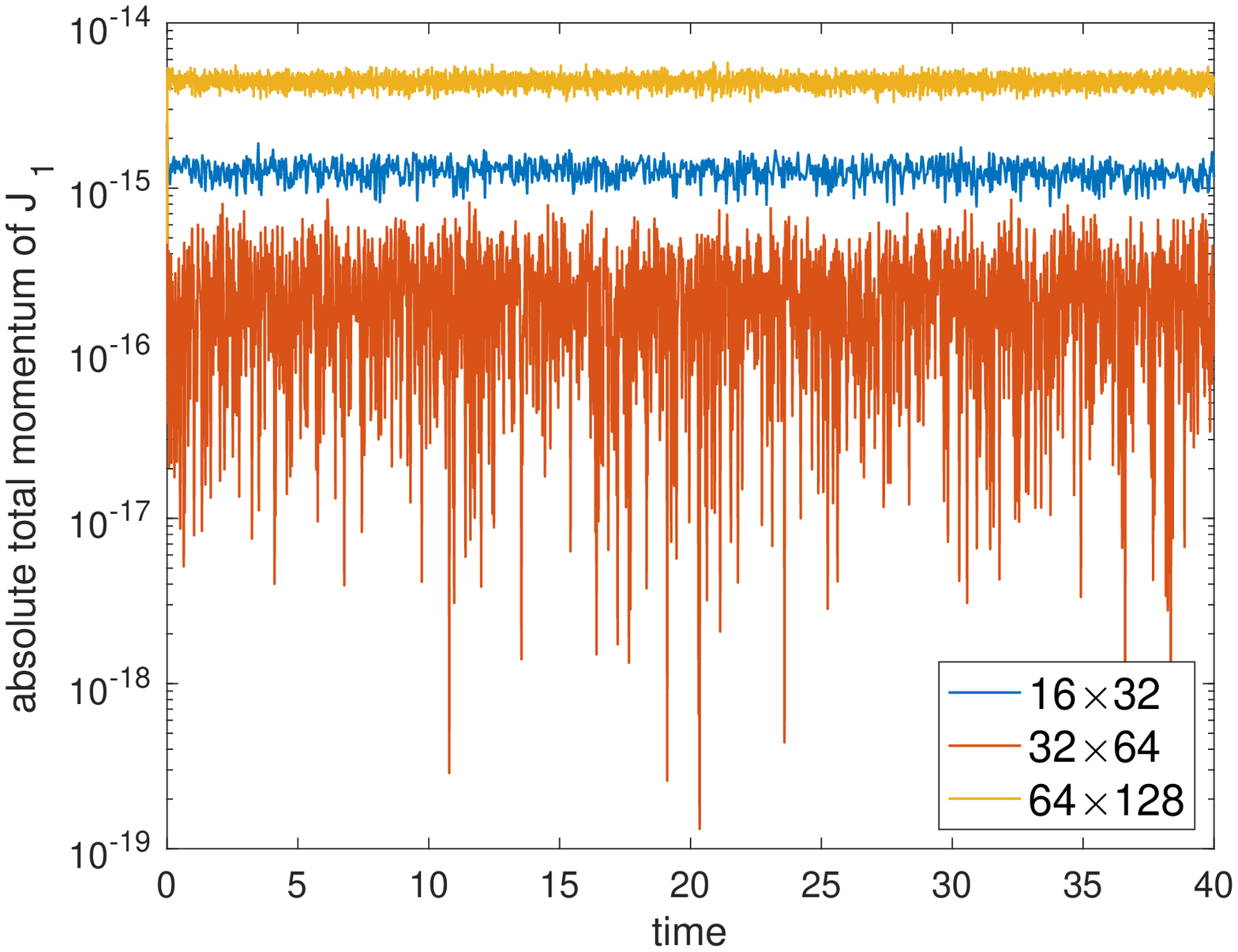}}
			\subfigure[]{\includegraphics[height=40mm]{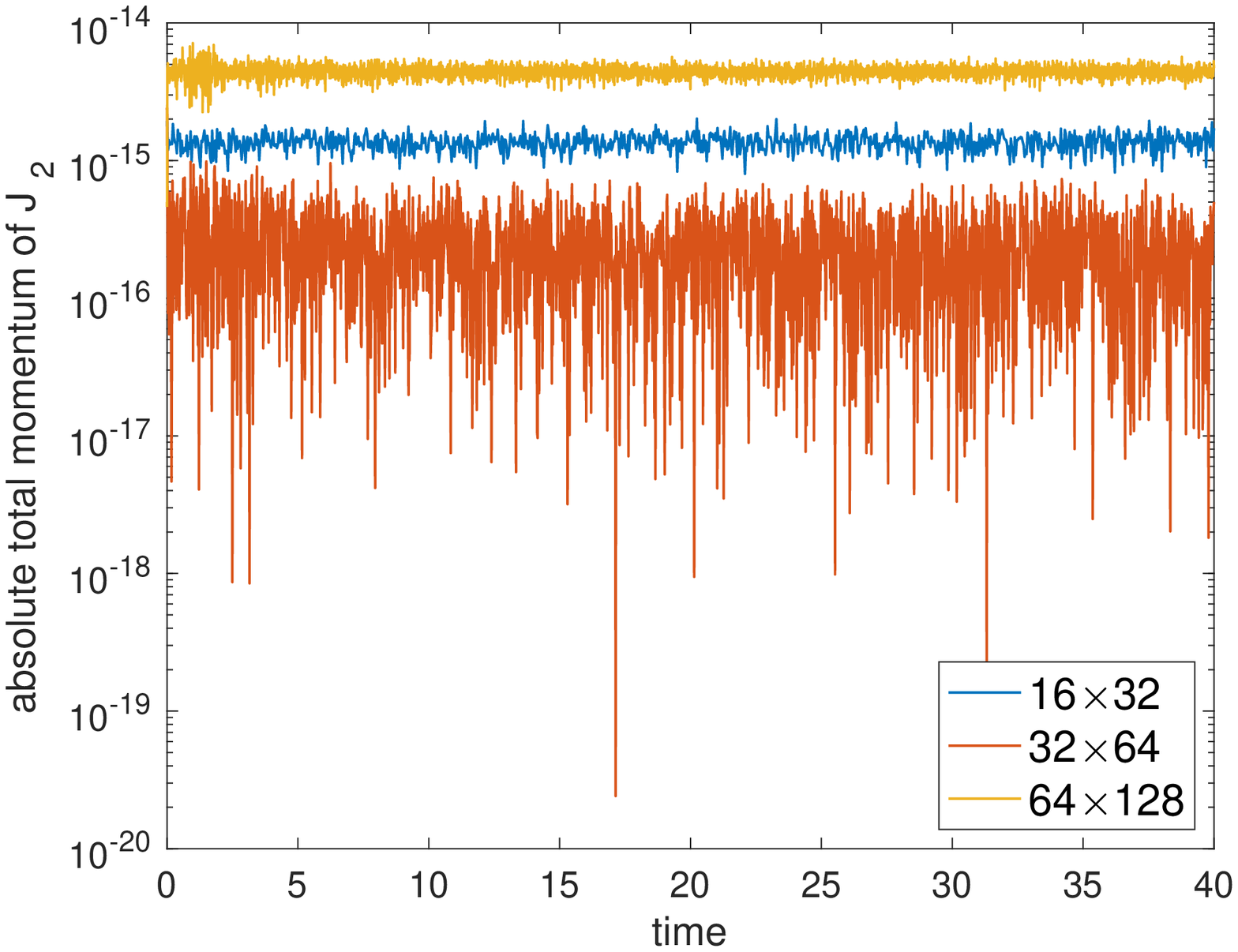}}
		\subfigure[]{\includegraphics[height=40mm]{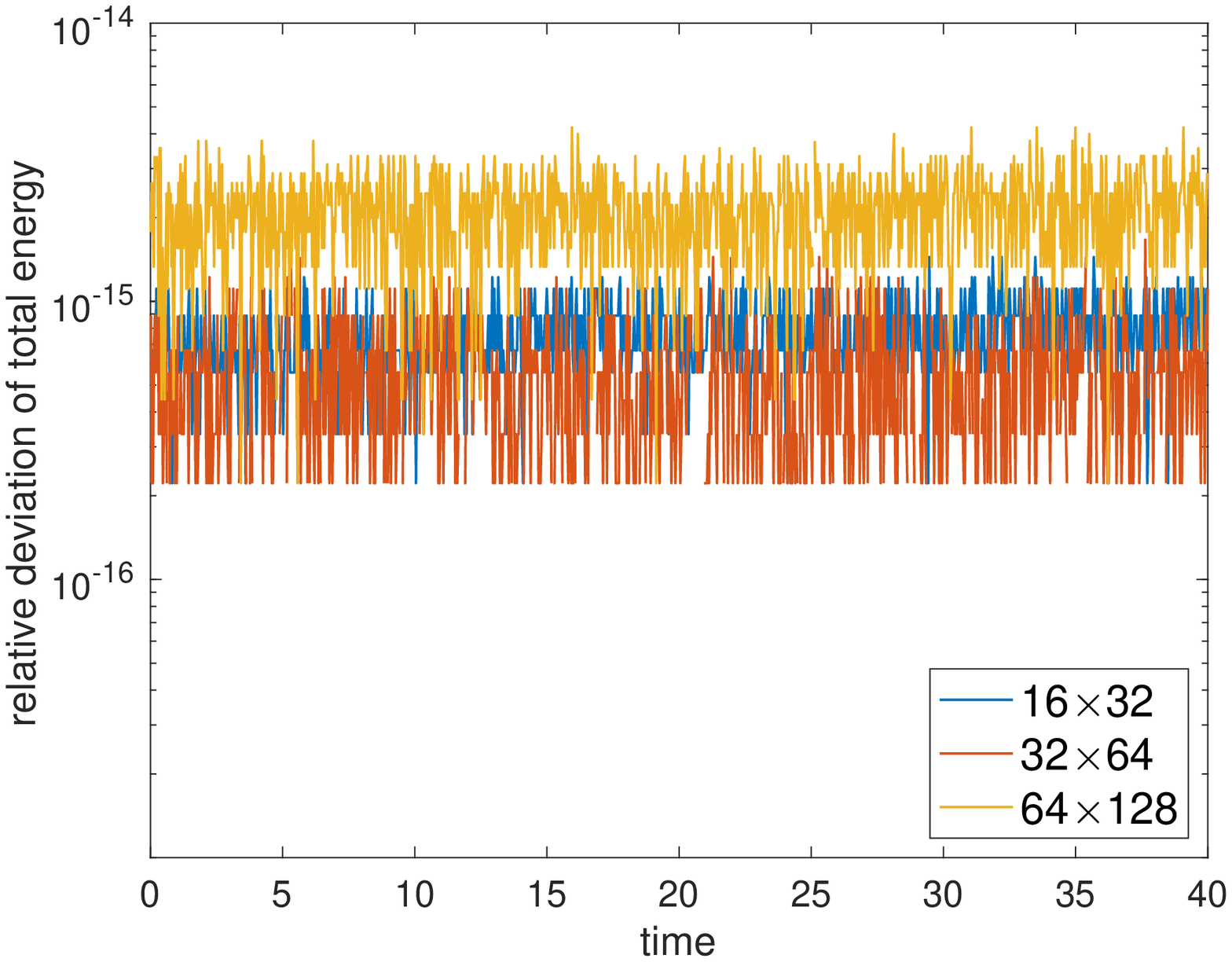}}
	\caption{Example \ref{ex:weak2d}. The time evolution of  electric energy (a), hierarchical ranks of the numerical solution of mesh size $N^2_x\times N^2_v=64^2\times128^2$ (b), relative deviation of total mass (c), absolute total momentum $J_1$ (d), absolute total momentum $J_2$ (e), and relative deviation of total energy (f). $\varepsilon=10^{-5}$. In (b), $r_{12}$ and $r_{34}$ are close. }
	\label{fig:weak2d_elec_con}
\end{figure}

 \end{exa}

 \begin{exa} \label{ex:two2d}(Two-stream instability.) We consider the 2D2V two-stream instability with initial condition
\begin{equation}
	\label{eq:two2d}
	f(\bx,\bv,t=0) =\frac{1}{2^d(2 \pi)^{d / 2}} \left(1+\alpha \sum_{m=1}^{d} \cos \left(k x_{m}\right)\right)\prod_{m=1}^d\left(\exp\left(-\frac{(v_m-v_0)^2}{2}\right) + \exp\left( -\frac{(v_m+v_0)^2}{2}\right)\right),
\end{equation}
where $d=2$, $\alpha=0.001$, $v_0=2.4$, and $k=0.2$.  The computation domain is set as $[0,L_x]^2\times[-L_v,L_v]^2$, where $L_x=\frac{2\pi}{k}$ and $L_v=8$. Let the truncation threshold be $\varepsilon=10^{-5}$.  In Figure \ref{fig:two2d_elec_con},  we report the time evolution of the electric energy, hierarchical ranks of the numerical solution of mesh size  $N_x^2\times N^2_v=128^2\times256^2$,  relative deviation of total mass and energy together with absolute total momentum $J_1$ and $J_2$. The observation is similar to the previous example that the proposed LoMaC low rank method is able to conserve the total mass, momentum, and energy up to the machine precision.

\begin{figure}[h!]
	\centering
	\subfigure[]{\includegraphics[height=40mm]{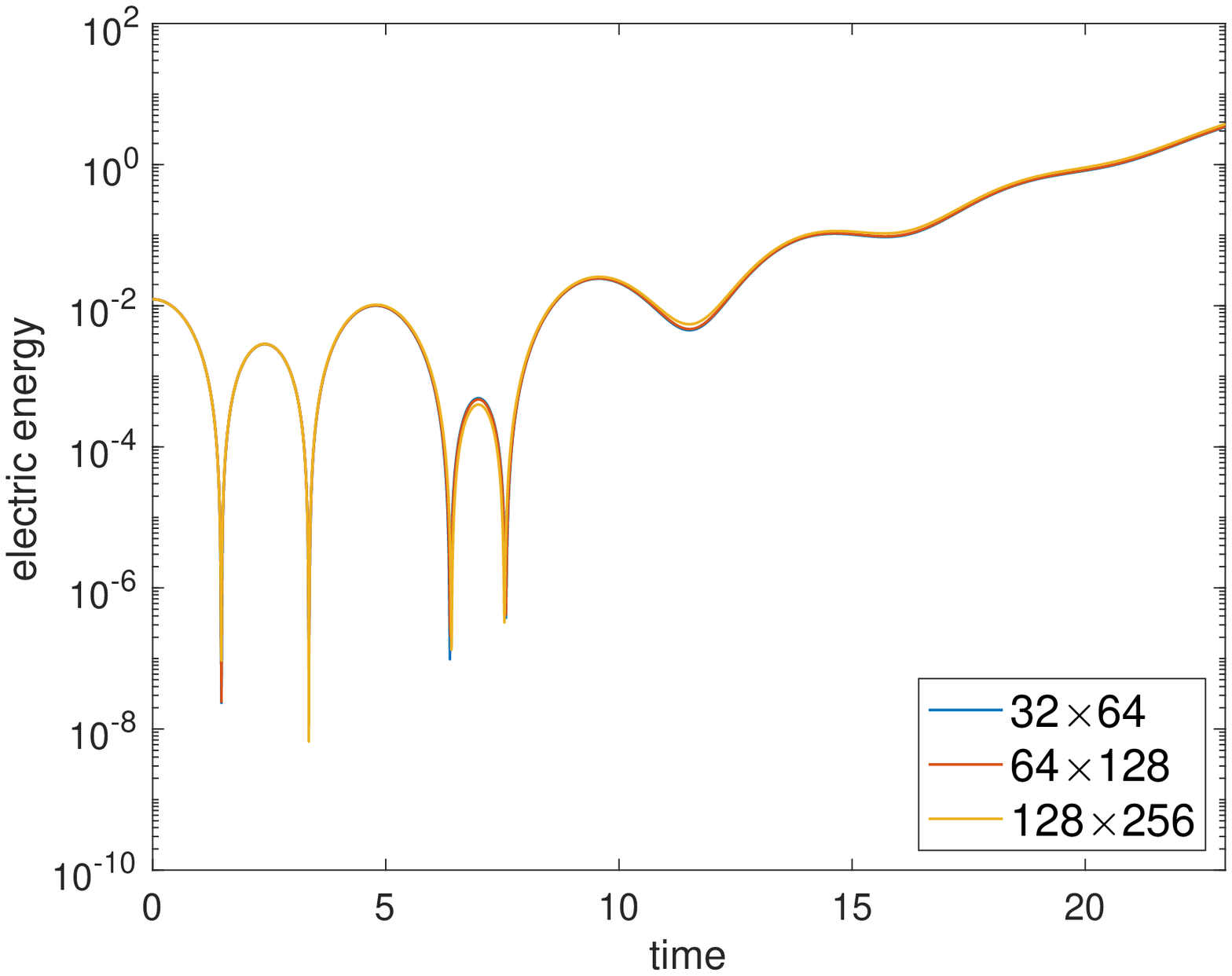}}
		\subfigure[]{\includegraphics[height=40mm]{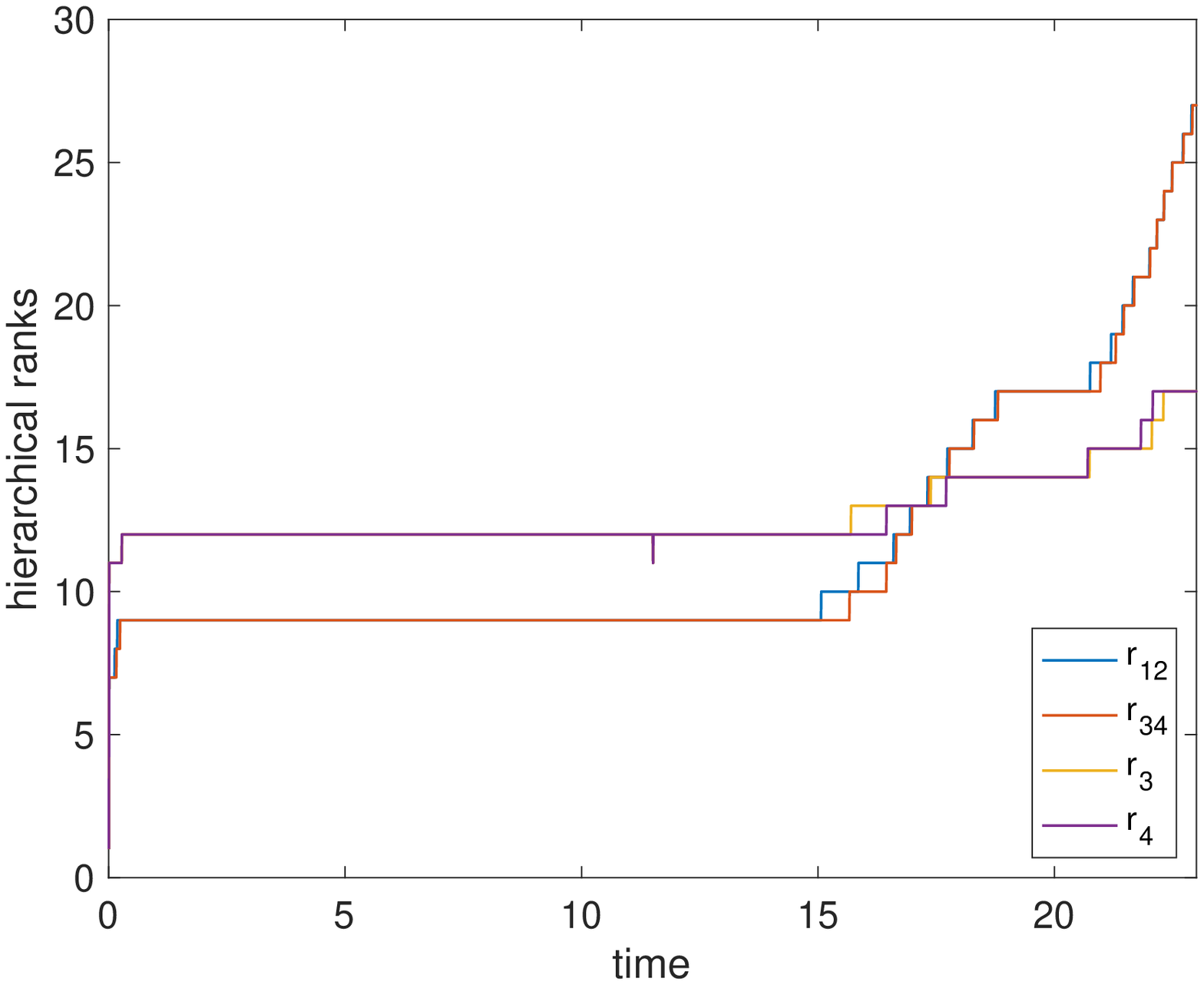}}
		\subfigure[]{\includegraphics[height=40mm]{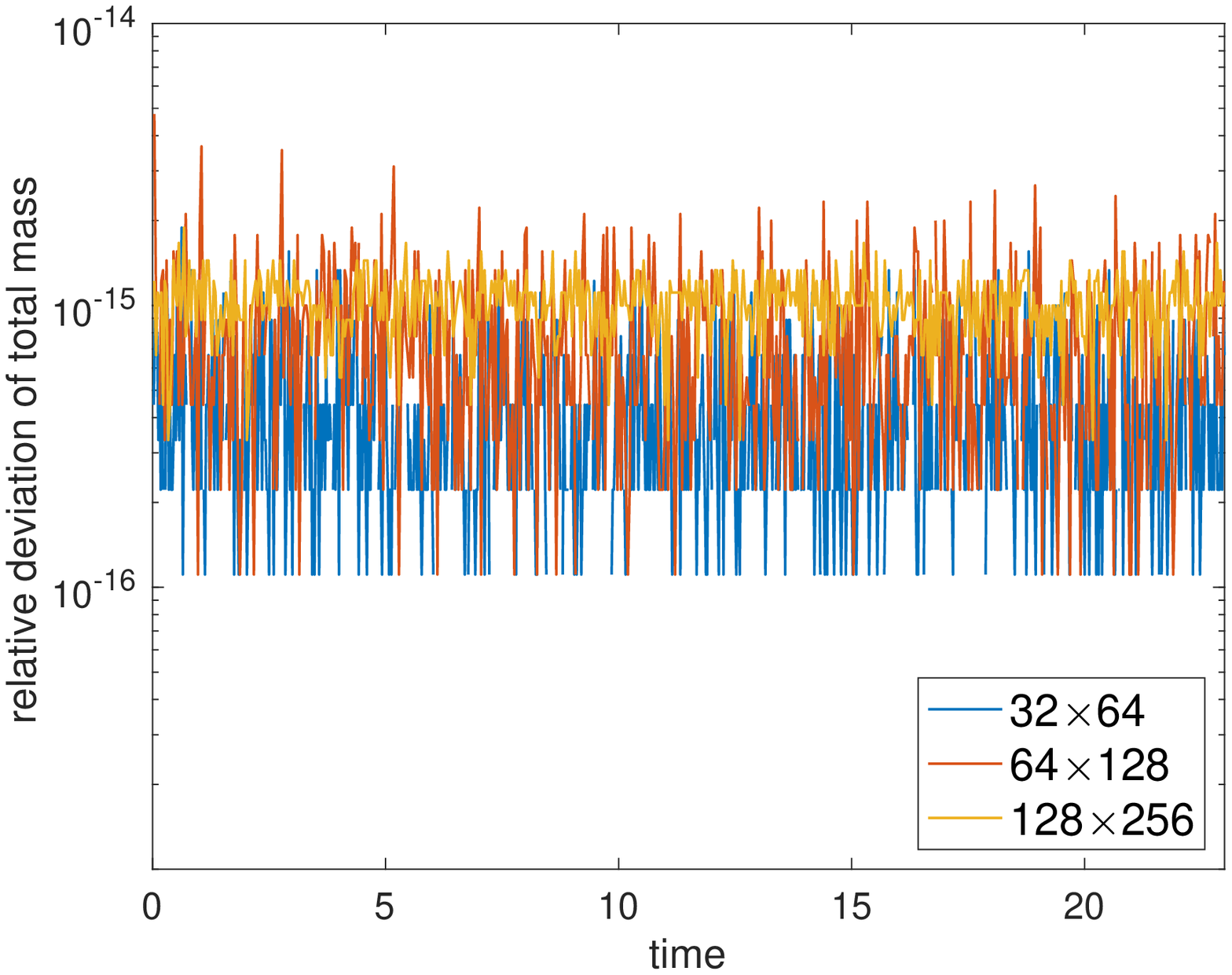}}
		\subfigure[]{\includegraphics[height=40mm]{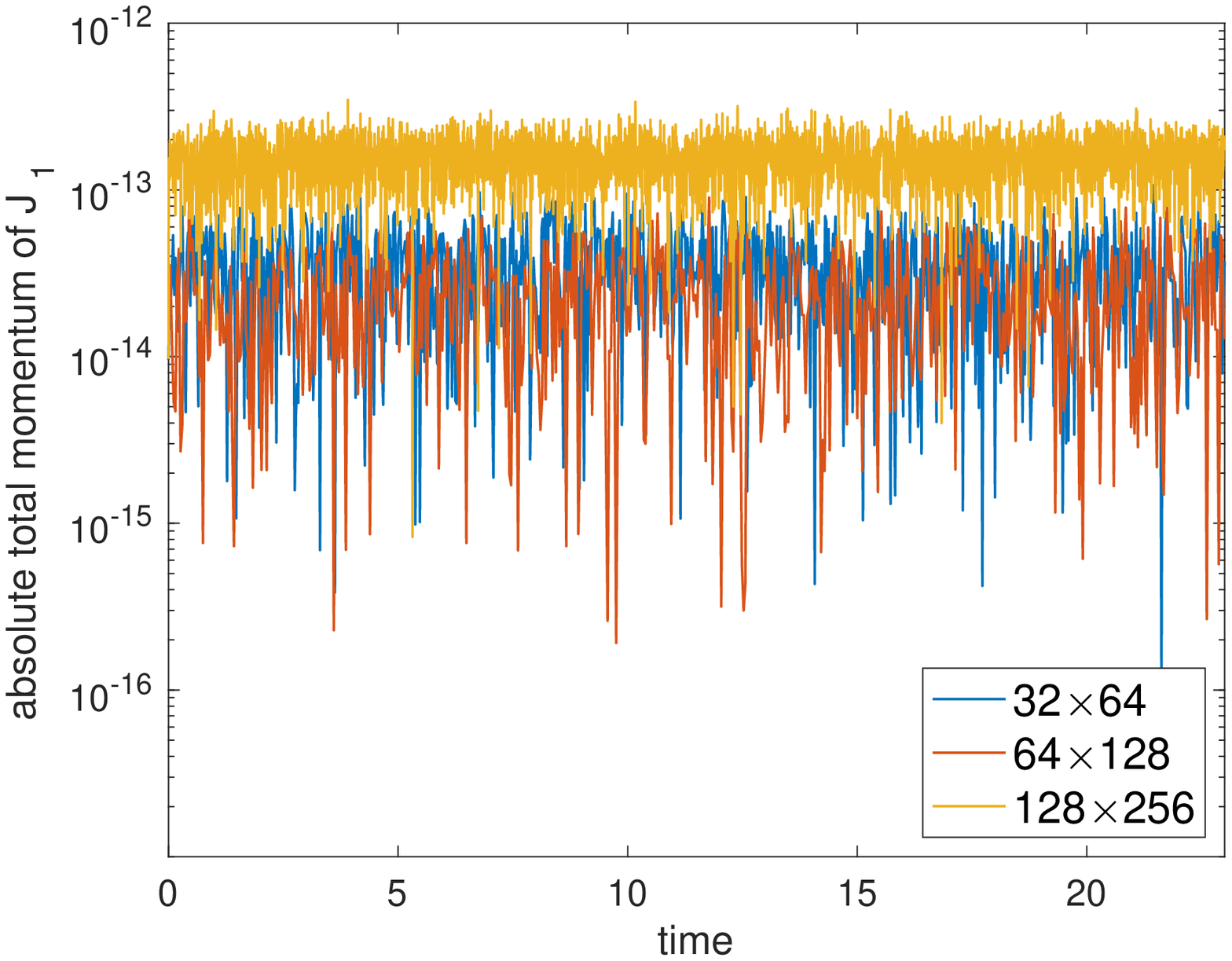}}
			\subfigure[]{\includegraphics[height=40mm]{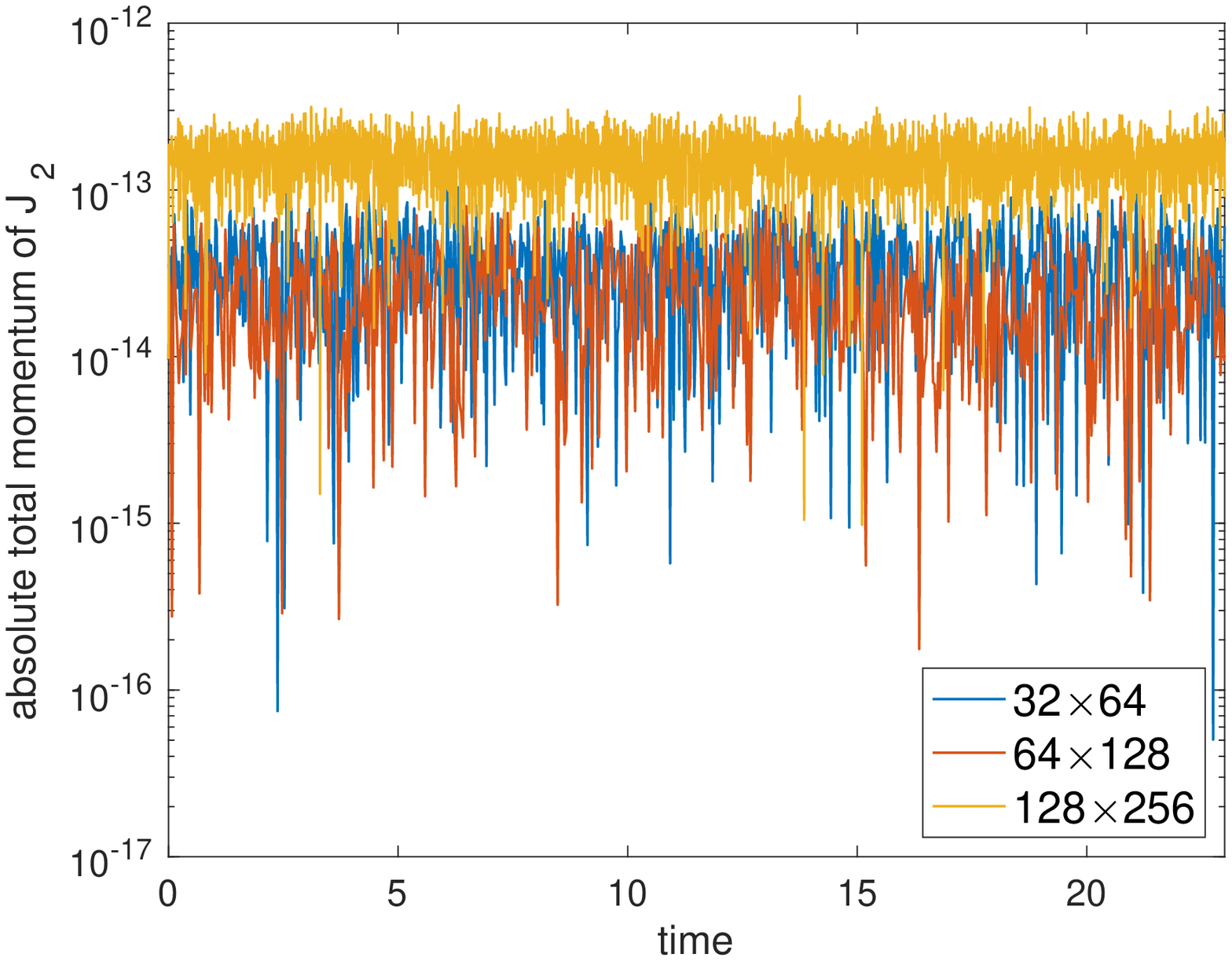}}
		\subfigure[]{\includegraphics[height=40mm]{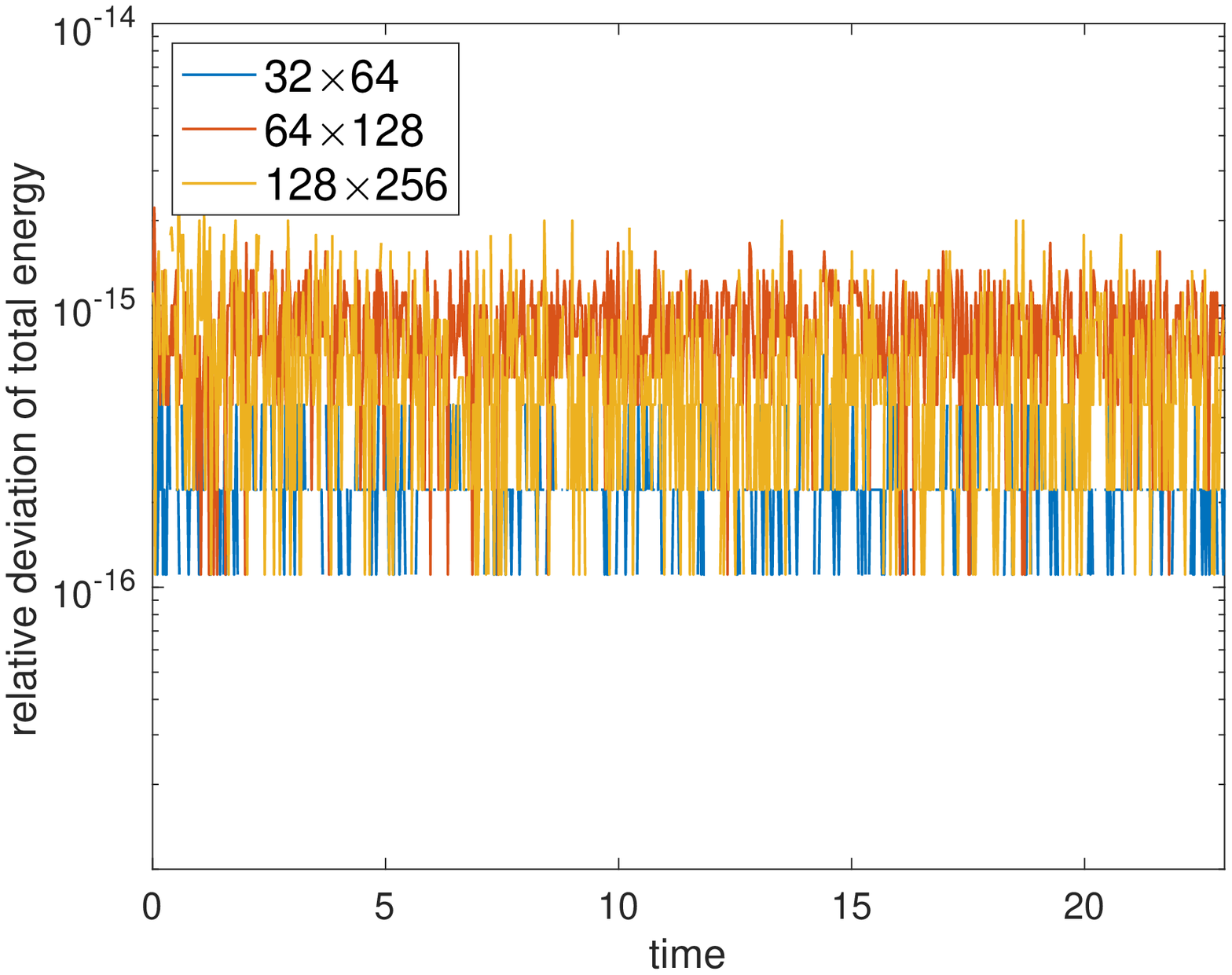}}
	\caption{Example \ref{ex:two2d}.  The time evolution of  electric energy (a), hierarchical ranks of the numerical solution of mesh size $N_x^2\times N^2_v=128^2\times256^2$ (b), relative deviation of total mass (c), absolute total momentum $J_1$ (d), absolute total momentum $J_2$ (e), and relative deviation of total energy (f). $\varepsilon=10^{-5}$. In (b), $r_{12}$ and $r_{34}$ are close, and $r_{3}$ and $r_{4}$.}
	\label{fig:two2d_elec_con}
	
\end{figure}

 \end{exa}

%% file: conclusion.tex
\section{Conclusion}
\setcounter{equation}{0}
\setcounter{figure}{0}
\setcounter{table}{0}

In this paper, we proposed a LoMaC low rank tensor approach for performing deterministic  Vlasov simulations in high dimensions. The newly developed algorithm simultaneously updates the macroscopic invariants in a local conservative fashion using kinetic flux vector splitting, 
alongside the evolution of the kinetic solution in a low rank fashion with adjustments on its macroscopic moments via an orthogonal projection to a subspace determined from updates of macroscopic moments. By construction, the method locally and globally conserves mass, momentum and energy at the fully discrete level. The algorithm is extended to the 2D2V VP system by a hierarchical Tucker structure with full rank (no reduction) in the physical space and low rank reduction for the phase space as well as for the linkage between phase and physical spaces. Further work includes the local marginal and global conservation of macroscopic observables with low rank structure in high dimensional physical spaces.

%% file: main_conser_energy.bbl
\begin{thebibliography}{10}

\bibitem{allmann2022parallel}
F.~Allmann-Rahn, R.~Grauer, and K.~Kormann.
\newblock A parallel low-rank solver for the six-dimensional vlasov-maxwell
  equations.
\newblock {\em arXiv preprint arXiv:2201.03471}, 2022.

\bibitem{birdsall2004plasma}
C.~K. Birdsall and A.~B. Langdon.
\newblock {\em Plasma physics via computer simulation}.
\newblock CRC press, 2004.

\bibitem{carroll1970analysis}
J.~D. Carroll and J.-J. Chang.
\newblock {Analysis of individual differences in multidimensional scaling via
  an N-way generalization of ``Eckart-Young" decomposition}.
\newblock {\em Psychometrika}, 35(3):283--319, 1970.

\bibitem{cheng2014energy}
Y.~Cheng, A.~J. Christlieb, and X.~Zhong.
\newblock Energy-conserving discontinuous galerkin methods for the
  vlasov--ampere system.
\newblock {\em Journal of Computational Physics}, 256:630--655, 2014.

\bibitem{dawson1983particle}
J.~Dawson.
\newblock Particle simulation of plasmas.
\newblock {\em Rev. Mod. Phys.}, 55(2):403, 1983.

\bibitem{de2012high}
B.~A. de~Dios and S.~Hajian.
\newblock {High order and energy preserving discontinuous Galerkin methods for
  the Vlasov-Poisson system}.
\newblock {\em arXiv preprint arXiv:1209.4025}, 2012.

\bibitem{de2000multilinear}
L.~De~Lathauwer, B.~De~Moor, and J.~Vandewalle.
\newblock A multilinear singular value decomposition.
\newblock {\em SIAM J. Matrix Anal. Appl.}, 21(4):1253--1278, 2000.

\bibitem{dektor2020dynamically}
A.~Dektor and D.~Venturi.
\newblock Dynamically orthogonal tensor methods for high-dimensional nonlinear
  pdes.
\newblock {\em Journal of Computational Physics}, 404:109125, 2020.

\bibitem{ehrlacher2017dynamical}
V.~Ehrlacher and D.~Lombardi.
\newblock {A dynamical adaptive tensor method for the Vlasov--Poisson system}.
\newblock {\em J. Comput. Phys.}, 339:285--306, 2017.

\bibitem{einkemmer2021mass}
L.~Einkemmer and I.~Joseph.
\newblock A mass, momentum, and energy conservative dynamical low-rank scheme
  for the vlasov equation.
\newblock {\em Journal of Computational Physics}, page 110495, 2021.

\bibitem{einkemmer2018low}
L.~Einkemmer and C.~Lubich.
\newblock {A low-rank projector-splitting integrator for the Vlasov--Poisson
  equation}.
\newblock {\em SIAM J. Sci. Comput.}, 40(5):B1330--B1360, 2018.

\bibitem{einkemmer2019quasi}
L.~Einkemmer and C.~Lubich.
\newblock A quasi-conservative dynamical low-rank algorithm for the vlasov
  equation.
\newblock {\em SIAM Journal on Scientific Computing}, 41(5):B1061--B1081, 2019.

\bibitem{einkemmer2020low}
L.~Einkemmer, A.~Ostermann, and C.~Piazzola.
\newblock A low-rank projector-splitting integrator for the vlasov--maxwell
  equations with divergence correction.
\newblock {\em J. Comput. Phys.}, 403:109063, 2020.

\bibitem{filbet2003comparison}
F.~Filbet and E.~Sonnendrucker.
\newblock {Comparison of eulerian Vlasov solvers}.
\newblock {\em Computer Physics Communications}, 150(3):247--266, 2003.

\bibitem{gottlieb2011strong}
S.~Gottlieb, D.~I. Ketcheson, and C.-W. Shu.
\newblock {\em {Strong stability preserving Runge-Kutta and multistep time
  discretizations}}.
\newblock World Scientific, 2011.

\bibitem{grasedyck2010hierarchical}
L.~Grasedyck.
\newblock Hierarchical singular value decomposition of tensors.
\newblock {\em SIAM J. Matrix Anal. Appl.}, 31(4):2029--2054, 2010.

\bibitem{griebel1990parallelizable}
M.~Griebel.
\newblock A parallelizable and vectorizable multi-level algorithm on sparse
  grids.
\newblock In W.~Hackbusch, editor, {\em Parallel algorithms for partial
  differential equations}, volume~31 of {\em Notes on numerical fluid
  mechanics}, pages 94--100. 1991.

\bibitem{guo2016sparse1}
W.~Guo and Y.~Cheng.
\newblock {A sparse grid discontinuous Galerkin method for high-dimensional
  transport equations and its application to kinetic simulations}.
\newblock {\em SIAM J. Sci. Comput.}, 38(6):A3381--A3409, 2016.

\bibitem{guo2021lowrank}
W.~Guo and J.-M. Qiu.
\newblock A low rank tensor representation of linear transport and nonlinear
  vlasov solutions and their associated flow maps.
\newblock {\em arXiv preprint arXiv:2106.08834}, 2021.

\bibitem{guo2022conservative}
W.~Guo and J.-M. Qiu.
\newblock A conservative low rank tensor method for the vlasov dynamics.
\newblock {\em arXiv preprint arXiv:2201.10397}, 2022.

\bibitem{guo2022lowrank}
W.~Guo and J.-M. Qiu.
\newblock A conservative low-rank tensor method for the vlasov dynamics.
\newblock {\em arXiv preprint arXiv:2106.08834}, 2022.

\bibitem{guo2022low}
W.~Guo and J.-M. Qiu.
\newblock A low rank tensor representation of linear transport and nonlinear
  vlasov solutions and their associated flow maps.
\newblock {\em Journal of Computational Physics}, 458:111089, 2022.

\bibitem{hackbusch2009new}
W.~Hackbusch and S.~K{\"u}hn.
\newblock A new scheme for the tensor representation.
\newblock {\em J. Fourier Anal. Appl.}, 15(5):706--722, 2009.

\bibitem{harshman1970foundations}
R.~A. Harshman et~al.
\newblock {Foundations of the PARAFAC procedure: Models and conditions for an
  ``explanatory" multimodal factor analysis}.
\newblock {\em UCLA Working Papers in Phonetics}, pages 1--84, 1970.

\bibitem{hitchcock1927expression}
F.~L. Hitchcock.
\newblock The expression of a tensor or a polyadic as a sum of products.
\newblock {\em J. Math. Phys.}, 6(1-4):164--189, 1927.

\bibitem{kolda2009tensor}
T.~G. Kolda and B.~W. Bader.
\newblock Tensor decompositions and applications.
\newblock {\em SIAM Rev.}, 51(3):455--500, 2009.

\bibitem{kormann2015semi}
K.~Kormann.
\newblock {A semi-Lagrangian Vlasov solver in tensor train format}.
\newblock {\em SIAM J. Sci. Comput.}, 37(4):B613--B632, 2015.

\bibitem{kormann2016sparse}
K.~Kormann and E.~Sonnendr{\"u}cker.
\newblock {Sparse grids for the Vlasov--Poisson equation}.
\newblock In {\em Sparse Grids and Applications-Stuttgart 2014}, pages
  163--190. Springer, 2016.

\bibitem{mandal1994kinetic}
J.~Mandal and S.~Deshpande.
\newblock Kinetic flux vector splitting for euler equations.
\newblock {\em Computers \& fluids}, 23(2):447--478, 1994.

\bibitem{oseledets2011tensor}
I.~Oseledets.
\newblock Tensor-train decomposition.
\newblock {\em SIAM J. Sci. Comput.}, 33(5):2295--2317, 2011.

\bibitem{oseledets2012solution}
I.~V. Oseledets and S.~V. Dolgov.
\newblock {Solution of linear systems and matrix inversion in the TT-format}.
\newblock {\em SIAM J. Sci. Comput.}, 34(5):A2718--A2739, 2012.

\bibitem{oseledets2009breaking}
I.~V. Oseledets and E.~E. Tyrtyshnikov.
\newblock {Breaking the curse of dimensionality, or how to use SVD in many
  dimensions}.
\newblock {\em SIAM J. Sci. Comput.}, 31(5):3744--3759, 2009.

\bibitem{peng2020low}
Z.~Peng, R.~G. McClarren, and M.~Frank.
\newblock A low-rank method for two-dimensional time-dependent radiation
  transport calculations.
\newblock {\em Journal of Computational Physics}, 421:109735, 2020.

\bibitem{rieke2015coupled}
M.~Rieke, T.~Trost, and R.~Grauer.
\newblock {Coupled Vlasov and two-fluid codes on GPUs}.
\newblock {\em Journal of Computational Physics}, 283:436--452, 2015.

\bibitem{shu2009high}
C.-W. Shu.
\newblock High order weighted essentially nonoscillatory schemes for convection
  dominated problems.
\newblock {\em SIAM review}, 51(1):82--126, 2009.

\bibitem{smolyak1963quadrature}
S.~Smolyak.
\newblock Quadrature and interpolation formulas for tensor products of certain
  classes of functions.
\newblock In {\em Dokl. Akad. Nauk SSSR}, volume~4, pages 240--243, 1963.

\bibitem{tao2018sparseguo}
Z.~Tao, W.~Guo, and Y.~Cheng.
\newblock {Sparse grid discontinuous Galerkin methods for the Vlasov-Maxwell
  system}.
\newblock {\em J. Comput. Phys: X}, 3:100022, 2019.

\bibitem{trost2017enhanced}
T.~Trost, S.~Lautenbach, and R.~Grauer.
\newblock {Enhanced conservation properties of Vlasov codes through coupling
  with conservative fluid models}.
\newblock {\em arXiv preprint arXiv:1702.00367}, 2017.

\bibitem{tucker1966some}
L.~R. Tucker.
\newblock Some mathematical notes on three-mode factor analysis.
\newblock {\em Psychometrika}, 31(3):279--311, 1966.

\bibitem{xu1995gas}
K.~Xu, L.~Martinelli, and A.~Jameson.
\newblock Gas-kinetic finite volume methods, flux-vector splitting, and
  artificial diffusion.
\newblock {\em Journal of computational physics}, 120(1):48--65, 1995.

\bibitem{zenger1991sparse}
C.~Zenger.
\newblock Sparse grids.
\newblock In {\em Parallel Algorithms for Partial Differential Equations,
  Proceedings of the Sixth GAMM-Seminar}, volume~31, 1990.

\end{thebibliography}
